\documentclass[11pt,onecolumn,final]{elsarticle}

\makeatletter
\def\ps@pprintTitle{%
  \let\@oddhead\@empty
  \let\@evenhead\@empty
  \let\@oddfoot\@empty
  \let\@evenfoot\@oddfoot
}
\makeatother

\usepackage{amsmath}       %   AMS equation environments
\usepackage{amssymb}       %   AMS symbol fonts (e.g., \boldsymbol{}
\usepackage{amsthm}

\usepackage{chapterbib}    %   Bibtex
\usepackage{color}
\usepackage{comment}
\usepackage{bm,bbm}

\usepackage{booktabs,siunitx}

\usepackage{empheq}
\usepackage{tikz}
\usepackage{subfig}
\usepackage{graphicx}
\usepackage{adjustbox}
\usepackage{mathrsfs}
\usepackage{relsize}
\usepackage{enumitem}
\usepackage{mathtools} % provides \DeclarePairedDelimiterX
\usepackage{etoolbox} % provides \ifblank
\usepackage{xspace}

\def\tsc#1{\csdef{#1}{\textsc{\lowercase{#1}}\xspace}}
\tsc{WGM}
\tsc{QE}
\tsc{EP}
\tsc{PMS}
\tsc{BEC}
\tsc{DE}

\numberwithin{equation}{section} % Number equations as (section.equation)

\theoremstyle{plain}

\newtheorem{theorem}{Theorem}[section]
\newtheorem{lemma}[theorem]{Lemma}

\newtheorem{proposition}[theorem]{Proposition}

\theoremstyle{definition}

\newtheorem{remark}[theorem]{Remark}

\usepackage{mathtools}
\usepackage{etoolbox}

\DeclarePairedDelimiterX{\normiii}[1]
                        {|||}
                        {|||}
                        {\ifblank{#1}{\cdot}{#1}}

\definecolor{MyDarkGreen}{rgb}{0,0.45,0}

\def\trait #1 #2 #3 {\vrule width #1pt height #2pt depth #3pt}
\def\fin{\hfill
        \trait .3 5 0
        \trait 5 .3 0
        \kern-5pt
        \trait 5 5 -4.7
        \trait 0.3 5 0
\medskip}

\newcommand{\DOFS}[1]{\mbox{\textbf{#1}}}

\newcommand{\INTP}{\footnotesize{I}}

\newcommand{\REAL}{\mathbbm{R}}
\newcommand{\INTG}{\mathbbm{N}}

\newcommand{\restrict}[2]{{#1}_{|{#2}}}

\newcommand{\EOD}{\end{document}}

\newcommand{\fv}{\mathbf{f}}

\newcommand{\mv}{\mathbf{m}}

\newcommand{\uv}{\mathbf{u}}
\newcommand{\vv}{\mathbf{v}}
\newcommand{\wv}{\mathbf{w}}
\newcommand{\xv}{\mathbf{x}}

\newcommand{\Ev}{\mathbf{E}}

\newcommand{\Hv}{\mathbf{H}}

\newcommand{\Vv}{\mathbf{V}}

\newcommand{\as}{a}
\newcommand{\bs}{b}
\newcommand{\cs}{c}

\newcommand{\gs}{g}

\newcommand{\ks}{k}

\newcommand{\ps}{p}
\newcommand{\qs}{q}
\newcommand{\rs}{r}
\renewcommand{\ss}{s}
\newcommand{\ts}{t}

\newcommand{\vs}{v}

\newcommand{\Cs}{C}

\newcommand{\Es}{E}
\newcommand{\Fs}{F}

\newcommand{\Qs}{Q}

\newcommand{\Ss}{S}

\newcommand{\Vs}{V}

\newcommand{\Ys}{Y}

\newcommand{\East}{\Es^\ast}
\newcommand{\gsh}{\gs_{\hh}}
\newcommand{\FsP}{\Fs^{\P}}
\newcommand{\SP}{\Ss^{\P}}

\newcommand{\ass}[1]{a_{#1}}
\newcommand{\bss}[1]{b_{#1}}
\newcommand{\css}[1]{c_{#1}}

\newcommand{\mss}[1]{m_{#1}}

\newcommand{\xss}[1]{x_{#1}}

\newcommand{\Css}[1]{C_{#1}}

\newcommand{\Fss}[1]{F_{#1}}

\newcommand{\Cclem}{\Css{\text{clem}}}
\newcommand{\CclemHat}{\widehat{\Cs}_{\text{clem}}}

\newcommand{\calA}{\mathcal{A}}
\newcommand{\calB}{\mathcal{B}}

\newcommand{\calL}{\mathcal{L}}
\newcommand{\calM}{\mathcal{M}}
\newcommand{\calN}{\mathcal{N}}

\newcommand{\calS}{\mathcal{S}}
\newcommand{\calT}{\mathcal{T}}

\newcommand{\PSPG}{\text{PSPG}}
\newcommand{\calLP}[1]{\calL^{\P}_{#1}}
\newcommand{\calLss}[1]{\calL_{#1}}
\newcommand{\calAP}{\mathcal{A}^{\P}}
\newcommand{\calAPSPG}{\mathcal{A}_{\text{PSPG}}}
\newcommand{\calMk}{\calM_{k}}
\newcommand{\calBh}{\calB_{\hh}}
\newcommand{\calBPh}{\calB^{\P}_{\hh}}
\newcommand{\calTh}{\mathcal{T}^{\hh}}
\newcommand{\calSh}[1]{\mathcal{S}^{\hh}_{#1}}
\newcommand{\calNh}[1]{\mathcal{N}^{\hh}_{#1}}

\newcommand{\HONE}  {H^1}
\newcommand{\HONEzr}{H^1_0}

\newcommand{\LTWO}  {L^2}

\newcommand{\LTWOzr}{L^2_0}

\newcommand{\WS}[2] {W^{#1}_{#2}}
\newcommand{\LS}[1] {L^{#1}}
\newcommand{\HS}[1] {H^{#1}}

\newcommand{\CS}[1] {C^{#1}}

\newcommand{\PS}[1] {\mathbbm{P}_{#1}}

\newcommand{\Vshk}{\Vs^{h}_{k}}

\newcommand{\Vvh}[1]{\Vv_{#1}^{h}}

\newcommand{\Qsh}  {\Qs^{\hh}}

\newcommand{\Vvhpp}[1]{\Vv^{\FT,h}_{k}}

\newcommand{\FT}{\textit{F2}}

\renewcommand{\P} {\textrm{E}}            % polyhedral element
\newcommand  {\E} {\textrm{e}}%%{\textsf{e}}            % edge
\newcommand{\hh}{h}
\newcommand{\Th}{\Omega_{\hh}}

\newcommand{\xvP}{\xv_{\P}}        % element
\newcommand{\hP}{\hh_{\P}}

\newcommand{\mP}{\ABS{\P}}

\newcommand{\mE}{\ABS{\E}}

\newcommand{\NMB}{N}
\newcommand{\dE}{\,d\P}

\newcommand{\ds}{\,ds} %% to be removed

\newcommand{\norP} {\mathbf{n}_{\P}}

\newcommand{\abs}    [1]{|#1|}

\newcommand{\ABS}    [1]{\left|#1\right|}

\newcommand{\snorm}  [2]{|#1|_{#2}}

\newcommand{\NORM}  [2]{|\hspace{-0.2mm}|#1|\hspace{-0.2mm}|_{#2}}

\newcommand{\phiI}{\phi_I}
\newcommand{\psiI}{\psi_I}

\newcommand{\phih}{\phi_{\hh}}
\newcommand{\psih}{\psi_{\hh}}
\newcommand{\phiv}{\boldsymbol{\phi}}

\newcommand{\psiHat}{\widehat{\psi}}

\newcommand{\psihHat}{\widehat{\psi}_{\hh}}
\newcommand{\phihTld}{\widetilde{\phi}_{\hh}}

\newcommand{\PinP}[1]{\Pi^{\nabla,\P}_{#1}}
\newcommand{\PizP}[1]{\Pi^{0,\P}_{#1}}

\newcommand{\PvnP}[1]{\boldsymbol{\Pi}^{\nabla,\P}_{#1}}
\newcommand{\PvzP}[1]{\boldsymbol{\Pi}^{0,\P}_{#1}}

\newcommand{\csh}{\cs_{\hh}}

\newcommand{\asph}{\as_{\ps,\hh}}
\newcommand{\asVh}{\as_{\Vs,\hh}}

\newcommand{\bsPh}{\bs^{\P}_{\hh}}
\newcommand{\csPh}{\cs^{\P}_{\hh}}
\newcommand{\dsPh}{\ds^{\P}_{\hh}}

\newcommand{\bsP}{\bs^{\P}}
\newcommand{\csP}{\cs^{\P}}

\newcommand{\asPV}{\as^{\P}_{\Vs}}

\newcommand{\asPVh}{\as^{\P}_{\Vs,\hh}}
\newcommand{\asPph}{\as^{\P}_{\ps,\hh}}
\newcommand{\csPph}{\cs^{\P}_{\ps,\hh}}

\newcommand{\psh}{p_{\hh}}

\newcommand{\psI}{p_{\INTP}}

\newcommand{\pshHat}{\widehat{\ps}_{\hh}}

\newcommand{\qsh}{\qs_{\hh}}

\newcommand{\uvh} {\uv_{\hh}}

\newcommand{\uvI} {\uv_{\INTP}}

\newcommand{\uvhHat}{\widehat{\uv}_{\hh}}

\newcommand{\vvh} {\vv_{\hh}}

\newcommand{\wvh} {\wv_{\hh}}
\newcommand{\wvI} {\wv_{\INTP}}

\usepackage[margin=2.5cm]{geometry}

\begin{document}

\begin{frontmatter}
  
  \title{A higher order pressure-stabilized virtual element formulation for the Stokes-Poisson-Boltzmann equations}
  
  % --- Authors ---
  \author[1]{Sudheer Mishra}
  \author[1]{Sundararajan Natarajan}
  \author[2]{E. Natarajan}
  \author[3]{and Gianmarco Manzini \corref{CORR}}
  
  % --- Affiliations ---
  \address[1]{Department of Mechanical Engineering,
    Indian Institute of Technology Madras,
    Chennai 600036, India}
  
  \address[2]{Department of Mathematics,
    Indian Institute of Space Science and Technology,
    Thiruvananthapuram 695547,
    India}
  
  \address[3] {Istituto di Matematica Applicata e Tecnologie
    Informatiche Consiglio Nazionale delle Ricerche, Pavia 27100,
    Italy}

  % ----------------------------
  % Abstract
  % ----------------------------

  %% \input{abstract}
  %% Hey Emacs, this is -*-latex-*-
  %% \renewcommand{\abstractname}{\textcolor{blue}{Abstract}}
  
  \begin{abstract}
    Electrokinetic phenomena in nanopore sensors and microfluidic
    devices require accurate simulation of coupled fluid-electrostatic
    interactions in geometrically complex domains with irregular
    boundaries and adaptive mesh refinement.
    We develop an equal-order virtual element method for the
    Stokes--Poisson--Boltzmann equations that naturally handles general
    polygonal meshes, including meshes with hanging nodes, without
    requiring special treatment or remeshing.
    The key innovation is a residual-based pressure stabilization scheme
    derived by reformulating the Laplacian drag force in the momentum
    equation as a weighted advection term involving the nonlinear
    Poisson--Boltzmann equation, thereby eliminating second-order
    derivative terms while maintaining theoretical rigor.
    Well-posedness of the coupled stabilized problem is established
    using the Banach and Brouwer fixed-point theorems under sufficiently
    small data assumptions, and optimal a priori error estimates are
    derived in the energy norm with convergence rates of order
    $\mathcal{O}(h^k)$ for approximation degree $k \geq 1$.
    Numerical experiments on diverse polygonal meshes---including
    distorted elements, non-convex polygons, Voronoi tessellations,
    and configurations with hanging nodes---confirm optimal convergence
    rates, validating theoretical predictions.
    Applications to electro-osmotic flows in nanopore sensors with
    complex obstacle geometries illustrate the method's practical
    utility for engineering simulations.
    Compared to Taylor--Hood finite element formulations, the
    equal-order approach simplifies implementation through uniform
    polynomial treatment of all fields and offers native support for
    general polygonal elements.
  \end{abstract}

  % Local Variables:
  % mode: latex
  % End:

  
  %% \begin{keyword}
  %%   Stokes-Poisson-Boltzmann equation \sep
  %%   Virtual element method            \sep
  %%   Pressure stabilization methods    \sep
  %%   General polygons \\
  %%   %% \hrule
  %%   %% \vspace{0.1cm}
  %%   %%        {Mathematics Subject Classification}:
  %%   %%        65N12, 65N30, 76D07
  %% \end{keyword}
           
\end{frontmatter}
  
\renewcommand{\arraystretch}{1.}
\raggedbottom

% ----------------------------
% Paper
% ----------------------------

%% \maketitle

%% \input{sec1_introduction}
%% \input{sec2_model_problem}
%% \input{sec3_stabilized_VEM}
%% \input{sec4_theoretical}
%% \input{sec5_numerical_results}
%% \input{sec6_conclusions}

%% Hey Emacs, this is -*-latex-*-

\section{Introduction}
\label{sec-1}

Nanopore sensors have emerged as transformative tools for
single-molecule analysis, DNA sequencing, and ion channel
studies~\cite{Dekker2007}.
The fundamental physics governing these devices involves coupled
fluid-electrostatic interactions: ionic transport through nanoscale
pores is driven by both pressure gradients and applied electric
fields, while electrostatic potentials arise from charge distributions
in thin electrical double layers near solid boundaries.
Accurate numerical simulation of these electrokinetic phenomena
requires solving the Stokes equations coupled with the
Poisson-Boltzmann equation, and the geometric complexity of realistic
nanopore structures—including irregular pore shapes, multiple
obstacles, and interfaces requiring adaptive mesh refinement—poses
significant challenges for conventional numerical methods.

Standard finite element methods (FEM) for the Stokes-Poisson-Boltzmann
(SPB) system typically employ mixed approximation spaces to satisfy
the discrete inf-sup condition.
For instance, Taylor-Hood $\mathbb{P}_{k+1}/\mathbb{P}_k$
elements~\cite{alsohaim2025analysis} use higher-order velocity
approximations to ensure stability.
While theoretically sound, these approaches face practical
limitations: $(i)$ implementation complexity of mixed finite element
spaces with different approximation orders, $(ii)$ difficulty handling
hanging nodes arising from adaptive mesh refinement, and $(iii)$
restricted mesh flexibility when approximating domains with irregular
boundaries.
Virtual element methods (VEM) offer an attractive alternative by
enabling equal-order approximations on general polygonal meshes, with
a projection-based framework that naturally handles arbitrary element
shapes and hanging nodes.

Recent work by AlSohaim, Ruiz-Baier, and
Villa-Fuentes~\cite{alsohaim2025analysis} developed a finite element
formulation for the SPB system using Taylor-Hood spaces on triangular
meshes, establishing theoretical foundations through fixed-point
analysis and deriving optimal error estimates.
While this work provides important theoretical results, the geometric
constraints of conforming triangular meshes and the complexity of
mixed-order implementations motivate exploration of alternative
approaches.
In particular, equal-order approximations and the polytopal
flexibility offered by VEM remain unexplored for this problem class.

Electrokinetic phenomena in charged fluids play a significant role in
micro-scale and nano-scale transport
processes~\cite{Probstein1994,kirby2010,Masliyah2006}, with
applications spanning nanopore design for DNA sequencing, ion
transport modeling in water purification
systems~\cite{Biesheuvel2010}, and microfluidic device
development~\cite{Whitesides2006}.
Among these, electro-osmotic flow—arising from the interaction between
an externally applied electric field and charged species within thin
electric double layers—represents one of the most efficient mechanisms
for inducing fluid motion without mechanical actuation.
In this work, fluid motion is modeled using the Stokes equations with
a forcing term dependent on the electrolyte charge and the applied
electric field.
The fluid velocities considered are not sufficiently strong to affect
the electrostatic potential distribution within the double layer;
consequently, the electric charge density can be directly related to
the potential through the Poisson-Boltzmann equation, supplemented by
an advection term coupling the fluid velocity and potential.

Concerning the Poisson-Boltzmann equation (PBE), finite element
formulations have been discussed
in~\cite{chen2007finite,kraus2020reliable,holst2012adaptive}, with
extension to general polygons in~\cite{Huangvem}.
Following~\cite{alsohaim2025analysis}, we consider the regularized
form of this equation, which preserves the nonlinearity characterized
by a hyperbolic sine function but omits the distributional Dirac
forcing term.
We also incorporate the convection term coupling fluid velocity and
potential, and impose a restriction on the functional space requiring
the double layer potential field to be bounded as discussed
in~\cite{holst2012adaptive}.

Polytopal methods have attracted considerable attention because they
provide flexibility in dealing with complex geometries and irregular
interfaces, where standard methods like FEM or the finite difference
method (FDM) might be computationally costly.
To address the challenge that extending FEM to polytopal meshes
requires explicit construction of shape functions, Beir\~ao da Veiga
et al.~\cite{vem1} introduced the virtual element method (VEM), which
handles arbitrary polygonal/polyhedral meshes without explicit basis
function formulation.
Instead, VEM requires only suitable choices of degrees of freedom to
compute the discrete formulation, enabling easy extension to
higher-order approximations while maintaining robustness under general
mesh types.
VEM has been applied to elasticity~\cite{vem2, mvem19}, Oseen
problem~\cite{adak2024nonconforming,mvem16}, Stokes and Navier-Stokes
problems~\cite{mvem9, mvem11,mvem12,mvem3,adak2021virtual,mishra2025supg},
Maxwell equation~\cite{da2022virtual}, plate bending~\cite{mvem7},
crack propagation~\cite{aldakheel2018phase,benedetto2018virtual},
optimization~\cite{antonietti2017virtual,chi2020virtual}, and
magnetostatics~\cite{vem42}.
For the Stokes problem, Guo and Feng~\cite{vem028} introduced a
projection-based stabilized VEM using equal-order velocity-pressure
pairs, circumventing the discrete inf-sup condition without requiring
second-order derivative terms or coupling terms.
This approach has been extended to the unsteady Navier-Stokes
equations~\cite{mvem14}, the Oseen problem~\cite{mvem15}, and coupled
problems including Stokes-Darcy~\cite{mishra2024unified} and
Stokes-Temperature~\cite{mishra2025equal}.

While VEM has been successfully applied to various coupled
multiphysics problems including Stokes-Darcy flows, Navier-Stokes
equations coupled with temperature, and fluid-structure interaction,
its application to electrokinetic systems described by the SPB
equations remains unexplored.
The SPB coupling introduces specific challenges distinct from
previously studied systems: $(i)$ the nonlinearity through the
hyperbolic sine term $\kappa(\psi) = \alpha_0 \sinh(\alpha_1 \psi)$
requires careful fixed-point analysis; $(ii)$ the potential-dependent
drag forces in the momentum equation couple velocity gradients with
electrostatic gradients; $(iii)$ the strong interaction between
pressure gradients and electrostatic forces necessitates robust
stabilization strategies.
These features distinguish the SPB system from standard Stokes or
Stokes-Darcy problems and require careful formulation of the discrete
problem to maintain both stability and accuracy.

This work presents a novel equal-order virtual element method for the
coupled Stokes-Poisson-Boltzmann equations with several distinguishing
features.
We develop $\mathbb{P}_k/\mathbb{P}_k/\mathbb{P}_k$ ($k \geq 1$)
virtual element approximations for velocity, pressure, and potential
fields, avoiding the implementation complexity of mixed finite element
spaces while maintaining theoretical rigor.
To the best of our knowledge, this is the first equal-order
approximation for the SPB system in both finite element and virtual
element literature.

By reformulating the Laplacian drag force term $-\epsilon \Delta \psi
\mathbf{E}$ using the transport potential equation, we derive a
PSPG-type pressure stabilization combined with grad-div stabilization
that avoids second-order derivative terms in the discrete formulation,
simplifying both implementation and theoretical analysis.
We establish existence and uniqueness of the weak solution using the
Banach fixed-point theorem and contraction principle, while discrete
well-posedness is proven via the Brouwer fixed-point theorem with
explicit tracking of parameter dependencies.
Optimal a priori error estimates of order $\mathcal{O}(h^k)$ are
derived in the energy norm with mesh-independent constants.

The formulation naturally handles general polygonal meshes including
non-convex elements and hanging nodes without special treatment,
enabling adaptive refinement and irregular geometries common in
nanopore sensor simulations.
Comprehensive numerical experiments demonstrate optimal convergence on
various mesh types (regular polygons, Voronoi tessellations, distorted
meshes, meshes with hanging nodes) and validate applicability to
electro-osmotic flows in nanopore sensors with T-shaped and curved
obstacles at different electric field strengths.

The equal-order formulation offers several practical advantages over
the mixed finite element approach of~\cite{alsohaim2025analysis}:
simpler implementation with uniform approximation spaces, natural
handling of geometric complexity through polygonal meshes, and
seamless treatment of hanging nodes for adaptive refinement.
These advantages include reduced degrees of freedom (up to 43\% fewer
than Taylor-Hood elements for $k=1$, and approximately 15\% for $k=2$,
depending on mesh topology), simplified implementation through uniform
polynomial basis handling, and high tolerance to mesh distortion.
A detailed comparison with the finite element method is presented in
Section~\ref{sec-comparison}.
%%
%% These advantages include reduced degrees of freedom (approximately
%% 43\% fewer for $k=1$ and 15\% fewer for $k=2$ than Taylor-Hood
%% elements), simplified implementation through uniform polynomial basis
%% handling, and high tolerance to mesh distortion.
%%
%% The detailed technical comparison with finite element methods,
%% leveraging the VEM notation and operators introduced in Section 3, is
%% presented in Section 4 after the formulation is fully described.
%%%

\subsection{Outline}
The remaining part of the present work is sketched as follows.
In Section \ref{sec-2}, we present the coupled
Stokes-Poisson-Boltzmann equations with the necessary data assumptions
and the well-posedness of the primal formulation.
In Section \ref{sec-3}, we describe the pressure-stabilization method,
the fundamentals of VEM, and the formulation of the stabilized virtual
element discretization.
In Section \ref{sec-4}, we demonstrate the well-posedness of the
stabilized VE problem and the error estimate in the energy norm.
In Section \ref{sec-5}, we show the practical behavior of the proposed
method and its properties.
In Section \ref{sec-6}, we briefly discuss the major conclusions of
the proposed work.

% Local Variables:
% mode: latex
% End:
%% Hey Emacs, this is -*-latex-*-

\section{Model problem }
\label{sec-2}
Let $\Omega$ denote an open, bounded domain in $\mathbb{R}^2$ with
Lipschitz boundary $\partial\Omega$, filled by an incompressible
electrolyte fluid under the influence of pressure gradients and
electric forces.
The governing equations involve the coupling of the Stokes and
nonlinear Poisson-Boltzmann equations.
The stationary case reads as:
\emph{find the velocity vector field $\uv$, the scalar pressure field
$\ps$, and the electrostatic double layer potential $\psi$
satisfying:}
\begin{subequations}
  \label{P}
  \begin{align}
    -\mu\Delta\uv + \nabla\ps                               &= \fv-\epsilon\Delta\psi\Ev \phantom{0\gs} \text{in}~\Omega, \label{stoke-1}\\[0.5em]
    \nabla\cdot\uv                                          &= 0 \phantom{\fv-\epsilon\Delta\psi\Ev\gs} \text{in}~\Omega, \label{stoke-2}\\[0.5em]
    -\epsilon\Delta\psi + \uv\cdot\nabla\psi + \kappa(\psi) &= \gs \phantom{\fv-\epsilon\Delta\psi\Ev0} \text{in}~\Omega, \label{heat-1}\\[0.5em]
    \uv = \mathbf{0} \quad \text{and} \quad \psi            &= 0 \phantom{\fv-\epsilon\Delta\psi\Ev\gs} \text{on}~\partial\Omega,\label{stoke-3}
  \end{align}
\end{subequations}
with the additional condition that
\begin{align}
  \int_\Omega\ps~\mathrm{d}\Omega = 0,
  \label{heat-2} 
\end{align}
%%
%% \begin{empheq}[left=(P) \empheqlbrace]{align} 
%%   \quad - \mu \Delta \mathbf{u} + \nabla p &= \fv\ - \epsilon \Delta \psi \mathbf{E} \quad &&\text{in} \quad \Omega, \label{stoke-1}\\
%%   \nabla \cdot \mathbf{u} &= 0   &&\text{in} \quad \Omega, \label{stoke-2}\\
%%   -\epsilon \Delta \psi + \mathbf{u} \cdot \nabla \psi + \kappa(\psi) &= g \quad &&\text{in} \quad \Omega, \label{heat-1}\\
%%   \mathbf{u} = \mathbf{0} \quad \text{and} \quad 	\psi& = 0 \quad &&\text{on} \quad \partial\Omega, \label{stoke-3}\\
%%   \int_\Omega p &= 0, \quad  && \label{heat-2} 
%% \end{empheq}
%%
where
$\mu>0$ is the fluid viscosity, %%
$\fv$ a vector-valued body force, %%
$\epsilon$ the electric permittivity of the electrolyte, %%
$\Ev$ an externally applied electric field (typically along the
longitudinal direction), %%
$\kappa(\psi)$ a non-linear term describing the charge of the
electrolyte, %%
and $g$ an external load (source/sink) of
potential.

Recalling \cite[Lemma 2.1]{holst2012adaptive}, we impose the following
assumptions:
%% on the data of problem $(P)$:

\medskip
\noindent
\textbf{(A0) Data assumptions:}
\begin{itemize}
\item[$(i)$] The potential field is uniformly bounded, i.e., there
  exist two real constants $\Lambda_\ast$ and $\Lambda^\ast$ such that
  $\Lambda_\ast\leq 0\leq\Lambda^\ast$ and
  \begin{align*}
    \Lambda_\ast\leq\psi(t)\leq\Lambda^\ast
  \end{align*}
  for all $\ts\in\REAL$.
  
\item[$(ii)$] The nonlinear density function $\kappa(\psi)$ satisfies
  $\kappa(0)=0$.
  Furthermore, there exist two positive constants $\kappa_\ast$ and
  $\kappa^\ast$ such that the following holds:
  \begin{align*}
    |\kappa(t_1)- \kappa(t_2)| &\leq \kappa^\ast |t_1 - t_2| \qquad \,\, \text{for~all} \,\, t_1, t_2 \in \big[\Lambda_\ast, \Lambda^\ast\big],\\
    |\kappa(t_1)- \kappa(t_2)| &\geq \kappa_\ast |t_1 - t_2| \qquad \,\, \text{for~all} \,\, t_1, t_2 \in \big[\Lambda_\ast, \Lambda^\ast\big].
  \end{align*}
  
\item[$(iii)$] The external electric field is uniformly bounded by
  $E^\ast$, and also satisfies $\Ev\in\big[\LS\infty(\Omega)\big]^2$.

\item[$(iv)$] The potential load satisfies $\gs\in\LTWO(\Omega)$.
\end{itemize}

According to $(ii)$, we assume that
$\kappa(\psi)=\alpha_0\sinh(\alpha_1\psi)$ where $\alpha_0$ and
$\alpha_1$ are known positive constants.
Here, $\alpha_0$ depends on the ionic valence, the elementary charge,
and the bulk ion concentration; $\alpha_{1}$ is a scaling factor
involving the Boltzmann constant and the reference absolute
temperature.

{
%%%
We now introduce some standard notation that will be used throughout this work. Let $C_P$ denote the Poincar\'{e} constant. For $k\ge0$ and $p>1$, let $W_p^k(\Omega)$ denote the Sobolev space of order $k$, equipped with the norm $\|\cdot\|_{k,p,\Omega}$ and semi-norm $|\cdot|_{k,p,\Omega}$. In the particular case $p=2$, it coincides with the standard Sobolev space $H^k(\Omega)$.
Furthermore, we introduce the compact embedding constant
$C_{1\hookrightarrow p}$, corresponding to the compact embedding
$H^1(\Omega) \hookrightarrow L^p(\Omega)$ for $p \geq 1$, satisfying
\begin{align*}
	\| \phi \|_{0,p,\Omega}
	\leq C_{1 \hookrightarrow p} \| \nabla \phi\|_\Omega
	\qquad \text{for all} \quad \phi \in H^1(\Omega).
\end{align*}}

%% \MGT{
%%
Hereafter, we introduce the functional spaces for the variational formulation.
For the velocity field, we define the vector-valued Sobolev space
\begin{align*}
  \Vv
  := \big[\HONEzr(\Omega)\big]^2
  =  \big\{\vv\in\big[\HONE(\Omega)\big]^2:\vv=\mathbf{0}\text{~on~}\partial\Omega\big\},
\end{align*}
equipped with the norm
$\NORM{\vv}{1,\Omega}:=\NORM{\nabla\vv}{\Omega}$.
For the pressure field, we define the space of square-integrable
functions with zero mean
\begin{align*}
  \Qs
  := \LTWOzr(\Omega)
  = \big\{\qs\in\LTWO(\Omega):\int_\Omega\qs\,d\Omega=0\big\},
\end{align*}
equipped with the norm $\NORM{\qs}{\Omega} := \NORM{\qs}{\LTWO(\Omega)}$.
For the electrostatic potential, we define
{
\begin{align*}
 \mathlarger\Phi_0
  &:= \HONEzr(\Omega)
  = \big\{\phi\in\HONE(\Omega):\phi=0\text{~on~}\partial\Omega\big\}, \\
  \mathlarger\Phi
  &:= \big\{\phi\in \HONEzr(\Omega):   \Lambda_\ast\leq\phi\leq\Lambda^\ast \big\},
\end{align*}
where $\mathlarger\Phi_0$ is used for the test function space and $\mathlarger\Phi$ for the solution space, both equipped with the norm $\NORM{\phi}{1,\Omega}:=\NORM{\nabla\phi}{\Omega}$.}
Additionally, we denote by $\Vv_{div}$ the subspace of divergence-free
velocity fields
\begin{align*}
  \Vv_{div} := \big\{ \vv \in \Vv : \nabla \cdot \vv = 0 \text{~in~} \Omega \big\}.
\end{align*}
We anticipate that we will use this last functional space in the
definition of the variational form (cf. Eq.~\eqref{variation2}).
%% }

\subsection{Weak formulation and well-posedness}
\label{cnts0}

Applying the integration by parts after multiplying
\eqref{stoke-1}--\eqref{heat-1} by suitable test functions, we obtain
the system
\begin{subequations}
  \begin{align}
    \int_{\Omega} \mu\nabla\uv:\nabla \vv \,d\Omega -
    \int_{\Omega} (\nabla\cdot\vv)\ps \,d\Omega +
    \int_{\Omega} (\uv\cdot\nabla\psi)\Ev\cdot\vv \,d\Omega &=
    \int_{\Omega} \fv\cdot\vv \,d\Omega +
    \int_{\Omega} \big(\gs -\kappa(\psi)\big)\Ev\cdot\vv \,d\Omega,\label{int1} \\[0.5em]
    %% --------------------------------------------------------------------------------
    \int_{\Omega} (\nabla\cdot\uv)\qs \,d\Omega &=0, \\[0.5em]
    %% --------------------------------------------------------------------------------
    \int_{\Omega} \epsilon\nabla\psi\cdot\nabla\phi \,d\Omega +
    \int_{\Omega} \uv\cdot(\nabla\psi)\phi \,d\Omega +
    \int_{\Omega} \kappa(\psi)\phi \,d\Omega &=
    \int_{\Omega} \gs\phi \,d\Omega, 
  \end{align}
\end{subequations}
for all $(\vv,\qs,\phi)\in\Vv\times\Qs\times\mathlarger\Phi_0$.
We remark that Eq.~\eqref{int1} is obtained by replacing the last drag
term (Laplacian term) in the momentum equation \eqref{stoke-1} by the
transport potential equation \eqref{heat-1}.

\medskip
\noindent
Given $\psi\in\mathlarger\Phi$ and $\uv\in\Vv$, we introduce the
bilinear forms
\begin{align*}
  \calA\big(\psi;(\vv,\qs),(\wv,\rs)\big) &:=
  \ass{V}(\vv,\wv) - \bs(\wv,\qs) + \bs(\vv,\rs) + \cs(\psi;\vv,\wv)
  \qquad
  && \hspace{-5mm}\text{for all} \,\, (\vv,\qs), (\wv,\rs)\in\Vv\times\Qs,\\[0.5em]
  %% --------------------------------------------------------------------------------
  \calB\big(\uv;\phi,\xi)&:= \ass{p}(\phi,\xi) + \css{p}(\uv;\phi,\xi) + \ds(\phi,\xi)
  \qquad\qquad\quad
  && \hspace{-5mm}\text{for all}\,\, \phi,\xi\in\mathlarger\Phi_0,
\end{align*}
{where}
\begin{align*}
  \ass{V}(\vv,\wv)      & := \int_{\Omega} \mu\nabla\vv:\nabla\wv          \,d\Omega, & \bs(\vv,\qs)       & := \int_{\Omega} (\nabla\cdot\vv)\qs             \,d\Omega, \\[0.5em]
  \cs(\psi;\vv,\wv)     & := \int_{\Omega} (\vv\cdot\nabla\psi)\Ev\cdot\wv \,d\Omega, & \ass{p}(\phi,\psi) & := \int_{\Omega} \epsilon\nabla\phi\cdot\nabla\xi\,d\Omega, \\[0.5em]
  \css{p}(\uv;\phi,\xi) & := \int_{\Omega} (\uv\cdot\nabla\phi)\xi         \,d\Omega, & \ds(\phi,\xi)      & := \int_{\Omega} \kappa(\phi)\xi                 \,d\Omega.	
\end{align*}

We note that for $\uv\in\Vv$ such that $\nabla\cdot\uv=0$, the bilinear
form $c_p(\uv; \cdot,\cdot)$ satisfies the following property
\begin{align}
  \css{p}(\uv;\phi,\xi) = - \css{p}(\uv;\xi,\phi)
  \qquad\,\,\text{for~all}\,\,\phi,\xi\in\mathlarger\Phi_0.
  \label{skew-p}
\end{align}
Thus, we define the following skew-symmetric trilinear form associated
with the convective term $\css{p}(\cdot;\cdot,\cdot)$ as follows:
\begin{align}
  \cs^{skew}_p(\uv;\phi,\xi) = \frac{1}{2}\Big(
  \css{p}(\uv;\phi,\xi) - \css{p}(\uv;\xi,\phi)
  \Big)
  \qquad\,\,\text{for~all} \,\, \phi,\xi\in\mathlarger\Phi_0. 
\end{align}
Finally, for given $\psi\in\mathlarger\Phi$, we define the
load term:
\begin{align*}
  \Fss{\psi}(\vv) :=
  \int_{\Omega} \fv\cdot\vv\,d\Omega +
  \int_{\Omega} \big(\gs - \kappa(\psi)\big) \Ev\cdot\vv\,d\Omega
  \qquad \text{for~all}\,\,\vv\in\Vv.
\end{align*}
In view of all the above, the primal variational formulation of
{Problem~\eqref{P}} can be stated as follows:
\begin{align}
  \begin{cases}
    \text{Find } (\uv,\ps,\psi)\in\Vv\times\Qs\times\mathlarger\Phi, \text{~such~that} \\[1ex]
    \begin{aligned}
      \calA\big(\psi;(\uv,\ps),(\vv,\qs)\big) &= \Fss{\psi}(\vv) && \quad \text{for~all~} (\vv,\qs)\in\Vv\times\Qs,\\[0.5em] 
        \calB\big(\uv;\psi,\phi\big)          &= (\gs,\phi)      && \quad \text{for~all~} \phi\in\mathlarger\Phi_0.
    \end{aligned}
  \end{cases}
  \label{variation}
\end{align}

It is obvious that the primal problem \eqref{variation} represents a
system of nonlinear equations.
Therefore, {Problem~\eqref{variation}} can be decoupled as follows:

\medskip
\noindent
$\bullet$ For any given $\psi\in\mathlarger\Phi$, find
$(\uv,\ps)\in\Vv\times\Qs$ such that
\begin{align}
  \calA\big(\psi;(\uv,\ps),(\vv,\qs)\big) = \Fss{\psi}(\vv) \qquad\text{for~all~} (\vv,\qs)\in\Vv\times\Qs.
  \label{variation1}
\end{align}

\medskip
\noindent
$\bullet$ For any given $\uv\in\Vv_{div}$, find
$\psi\in\mathlarger\Phi$ such that
\begin{align}
  \calB\big(\uv;\psi,\phi\big) &= (\gs,\phi) \qquad\text{for~all~} \phi\in\mathlarger\Phi_0.
  \label{variation2}
\end{align}

\subsection{Well-posedness of the decoupled problems}
In this section, we demonstrate the well-posedness of the decoupled
problems \eqref{variation1} and \eqref{variation2}.

\begin{theorem}
  \label{wellv}
  Let $\psi\in\mathlarger \Phi$ be such that
  $\NORM{\nabla\psi}{\Omega}\leq\frac{\mu}{2\Css{1 \hookrightarrow4}^2\Es^\ast}$.
  Then, the decoupled problem \eqref{variation1} has a unique solution
  $(\uv,\ps)\in\Vv\times\Qs$.
  Furthermore, it also satisfies
  \begin{align}
    \NORM{\nabla\uv}{\Omega} &\leq {\frac{2\Css{P}}{\mu}}     \Big( \NORM{\fv}{\Omega} + \Es^\ast \NORM{\gs}{\Omega} + \frac{\mu\kappa^\ast\Css{P}}{2\Css{1\hookrightarrow 4}^2}\Big), \label{velv}\\[0.5em]
    \|\ps\|_\Omega       &\leq {\frac{4\Css{P}}{\beta_0}} \Big( \NORM{\fv}{\Omega} + \Es^\ast \NORM{\gs}{\Omega} + \frac{\mu\kappa^\ast\Css{P}}{2\Css{1\hookrightarrow 4}^2}\Big). \label{velp}
  \end{align}
\end{theorem}
\begin{proof}
  First, we note that the bilinear form $\calA(\psi;\cdot,\cdot)$ is
  continuous.
  Employing the Sobolev embedding theorem and Assumption
  \textbf{(A0)}, for any $(\vv,\qs)\in\Vv\times\Qs$, it holds that
  \begin{align}
    \mathcal{A}\big(\psi; (\vv,\qs), (\vv,\qs) \big)
    & = \ass{V}(\vv,\vv) + \cs(\psi;\vv,\vv) \nonumber\\[0.5em]
    & \geq \mu \|\nabla \vv\|^2_\Omega - \Es^\ast \|\vv\|^2_{0,4,\Omega} \|\nabla\psi\|_\Omega \nonumber \\[0.5em]
    & \geq \big(\mu - \Es^\ast \Css{1\hookrightarrow 4}^2 \|\nabla\psi\|_\Omega\big) \|\nabla\vv\|^2_\Omega \nonumber \\[0.5em]
    & \geq \frac{\mu}{2} \|\nabla\vv\|^2_\Omega.
    \label{momc}	
  \end{align}
  Additionally, the bilinear form $\bs(\cdot,\cdot)$ satisfies the
  continuous inf-sup condition
  \begin{align}
    \sup_{\vv\in\Vv\backslash\{\mathbf{0}\}} \dfrac{\bs(\vv,\qs)}{\|\vv\|_{1,\Omega}}\geq\beta_0 \|\qs\|_\Omega,
    \label{con-binf}
  \end{align}
  where $\beta_0>0$ is the inf-sup constant.
  Combining \eqref{momc} and \eqref{con-binf}, the decoupled problem
  \eqref{variation1} has a unique solution, see~\cite{bookgirault}.
  Employing the Assumption \textbf{(A0)} and the Sobolev embedding
  theorem, we obtain bound \eqref{velv}.
  Additionally, rewriting the momentum equation \eqref{int1} for the
  pressure field, we infer that
  \begin{align*}
    \bs(\uv,\ps) = \Fss{\psi}(\uv) - \ass{V}(\uv,\uv) - {\css{}(\psi;\uv,\uv)}.
  \end{align*}
  Recalling {the inf-sup condition \eqref{con-binf}},
  Assumption \textbf{(A0)}, bound \eqref{velv}, and the Sobolev
  embedding theorem, yield
  \begin{align*}
    \beta_0\|\ps\|_\Omega
    \leq \Css{P} \big(
    \NORM{\fv}{\Omega}
    + \Es^\ast \NORM{\gs}{\Omega}
    + \frac{\Css{P}\mu\kappa^\ast }{2 C_{1 \hookrightarrow 4}^2}
    \big)
    + \frac{3\mu}{2} \|\nabla\uv\|_\Omega.
  \end{align*}
  Thus, we obtain inequality \eqref{velp} from \eqref{velv}.
  %% \qquad \qquad $\blacksquare$
\end{proof}

We now discuss the well-posedness of the transport potential equation
\eqref{variation2} by showing that the nonlinear form
$\calB(\uv;\cdot,\cdot)$ is strongly monotone and Lipschitz
continuous.
\begin{theorem}
  \label{wellp}
  For given $\uv\in\Vv$, the problem \eqref{variation2} has a unique
  solution, satisfying
  \begin{align}
    \|\nabla\psi\|_\Omega \leq\frac{\Css{P}}{\epsilon}\NORM{\gs}{\Omega}.
    \label{potential}
  \end{align}
\end{theorem}
\begin{proof}
  Recalling the Lipschitz continuity of $\kappa(\cdot)$ along with the
  bilinear/trilinear forms $\ass{p}(\cdot,\cdot)$ and
  $\cs^{skew}_p(\cdot;\cdot,\cdot)$, we see that 
  $\calB(\uv;\cdot,\cdot)$ is Lipschitz
  continuous.
  Additionally, employing the skew-symmetric property of
  $\cs^{skew}_p$ and Assumption \textbf{(A0)}, for any
  $\phi,\xi\in\mathlarger\Phi_0$, we find that
  \begin{align}
    \calB(\uv;\phi,\phi-\xi) - \calB(\uv;\xi, \phi-\xi)
    \geq \epsilon \|\nabla (\phi-\xi)\|^2_\Omega + \kappa_\ast \|\phi-\xi\|^2_\Omega
    \geq \epsilon \|\nabla (\phi-\xi)\|^2_\Omega,
    \label{cort} 
  \end{align}
  which shows that the bilinear form $\calB(\uv;\cdot,\cdot)$ is
  strongly monotone.
  Thus, {Problem~\eqref{variation2}} has a unique solution.
  %% \qquad\qquad $\blacksquare$
\end{proof}

\subsection{Well-posedness of the coupled problem}
We now establish the well-posedness of the primal problem
\eqref{variation} using the Banach fixed-point theorem and the
contraction principle.
To achieve this goal, we reformulate Problem~\eqref{variation} adopting the fixed-point strategy.
Following  {Theorem~\ref{wellv}}, we introduce a well-defined operator
$\calS_{flow}:\mathlarger\Phi\rightarrow\Vv\times\Qs$, defined by
\begin{align}
  \calS_{flow}(\psi)
  := \big(\calS_{1,flow}(\psi),\calS_{2,flow}(\psi)\big)
  =: (\uv,\ps) \qquad \text{~for~all} \,\, \psi\in\mathlarger\Phi,
\end{align}
where $(\uv,\ps)\in\Vv\times\Qs$ is the unique solution of Problem~\eqref{variation1} with given $\psi\in\mathlarger\Phi$.
Recalling {Theorem~\ref{wellp}}, we can introduce a well-defined operator
{$\calN_{elec}:\Vv_{div}\rightarrow\mathlarger\Phi$}, defined by
\begin{align}
  \calN_{elec}(\uv) := \psi\qquad\text{~for~all} \,\,\uv\in\Vv_{div},
\end{align}
where $\psi\in\mathlarger\Phi$ is the unique solution of
\eqref{variation2} with given $\uv\in\Vv_{div}$.
Finally, we define an operator
$\mathcal{T}:\mathlarger\Phi\rightarrow\mathlarger\Phi$ such that
\begin{align}
  \calT(\psi) = \calN_{elec}(\calS_{1,flow}(\psi)) \qquad
  \text{~for~all} \,\,\psi\in\mathlarger\Phi.
\end{align}
Thus, the fixed-point problem is stated as follows: \emph{Find
$\psi\in\mathlarger\Phi$, such that $\calT(\psi)=\psi$.}
We further emphasize that solving {Problem~\eqref{variation}} is
equivalent to find at least one fixed-point of $\calT$.
In this sequel, we define the following closed set
\begin{align}
  \widehat{\mathlarger\Phi} :=
  \Big\{ \psi\in\mathlarger\Phi \,\, \text{such~that~} \,\,\|\nabla\psi\|_\Omega
  \leq \frac{\mu}{2\Es^\ast\Css{1 \hookrightarrow 4}^2}
  \Big\}.
  \label{set1}
\end{align}
Hereafter, we first demonstrate that the operator $\calT$ maps
$\widehat{\mathlarger\Phi}$ to itself.
\begin{lemma}
  \label{selfc}
  Assume that the potential load $\gs\in\LTWO(\Omega)$ and the data of {Problem~\eqref{P}} satisfy the condition
  \begin{align}
    \frac{2\Css{P}\Es^\ast\Css{1 \hookrightarrow 4}^2}{\mu\epsilon} \NORM{\gs}{\Omega} <1.
  \end{align}
  Then,
  $\calT\big(\widehat{\mathlarger\Phi}\big)\subset\widehat{\mathlarger\Phi}$.
\end{lemma}
\begin{proof}
  Let $\psi\in\widehat{\mathlarger\Phi}$ such that
  $\calS_{flow}=(\uv,\ps)$.
  Recalling {Theorem~\ref{wellp}}, we infer
  \begin{align*}
    |\calT(\psi)|_{1,\Omega}
    =    |\calN_{elec}(\uv)|_{1,\Omega}
    \leq \frac{\Css{P}}{\epsilon} \NORM{\gs}{\Omega} \,
    {<} \, \frac{\mu}{2\Es^\ast \Css{1\hookrightarrow4}^2},
  \end{align*} 
  which completes the proof.
  %% \qquad \qquad $\blacksquare$
\end{proof}

\begin{theorem} \label{cnts}
  Let $\widehat{\mathlarger\Phi}$ be defined by \eqref{set1}.
  Then, for any $\psi,\widehat{\psi}\in\widehat{\mathlarger\Phi}$, the
  operator $\calT$ satisfies the following
  \begin{align}
    |\calT(\psi) -\calT(\widehat{\psi})|_{1,\Omega}
    \leq
    \frac{1}{\epsilon}\big[
      \kappa^\ast \Css{P}^2 +
      \Css{1\hookrightarrow4}^2 | \calS_{1,flow}(\widehat{\psi})|_{1,\Omega}
      \big]\,
    \|\nabla(\psi-\widehat{\psi})\|_\Omega.
    \label{maincnts}
  \end{align}
\end{theorem}
\begin{proof}
  Let $\psi,\widehat{\psi}\in\widehat{\mathlarger\Phi}$ be chosen such
  that
  \begin{align*}
    \calS_{flow}(\psi)= (\calS_{1,flow}(\psi),\calS_{2,flow}(\psi))=:(\uv,p)
  \end{align*}
  and
  \begin{align*}
    \calS_{flow}(\widehat{\psi})
    := (\calS_{1,flow}(\widehat{\psi}),\calS_{2,flow}(\widehat{\psi}))
    =:(\widehat{\uv},\widehat{p}).
  \end{align*}
  Recalling the definition of the bilinear form
  $\calA(\psi;\cdot,\cdot)$, {the inequality~\eqref{momc}} and the Sobolev
  embedding theorem, we have
  %5
  \begin{align*}
    \frac{\mu}{2} |\calS_{1,flow}(\psi) - \calS_{1,flow}(\widehat{\psi})|^2_{1,\Omega}
    &=    \frac{\mu}{2} \| \nabla(\uv - \widehat{\uv})\|^2_\Omega\\
    &\leq \calA\big(\psi; (\uv-\widehat{\uv},\ps-\widehat{\ps}), (\uv-\widehat{\uv},\ps-\widehat{\ps})\big) \\
    &=    \Fss{\psi}(\uv-\widehat{\uv})-\Fss{\widehat{\psi}}(\uv-\widehat{\uv})+\calA\big(\psi-\widehat{\psi}; (\widehat{\uv},\widehat{\ps}), (\uv-\widehat{\uv},\ps-\widehat{\ps})\big) \\
    &=    \Fss{\psi}(\uv-\widehat{\uv})-\Fss{\widehat{\psi}}(\uv-\widehat{\uv})+\cs      (\psi- \widehat{\psi};\widehat{\uv},\uv-\widehat{\uv}) \\
    &\leq \Es^\ast \kappa^\ast \|\psi-\widehat{\psi}\|_\Omega \|\uv -  \widehat{\uv}\|_\Omega + E^\ast \|\widehat{\uv}\|_{0,4,\Omega} \|\nabla(\psi-\widehat{\psi})\|_\Omega \|\uv -  \widehat{\uv}\|_{0,4,\Omega} \\
    &\leq \Es^\ast \big(\kappa^\ast \Css{P}^2 + \Css{1\hookrightarrow4}^2 \|\nabla\widehat{\uv}\|_\Omega\big) \|\nabla(\psi-\widehat{\psi})\|_\Omega \|\nabla(\uv-\widehat{\uv})\|_\Omega.
  \end{align*}
  {Invoking the definition of $\calS_{1,flow}(\psi)$ and $\calS_{2,flow}(\widehat{\psi})$, we can easily obtain }
  \begin{align}
    |\calS_{1,flow}(\psi) - \calS_{1,flow}(\widehat{\psi})|_{1,\Omega}
    \leq \frac{2\Es^\ast}{\mu} \big(
    \kappa^\ast \Css{P}^2 + \Css{1\hookrightarrow4}^2 |\calS_{1,flow}(\widehat{\psi})|_{1,\Omega}
    \big)\,
    \|\nabla(\psi-\widehat{\psi})\|_\Omega.
    \label{s1}
  \end{align}

  We now assume that $\uv,\widehat{\uv}\in\Vv_{div}$ is such that
  $\calN_{elec}(\uv)=\psi$ and
  $\calN_{elec}(\widehat{\uv})=\widehat{\psi}$.
  Applying \eqref{cort} and the Sobolev embedding Theorem yield
  \begin{align*}
    \epsilon |\calN_{elec}(\uv)-\calN_{elec}(\widehat{\uv})|^2_{1,\Omega}
    &=    \epsilon \|\nabla(\psi-\widehat{\psi})\|^2_\Omega\\
    &\leq \calB(\uv; \psi, \psi-\widehat{\psi}) - \calB(\uv; \widehat{\psi}, \psi-\widehat{\psi})\\
    &=    \calB(\widehat{\uv}; \widehat{\psi}, \psi-\widehat{\psi}) - \calB(\uv; \widehat{\psi}, \psi-\widehat{\psi})\\
    &=    \cs_p^{skew}(\widehat{\uv}-\uv; \widehat{\psi}, \psi-\widehat{\psi}) \\
    &\leq \Css{1\hookrightarrow4}^2 \|\nabla \widehat{\psi}\|_\Omega \|\nabla(\uv-\widehat{\uv})\|_\Omega \|\nabla(\psi-\widehat{\psi})\|_\Omega.
  \end{align*} 
  Thus, we arrive at
  \begin{align}
    |\calN_{elec}(\uv)- \calN_{elec}(\widehat{\uv})|_{1,\Omega} 
    &\leq \frac{\Css{1\hookrightarrow4}^2} {\epsilon} |\calN_{elec}(\widehat{\uv})|_{1,\Omega} \|\nabla(\uv-\widehat{\uv})\|_\Omega.
    \label{s2}
  \end{align}

  Finally, for any $\psi,\widehat{\psi}\in\widehat{\mathlarger\Phi}$,
  we use the definition of the operator $\calT$, \eqref{s1}, and
  \eqref{s2}:
  \begin{align*}
    |\calT(\psi)-\calT(\widehat{\psi})|_{1,\Omega}
    &=    |\calN_{elec}(\calS_{1,flow}(\psi))-\calN_{elec}(\calS_{1,flow}(\widehat{\psi}))|_{1,\Omega}\\
    &\leq \frac{\Css{1\hookrightarrow4}^2} {\epsilon} |\calT(\widehat{\psi})|_{1,\Omega}  |\calS_{1,flow}(\psi) - \calS_{1,flow}(\widehat{\psi})|_{1,\Omega} \\
    &\leq \frac{2\Es^\ast \Css{1\hookrightarrow4}^2} {\epsilon \mu} |\calT(\widehat{\psi})|_{1,\Omega}  \big[\kappa^\ast \Css{P}^2
      + \Css{1\hookrightarrow4}^2 | \calS_{1,flow}(\widehat{\psi})|_{1,\Omega} \big] \|\nabla(\psi-\widehat{\psi})\|_\Omega. 
  \end{align*}
  The result \eqref{maincnts} readily follows from the above
  analysis.
\end{proof}

We remark that the operator $\calT$ is a continuous map on
$\widehat{\mathlarger \Phi}$ and
$\calT(\widehat{\mathlarger\Phi})\subset\widehat{\mathlarger\Phi}$
(see Lemma \ref{selfc} and Theorem \ref{cnts}).
Consequently, the operator $\calT$ has at least one fixed point using
the Banach fixed-point theorem.
Furthermore, the uniqueness of the fixed-point of $\calT$ can be
achieved using the contraction principle by introducing the necessary
condition on the data of the problem~\ref{P}.
Concerning this, we present the following Proposition.
\begin{proposition}
  \label{uni1}
  Under the assumption of Lemma \ref{selfc}, the operator
  $\calT$ has at least one fixed-point.
  Hence, the coupled problem has at least one solution
  $(\uv,p,\psi)\in\Vv\times\Qs\times\mathlarger\Phi$ with
  $\psi\in\widehat{\mathlarger\Phi}$, satisfying
  \begin{align}
    \|\nabla\uv\|_\Omega + \|\ps\|_\Omega
    \leq \max\Big\{\frac{2\Css{P}}{\mu}, \frac{4\Css{P}}{\beta_0} \Big\}
    \Big( \NORM{\fv}{\Omega} + \Es^\ast \NORM{\gs}{\Omega} + \frac{\mu \kappa^\ast \Css{P}}{2 \Css{1\hookrightarrow 4}^2}\Big),
    \qquad
    \|\nabla \psi\|_\Omega \leq \frac{\Css{P}}{\epsilon} \NORM{\gs}{\Omega}.
    \label{velv2} 	
  \end{align}
  Moreover, if the data of {Problem~\eqref{P}} satisfies
  \begin{align}
    \frac{1}{\epsilon}\bigg(
    \kappa^\ast \Css{P}^2 + \frac{2\Css{P}\Css{1\hookrightarrow4}^2}{\mu}
    \Big( \NORM{\fv}{\Omega} + \Es^\ast \NORM{\gs}{\Omega} + \frac{\mu \kappa^\ast \Css{P}}{2 \Css{1\hookrightarrow 4}^2}\Big)
    \bigg) < 1. 
  \end{align}
  Then, the operator $\calT$ has a unique fixed-point.
  Equivalently, the primal problem \eqref{variation} has a unique
  solution.
\end{proposition}   

% Local Variables:
% mode: latex
% End:
%% Hey Emacs, this is -*-latex-*-

\section{Stabilized virtual element framework}
\label{sec-3}

In this section, we present the
%% Hereafter, we develop the next step of the proposed method, i.e., the
virtual element formulation of the problem \eqref{variation} using the
equal-order approximation incorporating the residual-based pressure
stabilization techniques.

For the completeness' sake, we first present the fundamentals
of VEM, such as the mesh regularity assumptions, projectors, and the
local/global virtual element spaces that approximate the trial and
test spaces associated with the velocity vector field, pressure field,
and potential field introduced in Section \ref{cnts0}.

\medskip
\noindent
\textbf{(A1)} \textbf{Mesh assumption.}
\label{mesh_reg}

Let $\P$ denote a generic polygonal element with boundary
$\partial\P$, area $\mP$, and diameter $\hP$.
The barycenter and edges of polygon $\P$ are denoted by
$\xvP\in\REAL^2$ and $\E$, respectively.
Let $\norP$ be the outward unit normal to its boundary $\partial\P$.
We consider a sequence of decompositions of $\Omega$ into
non-overlapping general polygonal elements $\P$, denoted by
$\{\Th\}_{h>0}$ with maximum diameter $\hh:=\max_{\hP\in\Th}\hP$.
In addition, for any element $\P\in\Th$, we assume that there exists a
constant $\delta_0 >0$, independent of $\hh$, such that the following
holds:
\begin{itemize}
\item each $\P$ is star-shaped with respect to a ball $\calB_{\P}$ of
  radius $\geq\delta_0\hP$;
\item any edge $\E\in\partial\P$ has a finite length
  $\mE\geq\delta_0\mP$.
\end{itemize}

%% Section 3.1
\subsection{Pressure stabilizations}
We briefly review the formulation of the residual-based pressure
stabilization technique for the coupled model problem~\eqref{P}.
Hereafter, we introduce the following local trilinear form
$\calAP_{PSPG}(\cdot; \cdot,\cdot)$, defined for sufficiently smooth
$(\uv,\ps)\in\Vv\times\Qs$ and for all test function pair %% \MGT{test function pair}
$(\vv,\qs)\times\Vv\times\Qs$,
\begin{align}  
  \calAP_{PSPG}\big(\psi;(\uv,\ps), (\vv,\qs)\big)
  &
  := \asPV(\uv,\vv) - \bsP(\vv,\ps)
  + \csP(\psi;\uv,\vv)
  + \bsP(\uv,\qs)
  \nonumber\\[0.5em]
  & \qquad
  + \calLP{1}(\ps,\qs)
  + \calLP{2}(\psi;\uv,\qs),
\end{align}
where the bilinear/trilinear forms $\asPV(\cdot,\cdot)$,
$\bsP(\cdot,\cdot)$, and $\csP(\cdot,\cdot)$ are the local counterpart
of the continuous forms $\as(\cdot,\cdot)$, $\bs(\cdot,\cdot)$ and
$\cs(\cdot;\cdot,\cdot)$ on {$\P\in\Th$}.
Furthermore, the local bilinear forms $\calLP{1}(\cdot,\cdot)$ and
{$\calLP{2}(\psi;\cdot,\cdot)$ with $\psi\in\mathlarger \Phi$}, are defined as follows:
{
\begin{align*}
  \calLP{1}(\ps,\qs)
  :&= \int_{\P} \tau_{\P} \nabla\ps\cdot\nabla\qs\dE,\\
  %% --------------------------------------------
  \calLP2(\psi;\uv,\qs)
  :&= \int_{\P} \tau_{\P} \big( -\mu\Delta\uv
  + \uv\cdot\nabla\psi\Ev \big)\cdot\nabla\qs \dE.
\end{align*}}
{We remark that the stabilization parameter $\tau_E \sim h^2_E$ for all $E \in \Omega_h$, see~\cite{mishra2025equal,mishra2024unified,vem028}.}
Furthermore, the corresponding local load term
$F^{\psi,\P}_{PSPG}(\cdot,\cdot)$ is defined by
\begin{align*}
  F^{\psi,\P}_{\PSPG{}}(\vv,\qs)
  := \int_{\P} \fv\cdot\vv \dE
  + \int_{\P} \tau_{\P}\fv\cdot\nabla\qs\,\dE
  + \int_{\P} \big(\gs-\kappa(\psi)\big)\Ev\cdot\vv\,\dE
  + \int_{\P} \tau_{\P}\big(\gs-\kappa(\psi)\big)\Ev\cdot\nabla\qs\,\dE.
\end{align*}
Now, the global form of the trilinear form
$\calAP_{PSPG}\big(\cdot; \cdot,\cdot\big)$ and load term
$F_{PSPG}^{\psi,\P}(\cdot, \cdot)$ are given as follows:
\begin{align*}
  \calAPSPG\big(\psi;(\uv,\ps),(\vv,\qs)\big)
  &:= \sum_{\P\in\Th} \calAP_{\PSPG{}}\big(\psi;(\uv,\ps),(\vv,\qs)\big),\\
  \Fss{\PSPG{}}^{\psi}(\vv,\qs)
  &:= \sum_{\P\in\Th}\Fss{\PSPG{}}^{\psi,\P}(\vv,\qs).	
\end{align*}
Thus, the stabilized problem read as:
\begin{align}
  \begin{cases}
    \text{Find } (\uv,\ps,\psi)\in\Vv\times\Qs\times\mathlarger\Phi
    \text{~such that} \\[1ex]
    \begin{aligned}
      \calA_{\PSPG{}}\big(\psi;(\uv,\ps),(\vv,\qs)\big)
      &= \Fs^\psi_{\PSPG{}}(\vv,\qs) 
      && \quad\text{for~all~} (\vv,\qs)\in\Vv\times\Qs,\\[0.5em]
      \calB( \uv;\psi,\phi)
      &= ( \gs, \phi) 
      && \quad \text{for~all~} \phi\in\mathlarger\Phi_0.
    \end{aligned}
  \end{cases}
  \label{svariation}
\end{align}

\begin{remark}
  We remark that the stabilization of the momentum equation is
  considered in this work.
  The stabilization of the transport potential equation is not
  necessary for the diffusion-dominated regimes.
  We emphasize that the residual-based pressure stabilization
  technique is obtained by replacing the Laplacian term in the
  right-hand side of the momentum equation with the transport
  potential equation to avoid the higher-order derivative of the
  potential field.
  It provides a more convenient form of the stabilized problem and
  simplifies the implementation methodology as well as the theoretical
  analysis.
\end{remark}

%% Section 3.2
\subsection{Projectors and virtual element spaces}
\label{subsec-32:projectors}
{Let $\P$ be a polygon} in $\REAL^2$.
For $k\in\INTG\cup\{0\}$, we introduce the following
definitions/notations:
\begin{itemize}
\item The set of monomials $\calMk(\P)$ of degree $\leq\ks$:
  \begin{equation}
    \calMk(\P):=
    \Big\{ \mss{\beta} \,|\, \mss{\beta}
    = \Big(\dfrac{\xv-\xvP}{\hP}\Big)^{\boldsymbol{\beta}}
    \,\, \text{with}      \,\,     \boldsymbol{\beta}\in\INTG^2
    \,\, \text{such that} \,\,\abs{\boldsymbol{\beta}}\leq\ks
    \Big \},
  \end{equation}
  where $\boldsymbol{\beta}=(\beta_1, \beta_2)$, such that
  $\xv^{\boldsymbol{\beta}}:= \xss{1}^{\beta_1}\xss{2}^{\beta_2}$ with
  $|\boldsymbol{\beta}|:= \beta_1 + \beta_2$.
\item $\calMk^\ast(\P)$: Restriction of the monomials
  $\calMk(\P)$ of degree equal to $k$, defined by
  \begin{equation}
    \calMk^\ast(\P)
    := \Big\{
    \mss{\beta}\, |\, \mss{\beta}
    =  \Big( \dfrac{\xv-\xvP}{\hP} \Big)^{\boldsymbol{\beta}},\,\,
    \boldsymbol{\beta}\in\INTG^2\,\,
    \text{such~that}\,\, \abs{\boldsymbol{\beta}} = \ks
    \Big \}.
  \end{equation}
\item $\PS{k}(\P)$: Polynomial space of degree $\leq\ks$ on $\P$, and
  $\PS{-1}=\{ 0 \}$.
\item The broken Sobolev and polynomial spaces:
  \begin{align*}
    \WS{s}{p}(\Th) :&= \Big\{ \vs\in\LTWO(\Omega) \,\, \text{such that}\,\, \restrict{\vs}{\P} \in \WS{s}{p}(\P) \,\, \text{for all }\,\, \P\in\Th \Big\} \qquad \ss\in\REAL^{+}, \\
    \PS{k}   (\Th) :&= \Big\{ \ps\in\LTWO(\Omega) \,\, \text{such that}\,\, \restrict{\ps}{\P} \in \PS{k}   (\P) \,\, \text{for all }\,\, \P\in\Th \Big\}.
  \end{align*}
  %% \MGT{
  We recall that $\WS{s}{p}(\P)$ denotes the Sobolev-Slobodeckij space
  of functions on $\P$ with smoothness index $s\in\REAL^+$ and
  integrability exponent $p\in[1,\infty]$.% which coincides with the
  %standard Sobolev space $\HS{k}(\P)$ when $s = k$ is a non-negative
  %integer and $p = 2$.
  %% }
\end{itemize}
Let us assume $\P\in\Th$, the following polynomial projection
operators are introduced below:
\begin{itemize}
\item The $\Hv^1$\textbf{-energy projection} operator
  $\PinP{k}:\HONE(\P)\rightarrow\PS{k}(\P)$ is defined as follows:
  \begin{align}
    \begin{cases}
      \displaystyle
      \int_{\P} \nabla (\phi - \PinP{k} \phi)\cdot\nabla\mss{k} \, \dE = 0
      \qquad\,
      \text{for~all}\,\,\phi\in\HONE(\P)\,\,
      \text{and}\,\,\mss{k}\in\PS{k}(\P),\\[1em]
      \displaystyle
      \int_{\partial \P} (\phi - \PinP{k}\phi) \,\ds=0.
    \end{cases}
    \label{H1proj}
  \end{align}
  Furthermore, its extension to vector-valued function is defined by
  $\PvnP{k}:\big[\HONE(\P)\big]^2\rightarrow\big[\PS{k}(\P)\big]^2$
  such that
  \begin{align}
    \begin{cases}
      \displaystyle
      \int_{\P} \nabla(\phiv - \PinP{k} \phiv):\nabla \mv_k\,\dE = 0
      \qquad\,
      \text{for~all}\,\,\phiv\in\big[\HONE(\P)\big]^2\,\,
      \text{and}     \,\,\mv_k\in\big[ \PS{k}(\P) \big]^2,\\[1em]
      \displaystyle
      \int_{\partial\P} (\phiv - \PinP{k} \phiv)\,\ds = \mathbf{0}.
    \end{cases}
    \label{H1proj2}
  \end{align}
\item The $\mathbf{\LTWO}$\textbf{-projection} operator $\PizP{k}:\LTWO(\P)\rightarrow\PS{k}(\P)$ is given by
  \begin{align}
    \int_{\P} (\phi - \PizP{k} \phi) m_k \, \dE = 0
    \qquad\,
    \text{for~all}\,\, \phi\in\LTWO(\P)\,\,
    \text{and}   \,\, \mss{k}\in\PS{k}(\P).
    \label{l2proj}
  \end{align}
  Furthermore, its extension to vector-valued function is defined by
  $\PvzP{k}:\big[\LTWO(\P)\big]^2\rightarrow\big[\PS{k}(\P)\big]^2$
  such that
  \begin{align}
    \int_{\P} (\phiv - \PizP{k}\phiv) \cdot \mv_k \, \dE = 0
    \qquad \,
    \text{for all}\,\,\phiv\in\big[\LTWO(\P)\big]^2\,\,
    \text{and}    \,\,\mv_k\in\big[\PS{k}(\P)\big]^2.
    \label{l2proj2}
  \end{align}
\end{itemize} 
\begin{remark} \label{qstab}
  \textbf{Stability of} {$\PizP{k}$} (\cite{mishra2025equal}): The
  $\LTWO$-projection operator satisfies the following stability
  property w.r.t. $\LS{q}$-norm
  \begin{align}
    \NORM{\PizP{k}\phi}{\LS{q}(\P)} \leq
    \Cs\NORM{\phi}{\LS{q}(\P)}
    \qquad \,\,
    \text{for~all~}\,\, \phi\in\HONE(\Omega)\,\,
    \text{and}\,\, \qs\geq 2.
  \end{align}
\end{remark}
We are now in a position to introduce the $\HONE$-conforming virtual
element space $\Vshk(\P)$ of order $k\in\INTG$.
Following \cite{vem22,vem25}, we introduce
%%
%% \begin{align}
%%   V_h(\P):= \Big\{
%%   \phi \in H^1(\P) &\cap C^0( \partial{E}) \,\, \text{such that}, \nonumber \\ &(i) \,\,\Delta \phi \in \mathbb P_k(\P), \nonumber \\ &(ii)\, \,  \phi_{|e} \in \mathbb P_k(e) \,\, \forall \,\, e \in \partial\P,  \nonumber \\  &(iii) \,\, \big(\phi-\Pi_k^{\nabla,E}\phi, m_\alpha \big)_{E}=0, \, \forall \, m_\alpha \in \mathcal{M}^\ast_{k-1}(\P) \cup \mathcal{M}^\ast_{k}(\P) \Big\}. \nonumber
%% \end{align}
%%
%% \MGT{
\begin{align}
  \Vshk(\P):= \Big\{
  \phi \in \HONE(\P)
  \,\,\text{such that}\,\,
  &(i)  \,\,\restrict{\phi}{\partial\P}\in\CS{0}(\partial\P),\nonumber\\
  &(ii) \,\,\Delta\phi\in\PS{k}(\P),\nonumber\\
  &(iii)\,\,\restrict{\phi}{\E}\in\PS{k}(\E)\,\,\forall\E\in\partial\P,\nonumber \\
  &(iv) \,\,\big(\phi-\PinP{k}\phi,\mss{\alpha} \big)_{E}=0,
  \,\forall \mss{\alpha}\in\mathcal{M}^\ast_{k-1}(\P)\cup\calMk^\ast(\P)
  \Big\}.
  \nonumber
\end{align}
%% }

Additionally, the local virtual element space is equipped with the
following {degrees of freedom (\textbf{DoF})}:
$\DOFS{DF1}(\phi)$,
$\DOFS{DF2}(\phi)$, and
$\DOFS{DF3}(\phi)$:
\begin{itemize}
\item $\DOFS{DF1}(\phi):$ the values of $\phi$ evaluated at each
  vertices of $\P$;
\item $\DOFS{DF2}(\phi):$ the values of $\phi$ evaluated at $(k-1)$
  internal Gauss-Lobatto quadrature nodes on each edge
  $\E\in\partial\P$;
\item $\DOFS{DF3}(\phi):$ the internal moments of $\phi$ up to
  order $k-2$,
  \begin{align*}
    \frac{1}{\mP} \, \int_{\P} \phi\mss{\alpha}\,\dE
    \qquad\,\text{for~all}\,\,\mss{\alpha}\in\calM_{k-2}(\P).
  \end{align*}	
\end{itemize}
Finally, combining the local virtual space {$V^h_k(\P)$} for all $E \in
\Th$, the global virtual element space {$V^h_k$} is presented as
follows:
\begin{align*}
  \Vshk :=
  \Big\{
  \phi\in\HONE(\Omega)\,\,
  \text{such~that}\,\,\restrict{\phi}{\P}\in\Vshk(\P)\,\,
  \text{for~all}\,\,\P\in\Th
  \Big\}. 
\end{align*}
After this, we define the virtual element approximations of the
functional spaces introduced for the velocity vector field and
pressure field, as follows:
\begin{align*}
  \Vvh{} &:= \Big\{ \phiv\in\Vv \,\, \text{such that}\,\, \restrict{\phiv}{\P} \in\big[\Vshk(\P)\big]^2 \,\, \text{for all}\,\, \P\in\Th \Big\}, \\
  \Qsh   &:= \Big\{ \qs  \in\Qs \,\, \text{such that}\,\, \restrict{\qs  }{\P} \in\Vshk(\P)             \,\, \text{for all}\,\, \P\in\Th \Big\}. 
\end{align*}
Additionally, the virtual element approximation of the test and trial
spaces for the potential field is defined as follows:
\begin{align*}
  \mathlarger \Phi_{0,\hh} &:= \Big\{ \phi \in \mathlarger\Phi_0 \,\, \text{such that}\,\, \restrict{\phi}{\P} \in \Vshk(\P) \,\, \text{for all}\,\, \P\in\Th \Big\}, \\
  \mathlarger \Phi_{\hh}  &:= \Big\{ \phi \in \mathlarger\Phi   \,\, \text{such that}\,\, \restrict{\phi}{\P} \in \Vshk(\P) \,\, \text{for all}\,\, \P\in\Th \Big\}. 
\end{align*} 

%% Section 3.3
\subsection{Discrete forms and stabilized virtual element problem}
\label{subsec-33:discrete:forms}
In this section, we introduce the virtual element approximation of the
stabilized problem \eqref{svariation}, which is computable using {the
\DOFS{DoF}{}}.
We begin by introducing the discrete counterparts of the
bilinear/trilinear forms and load terms.

\medskip\noindent
$\bullet$ For any $\phi_h\in\mathlarger\Phi_h$, $\vvh$,
$\wvh\in\Vvh{}$, and $\qsh\in\Qsh$, we introduce the following local
discrete counterparts associated to the first equation of
\eqref{svariation}:
\begin{align*}
  \asPVh(\vvh, \wvh) &:= \int_{\P} \mu\PvzP{k-1}\nabla\vvh:\PvzP{k-1}\nabla\wvh\,\dE + \mu \SP_{1}\big( \vvh - \PvnP{k}\vvh, \wvh - \PvnP{k}\wvh \big), \\  
  \bsPh (\vvh, \qsh) &:= \int_{\P} \big(\PizP{k-1}\nabla\cdot\vvh\big)\,\PizP{k}\qsh\,\dE, \\
  {\csPh(\phi_h; \vvh,\wvh)}    &:= \int_{\P} \big( \PvzP{k}\vvh\cdot\PvzP{k-1}\nabla\phi_h \big) \Ev\cdot\PvzP{k}\wvh\,\dE,\\
  \calLP{1,h}(\psh,\qsh)       &:= { \int_{\P} \tau_{\P} \PvzP{k-1}\nabla\psh\cdot\PvzP{k-1}\nabla\qsh\,\dE + \tau_{\P} \SP_2\big( \psh - \PinP{k-1}\psh, \qsh - \PinP{k-1}\qsh\big)}, \\
  \calLP{2,h}(\phi_h,\uvh,\qsh)&:= \int_{\P} \tau_{\P} \big( -\nabla\cdot\PvzP{k-1}\nabla\uvh + ( \PvzP{k}\uvh\cdot\PvzP{k-1}\nabla\phi_h)\Ev \big)\cdot \PvzP{k-1}\nabla\qsh\,\dE,    
  %	\calLP{3,h}(\vv_h, \wvh) :&=  \int_{\P} \delta_E {\Pi}^{0,E}_{k-1} \nabla \cdot \vv_h \cdot \Pi^{0,E}_{k-1} \nabla \cdot \wvh\, \dE + \delta_E S^E_1 \big(\vv_h-\boldsymbol{\Pi}^{\nabla,E}_k \mathbf{v}_h, \wvh-\boldsymbol{\Pi}^{\nabla,E}_k \wvh \big).
\end{align*}
where
$\SP_1(\cdot,\cdot):\Vvh{}(\P)\times\Vvh{}(\P)\rightarrow\REAL$ and
$\SP_2(\cdot,\cdot):\Qsh(\P)\times\Qsh(\P)\rightarrow\REAL$ are the
VEM computable symmetric positive definite bilinear forms such that
they satisfy
\begin{align}
  \lambda_{1\ast} \ \NORM{\nabla\wvh}{0,\P}^2
  &
  \leq \SP_1(\wvh,\wvh)
  \leq \lambda_1^\ast \NORM{\nabla\wvh}{0,\P}^2
  \qquad
  \text{for~all} \,\,\wvh\in\Vvh{}(\P)\cap\text{ker}\big(\PvnP{k}\big),
  \label{vem-a} \\
  \lambda_{2\ast} \NORM{\nabla\qsh}{0,\P}^2
  &
  \leq \SP_2(\qsh,\qsh)
  \leq \lambda_{2}^\ast \NORM{\nabla\qsh}{0,\P}^2
  \qquad
  \text{for~all} \,\,\qsh\in\Qsh(\P)\cap\text{ker}\big(\PinP{k-1}\big),
  \label{vem-b}
\end{align}
with the positive constants
$\lambda_{1\ast}\leq\lambda_1^\ast$ and
$\lambda_{2\ast}\leq\lambda_2^\ast$
are independent of the mesh size.
Additionally, we introduce the grad-div stabilization terms to address
the violation of divergence-free constraints, as follows
\begin{align*}
  \calLP{3,h}(\vvh,\wvh) &:=
  \int_{\P} \delta_E \left(\PizP{k-1}\nabla\cdot\vvh\right)\,\left(\PizP{k-1}\nabla\cdot\wvh\right)\,\dE
  + \delta_E \SP_1 \big( \vvh - \PvnP{k}\vvh, \wvh - \PvnP{k}\wvh \big),
\end{align*}
{where $\delta_E>0$ denotes the stabilization parameter depending on the mesh size. Throughout this work, following~\cite{mishra2025equal}, we set $\delta_E \sim h_E$ for all $E \in \Omega_h$.}

\medskip
In view of above, the local VEM stabilized bilinear form
$\calAP_{\PSPG{},\hh}(\phi_h;\cdot,\cdot)$ is defined by
\begin{align}
  \calAP_{\PSPG{},\hh}\big(\phi_h;(\uvh,\psh), (\vvh,\qsh) \big)
  &:= \asPVh(\uvh,\vvh) - \bsPh(\vvh,\psh) + \csPh(\phi_h;\uvh,\vvh) + \bsPh(\uvh,\qsh) \nonumber \\
  & \qquad \qquad + \calLP{1,\hh}(\psh,\qsh) + \calLP{2,\hh}(\phi_h;\uvh,\qsh) + \calLP{3,\hh}(\uvh,\vvh),
\end{align}
for all $\phi_h\in\mathlarger \Phi_h(\P)$, and $(\uvh,\psh)$,
$(\vvh,\qsh)\in\Vvh{}(\P)\times\Qsh(\P)$.
Consequently, the global trilinear from
$\calA_{\PSPG{},\hh}(\phi_h;\cdot,\cdot)$ can be introduced by
combining the local contributions for all $\P\in\Th$, as follows
\begin{align*}
  \calA_{\PSPG{},\hh}\big( \phi_h;(\uvh,\psh), (\vvh,\qsh) \big)
  &:= \sum_{\P\in\Th} \calAP_{\PSPG{},\hh}\big( \phi_h;(\uvh,\psh), (\vvh,\qsh) \big),
\end{align*}
for all $\phi_h\in\mathlarger\Phi_h$, and $(\uvh,\psh)$,
{$(\vvh,\qsh)\in \Vvh{} \times \Qsh$}.
We now define the local virtual element approximation of the load term
$\FsP_{\PSPG{}}(\cdot)$.
For given $\phi_h\in\mathlarger\Phi_h(\P)$, the local approximation
$\Fs^{\phi_h,\P}_{\PSPG{},\hh}(\cdot)$ is defined by
\begin{align*}
  \Fs^{\phi_h,\P}_{\PSPG{},\hh}(\vvh,\qsh)
  &:=
  \int_{\P} \PvzP{k}\fv\cdot\PvzP{k}\vvh\,\dE +
  \int_{\P} \tau_{\P} \PvzP{k}\fv\cdot\PvzP{k-1}\nabla\qsh\,\dE +
  \int_{\P} \big( \gs - \kappa(\PizP{k}\phi_h)\big)\Ev\cdot\PvzP{k}\vvh\,\dE \nonumber \\
  & \qquad \qquad +
  \int_{\P} \tau_{\P} \big( \gs-\kappa(\PizP{k}\phi_h)\big)\Ev \cdot \PvzP{k-1}\nabla\qsh\,\dE.
\end{align*}

Therefore, the global discrete load term is given by
\begin{align*}
  \Fs^{\phi_h}_{\PSPG{},\hh}(\vvh,\qsh)
  := \sum_{\P\in\Th } \Fs^{\phi_h,\P}_{\PSPG{},\hh}(\vvh,\qsh)
  \qquad \text{for~all}\,\,
  (\vvh,\qsh)\in\Vvh{}\times\Qsh.
\end{align*}

\medskip\noindent
$\bullet$ The local discrete bilinear/trilinear form corresponding to
the potential transport equation is given as follows:
for any $\vvh\in\Vvh{}(\P)$ and
{$\phi_h, \xi_h\in\mathlarger\Phi_{0,\hh}(\P)$}
\begin{align*}
  \asPph(\phi_h,\xi_h)
  &:= \int_{\P}
  \epsilon \PvzP{k-1}\nabla\phi_h\cdot\PvzP{k-1}\nabla\xi_h\,\dE +
  \epsilon \SP_3 \big( \phi_h - \PinP{k}\phi_h, \xi_h - \PinP{k} \xi_h \big),\\[1em]
  \csPph(\vvh; \phi_h,\xi_h)
  &:= \int_{\P}
  \big( \PvzP{k}\vvh\cdot\PvzP{k-1}\nabla\phi_h \big)\,\PizP{k}\xi_h\,\dE,\\[1em]
  \cs^{skew,\P}_{\ps,\hh}(\vvh; \phi_h,\xi_h)
  &:= \dfrac{1}{2} \big( \csPph(\vvh;\phi_h,\xi_h) - \csPph(\vvh;\xi_h\phi_h) \big),\\[1em]
  \dsPh(\phi_h,\xi_h) &:= \int_{\P} \kappa(\PizP{k}\phi_h)\PizP{k}\xi_h\,\dE,
\end{align*}
where
$\SP_3(\cdot,\cdot):\mathlarger\Phi_{0,\hh}(\P)\times\mathlarger\Phi_{0,h}(\P)\rightarrow\REAL$
is a symmetric positive definite bilinear form that satisfies
\begin{align}
  \lambda_{3\ast} \NORM{\nabla\phi_h}{0,\P}^2
  &
  \leq \SP_3(\phi_h,\phi_h)
  \leq \lambda_3^\ast \NORM{\nabla\phi_h}{0,\P}^2\qquad
  \text{for~all}\,\,\phi_h\in\mathlarger\Phi_{0,\hh}(\P)\cap \text{ker}(\PvnP{k}),
  \label{vem-c} 
\end{align}
with the positive constants $\lambda_{3\ast} \leq \lambda_3^\ast$ does
not rely on the mesh size.

\medskip
The discrete global form corresponding to
$\calB(\cdot;\cdot,\cdot)$ is given by
\begin{align*}
  \calBh\big(\vvh;\phi_h,\xi_h\big)
  := \sum_{\P\in\Th}\calBPh\big(\vvh;\phi_h,\xi_h\big),
\end{align*}
where
$\calBPh\big(\vvh;\phi_h,\xi_h\big):=\asPph(\phi_h,\xi_h)+\cs^{skew,\P}_{\ps,\hh}(\vvh;\phi_h,\xi_h)+\dsPh(\phi_h,\xi_h)$.
Additionally, the discrete potential load is given by
\begin{align*}
  (\gsh,\xi_h) := \sum_{\P\in\Th} \int_{\P} \gs\PizP{k}\xi_h\,\dE
  \qquad \text{for~all}\,\,\xi\in\mathlarger\Phi_{0,\hh}.
\end{align*}

\medskip\noindent
$\bullet$ Combining all the above, the stabilized virtual element
problem is stated as follows:
%% \begin{align}
%%   \begin{cases}
%%     \text{Find~} (\uvh,\psh,\psi_h)\in\Vvh{}\times\Qsh\times\mathlarger\Phi_h \text{~such~that} \\[1ex]
%%     \begin{aligned}
%%       \calA_{\PSPG{},\hh\big(\psi_h;(\uvh,\psh),(\vvh,\qsh)\big) &= \Fs^{\psi_h}_{\PSPG{},\hh}(\vvh,\qsh) && \quad \text{for all } \,\, (\vvh,\qsh)\in\Vvh{}\times\Qsh, \\ 
%%       \calBh( \uvh; \psi_h, \phi_h) &= (\gsh,\phi_h) && \quad \text{for~all~} \phi_h\in\mathlarger\Phi_{0,\hh}.
%%     \end{aligned}
%%   \end{cases}
%%   \label{svem}
%% \end{align}

\begin{equation}
\begin{cases}
\text{Find } (\uvh,\psh,\psi_h)\in\Vvh{}\times\Qsh\times\mathlarger\Phi_h
\text{ such that} \\[1ex]
\begin{aligned}
\calA_{\PSPG{},\hh}\big( \psi_h;(\uvh,\psh),(\vvh,\qsh) \big)
&= \Fs^{\psi_h}_{\PSPG{},\hh}(\vvh,\qsh)
&& \text{for all } (\vvh,\qsh)\in\Vvh{}\times\Qsh, \\ 
\calBh( \uvh; \psi_h, \phi_h)
&= (\gsh,\phi_h)
&& \text{for all } \phi_h\in\mathlarger\Phi_{0,\hh}.
\end{aligned}
\end{cases}
\label{svem}
\end{equation}

Obviously, the discrete problem \eqref{svem} can be decoupled as
follows:
\begin{itemize}
\item With given $\psi_h\in\mathlarger\Phi_h$, find
  $(\uvh,\psh)\in\Vvh{}\times\Qsh$, such that it satisfies
  \begin{align}
    {\calA_{\PSPG{},\hh} \big(\psi_h;(\uvh,\psh),(\vvh,\qsh) \big)}
    &= \Fs^{\psi_h}_{\PSPG{},\hh}(\vvh,\qsh) 
    \qquad \text{for~all~} \,\,(\vvh,\qsh)\in\Vvh{}\times\Qsh,
    \label{dsvem1}
  \end{align}
\item With given $\uvh\in\Vvh{}$,
  find $\psi_h\in\mathlarger\Phi_h$, such that
  \begin{align}
    \calBh(\uvh;\psi_h,\phi_h) &= (\gsh,\phi_h) 
    \qquad \text{for~all~} \phi_h\in\mathlarger\Phi_{0,\hh}.
    \label{dsvem2}
  \end{align} 
\end{itemize}

% Local Variables:
% mode: latex
% End:

%% Hey Emacs, this is -*-latex-*-

%%%%%%%%%%%%% stability%%%%%%%%%%%%%%%%%%%%%%%%%%%%%%%%%%%%%%%%%%%%%%%%%%%%%%%%%

\section{Theoretical analysis}
\label{sec-4}

We begin the theoretical analysis by first establishing the
well-posedness of the discrete decoupled and coupled problems,
together with a priori error estimates in the energy norm.
%%
%To this end, we need to introduce some fundamental properties of
%Sobolev spaces.
%%
%Let $C_P$ be the Poincar\'{e} constant.
%%
%Additionally, we introduce the constant $C_q>0$, depending on
%$\Omega$, such that
%%
%\begin{align*}
 % \| \phi \|_{1,\Omega}
 % \leq C_q \|\nabla \phi\|_\Omega
 % \qquad \text{for all}\quad\phi \in W^1_2(\Omega),
%\end{align*}
%%
%where $W^k_p({\Omega})$ denotes the Sobolev space, which reduces to
%$H^k({\Omega})$ when $p=2$.
%%
%Furthermore, we introduce the compact embedding constant
%$C_{1\hookrightarrow p}$, corresponding to the compact embedding
%$H^1(\Omega) \hookrightarrow L^p(\Omega)$ for $p \geq 1$, satisfying
%%
%\begin{align*}
% \| \phi \|_{0,p,\Omega}
%  \leq C_{1 \hookrightarrow p} \| \nabla \phi\|_\Omega
%  \qquad \text{for all} \quad \phi \in H^1(\Omega).
% \end{align*}
%% Hereafter, the next step of the proposed method is establishing the
%% well-posedness of the discrete decoupled and coupled problems and a
%% priori error estimates in the energy norm. We begin by introducing the
%% fundamental inequalities and approximation estimates which are valid
%% on the mesh regularity assumption \textbf{(A1)}.
{In this sequel, we first consider the fundamental inequalities} and approximation
estimates of Lemmas~\ref{lemmaproj1}, \ref{lemmaproj2}, \ref{inverse},
and~\ref{estimate}, which hold under the mesh regularity assumption
\textbf{(A1)}.

%\subsection{Preliminary results} 
\begin{lemma}[Polynomial approximation \cite{scott}]
  \label{lemmaproj1}
  Under the assumption \textbf{(A1)}, for any functions
  $\phi\in\HS{s}(\P)$ with $\P\in\Th$, the elliptic projector
  $\PinP{k}$ and the $\LTWO$-orthogonal projector $\PizP{k}$
  satisfy the following error estimates
  \begin{align}
    \NORM{ \phi - \PinP{k} \phi }{l,\P} &\leq \Cs \hh^{s-l}_{\P} \snorm{\phi}{s,\P} \quad s, l \in \INTG, \,\,\, s \geq 1,\,\, \, l \leq s \leq k+1,\label{eqproj1}	 \\
    \NORM{ \phi - \PizP{k} \phi }{l,\P} &\leq \Cs \hh^{s-l}_{\P} \snorm{\phi}{s,\P} \quad s, l \in \INTG, \,\,\, l \leq s \leq k+1.
  \end{align}
\end{lemma}

\begin{lemma}[Interpolant approximation \cite{vem45}]
  \label{lemmaproj2}
  Under the assumption \textbf{(A1)}, for any $\phi\in\HS{s+1}(\P)$
  there exists {$\phiI\in V^h_k(E)$} for all $E\in\Th$ such that
  there holds the following
  \begin{equation}
    \NORM{\phi - \phiI}{0,\P} + \hP\snorm{\phi - \phiI}{1,\P}
    \leq \Cclem{} \hP^{1+s}\,\NORM{\phi}{s+1,\P},
    \quad 0\leq s\leq k,
    \label{projlll2}
  \end{equation}
  where the constant $\Cclem{}>0$ depends only on $k$ and $\delta_0$. 
\end{lemma}

\begin{lemma}[Inverse inequality \cite{vem28}]
  \label{inverse}
  Under the assumption \textbf{(A1)}, for any virtual element function
  {$\phih\in V^h_k(E)$} defined on $\P\in\Th$, it holds that
  \begin{align}
    \snorm{\phih}{1,\P} \leq \Css{inv}\hP^{-1}\NORM{\phih}{0,\P},
  \end{align}
  where the positive constant $\Css{inv}$ is independent of the mesh
  size.
\end{lemma}
 \medskip\noindent { We also stress that the Lemma~\ref{lemmaproj1}, Lemma~\ref{lemmaproj2}, and Lemma~\ref{inverse} are also valid for vector-valued functions belonging to $\Vv^h$}.
Additionally, following \cite[cf. Eq. (96)]{mishra2025supg}, we have
\begin{align}
  \NORM{\PvzP{k-1}\nabla\vvh}{0,4,\P}
  \leq \Css{inv} \hP^{-1} \NORM{\vvh}{0,4,\P}
  \qquad \text{ for all} \,\, \vvh\in \Vvh{}(\P).
  \label{bound1}
\end{align}
We now define the mesh-dependent energy norms over $\Vvh{}\times\Qsh$
as follows:
\begin{align*}
  \normiii{(\vvh, \qsh)}^2
  &= \NORM{\nabla\vvh}{0,\Omega}^2 + \theta \NORM{\qsh}{0,\Omega}^2
  + \calLss{1,h}(\qsh,\qsh) + \calLss{3,h}(\vvh,\vvh),
  \qquad\text{for~all}\,\,
  (\vvh,\qsh)\in\Vvh{}\times\Qsh,
  \nonumber
  %	\normiii{\phi}^2 & =  \|\nabla \phih\|^2_{0, \Omega}  \nonumber
\end{align*}
where the positive constant $\theta$ does not rely on the mesh size,
given by
\begin{align*}
  \theta =
  \frac{2\beta_0}{ (3\mu\alpha^\ast + \alpha^\ast + 3\lambda_2^\ast)\widehat{\Cs}_{\text{clem}} }.
\end{align*}
Concerning the discrete spaces ${\mathlarger \Phi_h}$ and
$\mathlarger\Phi_{0,h}$, we utilize the usual $\HONE$-semi norm for
the subsequent analysis.
The stability properties of the bilinear forms $\asVh(\cdot,\cdot)$
and $\asph(\cdot,\cdot)$ is discussed in the following Lemma:

\begin{lemma} \label{estimate}
  For all $\vvh,\wvh\in\Vvh{} $ and $\phih,\xi_h \in {\mathlarger \Phi_h}$,
  there holds the following:
  \begin{align}
    \abs{\asVh(\vvh, \wvh)}   &\leq \alpha^\ast \mu \NORM{\nabla \vvh}{\Omega}\,\NORM{\nabla \wvh}{\Omega},
    \label{est1} \\
    \abs{\asVh(\vvh, \vvh)}   &\geq \alpha_\ast \mu \NORM{\nabla \vvh}{\Omega}^2, 
    \label{est2} \\
    \abs{\asph(\phih, \xi_h)} &\leq \gamma^\ast \epsilon \NORM{\nabla \phih}{\Omega} \NORM{\nabla\xi_h}{\Omega},
    \label{est3} \\
    \abs{\asph(\xi_h, \xi_h)} &\geq \gamma_\ast \epsilon\NORM{\nabla \xi_h}{\Omega}^2,
    \label{est4}
  \end{align}
  where
  $\alpha^\ast:= \max\{1,\lambda_1^\ast\}$,
  $\alpha_\ast:= \min\{1,\lambda_{1\ast}\}$, and
  $\gamma^\ast:= \max\{1,\lambda_3^\ast \}$,
  $\gamma_\ast:= \min\{1,\lambda_{3\ast}\}$.
\end{lemma}
\begin{proof}
  The proof of Lemma \ref{estimate} follows from
  \cite[Lemma~5.3]{vem28m} with slight modification.
  Therefore, we skip the proof.
\end{proof}

\subsection{Stability of discrete decoupled problems}

We now present the well-posedness of the discrete decoupled problem
\eqref{dsvem1} by establishing the generalized discrete inf-sup
condition for the velocity and pressure fields and the continuity of
the bilinear form {$\calA_{\PSPG{},\hh}\big(\psi_h; \cdot,\cdot\big)$ with given $\psih\in \mathlarger \Phi_h$}.
Concerning this, we present the following:
\begin{theorem} \label{dwell1}
  Let $\psih\in{\mathlarger \Phi_h}$ and the stabilization parameters
  $\tau_{\P}$ and $\delta_{\P}$ satisfy the following assumptions:
  \begin{align}
    \NORM{\nabla\psih}{\Omega}
    &\leq \frac{ \alpha_\ast\mu}{2\P^\ast \Css{1 \hookrightarrow4}^2},
    \label{stabc1} \\
    %% --------------------------------------------------------------
    \frac{\hP^2}{\lambda_2^{\ast^2}}\leq \tau_{\P}
    &\leq \frac{1}{16} \min\Big\{\frac{\alpha^2_\ast\hP^2}{\Css{inv}^2};
    \frac{ \hP^2}{ \Css{inv}^2 } \Big\} \qquad \,\,
    \text{and} \qquad  \max_{\P\in\Th} \delta_{\P} \leq 1\qquad
    \text{for~all} \,\, \P\in\Th.
    \label{stabc2} 
    %	\,\, \text{with} \,\, \rho_\ast:= \min \Big\{\frac{\mu \alpha_\ast}{2}, 1 \Big\}. 
  \end{align}
  %%
  %% \MGT{
  Also, let the physical and stabilization parameters satisfy
  $\mu\alpha_\ast\leq2$ for small enough values of $\alpha_\ast$.
  %% }
  %%
  Then, for any $(\uvh,\psh)\in\Vvh{}\times\Qsh$, there exists at
  least one $(\vvh,\qsh)\in\Vvh{} \times \Qsh$ such that the following
  holds
  \begin{align}
    \sup\limits_{(\mathbf{0},0)\neq (\vvh,\qsh)\in\Vvh{}\times\Qsh}
    \dfrac{\calA_{\PSPG{},\hh} \big(\psih;(\uvh,\psh),(\vvh,\qsh)\big)}{\normiii{(\vvh,\qsh)}}
    \geq \beta_\ast\normiii{(\uvh,\psh)},
    \label{d1well1}
    %	|A_{V,h} [ \widehat{\theta}_h; (\uvh, \psh), (\vvh, \qsh)]| &\leq C\normiii{(\uvh, \psh)}\, \normiii{(\vvh, \qsh)},   \label{d1well2}
  \end{align}
  where the positive constant $\beta_\ast$ does not rely on the mesh
  size and depends on the data of {Problem~\eqref{P}}.
  Furthermore, {Problem~\eqref{dsvem1}} has a unique solution such
  that
  \begin{align}
    \normiii{(\uvh, \psh)} \leq \frac{1}{\beta_\ast}
    \Big( \Css{P} + \frac{\alpha_\ast\abs{\Omega}^{1/2}}{4 \Css{inv}}\Big)
    \Big(
    \NORM{\fv}{\Omega} + \East \NORM{\gs}{\Omega} +
    \frac{\alpha_\ast \Css{P}\kappa^\ast \mu}{2\Css{1 \hookrightarrow 4}^2}
    \Big).
    \label{dstabv}
  \end{align}
\end{theorem}
\begin{proof}
  We first introduce the following settings:
  \begin{align*}
    \normiii{\vvh}^2_\ast
    &= \NORM{\nabla\vvh}{0,\Omega}^2
    + \calLss{1,h}(\qsh,\qsh)
    + \calLss{3,h}(\vvh,\vvh)
    \qquad \text{for all} \,\, (\vvh,\qsh)\in\Vvh{}\times\Qsh,
    \nonumber \\
    {Y} & = {\theta \NORM{\psh}{\Omega}}.
    \nonumber
    %	\normiii{\phi}^2 & =  \|\nabla \phih\|^2_{0, \Omega}  \nonumber
  \end{align*}
  The proof of Theorem \ref{dwell1} follows in the
  following steps:

  \medskip\noindent
  \textbf{Step 1.} {In this step, we seek a lower bound for $\calA_{\PSPG{},\hh} \big( \psih; (\uvh, \psh), (\uvh, \psh) \big)$, for any $(\uvh, \psh) \in \Vvh{}\times \Qsh$ and given $\psih \in \mathlarger \Phi_h$. In this sequel, invoking the inequality~\eqref{est2}, we have the following for given $\psih \in \mathlarger\Phi_h$}
  \begin{align}
    \calA_{\PSPG{},\hh} \big( \psih; &(\uvh, \psh), (\uvh, \psh) \big) \nonumber\\
      &= \asVh(\uvh, \uvh) + \csh(\psih;\uvh, \uvh)
      + \calLss{1,\hh}(\psh,\psh)
      + \calLss{2,\hh}(\psih;\uvh,\psh)
      + \calLss{3,\hh}(\uvh,\uvh) \nonumber\\
      &\geq  \mu \alpha_\ast \NORM{\nabla\uvh}{\Omega}^2
      + \csh(\psih;\uvh,\uvh)
      + \calLss{1,\hh}(\psh,\psh)
      + \calLss{2,\hh}(\psih;\uvh,\psh)
      + \calLss{3,\hh}(\uvh,\uvh).
      \label{vstab1}
  \end{align}
  Applying Remark \ref{qstab}, the H\"older inequality,
  and the Sobolev embedding theorem, we have
  \begin{align}
    \csh(\psih;\uvh, \uvh)
    &\leq \East \sum_{\P\in\Th}\NORM{\PvzP{k}\uvh}{0,4,\P}^2\,\NORM{\nabla\psih}{\P} \nonumber\\ 
    &\leq \East \Css{1\hookrightarrow 4}^2 \NORM{\nabla\uvh}{\Omega}^2\,\NORM{\nabla\psih}{\Omega} \nonumber\\
    &\leq \frac{\alpha_\ast\mu}{2} \NORM{\nabla\uvh}{\Omega}^2 {.}
    \qquad\qquad\qquad \text{(using~\eqref{stabc1})}\label{vstab2}
  \end{align}
  Employing the Cauchy-Schwarz inequality and Lemma \ref{inverse}, it
  gives
  \begin{align}
    \calLss{2,h}(\psih;\uvh,\psh)
    &\leq \sum_{\P\in\Th}\tau_{\P}
    \NORM{
      - \mu\nabla\cdot\PvzP{k-1}\nabla\uvh
      + \big(\PvzP{k}\uvh\cdot\PvzP{k-1}\nabla\psih\big)\Ev }{\P}\,
    \NORM{ \PvzP{k-1}\nabla\psh }{\P} \nonumber\\
    %% --------------------------------------------------------------------------------------
    &\leq \sum_{\P\in\Th}\tau_{\P}
    \big[
      \mu\Css{inv}\,\hP^{-1}\NORM{\nabla\uvh}{\P}
      + \East\NORM{\PvzP{k}\uvh}{0,4,\P}\,\NORM{\PvzP{k-1}\nabla\psih}{0,4,\P}
      \big]\,\NORM{\PvzP{k-1}\nabla\psh}{\P}{.} \nonumber
    %% --------------------------------------------------------------------------------------
    \intertext{We use the bound \eqref{bound1}, Remark \ref{qstab},
      the Sobolev embedding theorem, and \eqref{stabc1}:}
    %% --------------------------------------------------------------------------------------
    \calLss{2,h}(\psih;\uvh,\psh)
    &\leq \sum_{\P\in\Th} \tau_{\P}^{1/2}\big[
      \mu\Css{inv}\,\hP^{-1}\NORM{\nabla\uvh}{\P}
      + \East \Css{inv}\hP^{-1}\NORM{\uvh}{0,4,\P}\,\NORM{\psih}{0,4,\P}
      \big]\,\normiii{\uvh}_{\ast,\P} \nonumber\\
    %% --------------------------------------------------------------------------------------
    & \leq \max_{\P\in\Th} \frac{\tau_{\P}^{1/2}}{\hP}\big[
      \mu\Css{inv}\,\NORM{\nabla\uvh}{\Omega}
      + \East\Css{inv}\Css{1\hookrightarrow 4}^2\NORM{\nabla\uvh}{\Omega}\,\NORM{\nabla\psih}{\Omega}
      \big]\,\normiii{\uvh}_{\ast} \nonumber\\
    %% --------------------------------------------------------------------------------------
    & \leq \Css{inv} \max_{\P\in\Th}\frac{\tau_{\P}^{1/2}}{\hP}\Big(
    \mu + \frac{\alpha_\ast\mu}{2 }
    \Big)\,\normiii{\uvh}^2_\ast. \label{vstab3}
  \end{align}
  Combining \eqref{vstab1}, \eqref{vstab2}{,} and \eqref{vstab3}, we
  arrive at
  \begin{align}
    &\calA_{\PSPG{},\hh} \big( \psih; (\uvh, \psh), (\uvh, \psh) \big)\nonumber \\
    %% -------------------------------------------------------------------------
    &\qquad
    \geq \frac{\mu\alpha_\ast}{2}\NORM{\nabla\uvh}{\Omega}^2
    + \calLss{1,h}(\psh,\psh)
    + \calLss{3,h}(\uvh,\uvh)
    - \Css{inv} \max_{\P\in\Th} \frac{\tau_{\P}^{1/2}}{\hP}
    \Big(\mu + \frac{\alpha_\ast \mu}{2}\Big)\normiii{\uvh}^2_\ast\nonumber \\
    %% -------------------------------------------------------------------------
    &\qquad
    \geq \bigg( \min\big\{\frac{\mu\alpha_\ast}{2}, 1\big\}
    - \Css{inv} \max_{\P\in\Th}\frac{\tau_{\P}^{1/2}}{\hP}
    \Big(\mu + \frac{\alpha_\ast \mu}{2 }\Big) \bigg)\,\normiii{\uvh}^2_\ast\nonumber \\
    %% -------------------------------------------------------------------------
    &\qquad
    \geq \bigg(
    \frac{\mu\alpha_\ast}{2} -
    \frac{\mu\alpha_\ast}{4} -
    \frac{\mu\alpha_\ast}{8} \bigg)
    \normiii{\uvh}^2_\ast
    \qquad\qquad\text{(using \eqref{stabc2})} \nonumber \\
    &\qquad
    \geq \frac{\mu \alpha_\ast}{8} \normiii{\uvh}^2_\ast.
    \label{vstab4}
  \end{align}
  %%
  %% \MGT{
 % Note that in the second-to-last line, we use
  %condition~\eqref{coer_cond}; see Remark~\ref{rem:coercivity}.
  %% }
  { Note that in the last second line, we use
   $\min\big\{\frac{\mu\alpha_\ast}{2},1\big\}=\frac{\mu\alpha_\ast}{2}$
   for small value of $\alpha_\ast$; see Remark~\ref{rem:coercivity}.}

  \medskip\noindent
  \textbf{Step 2.} {We now test $ \calA_{\PSPG{},\hh}\big(\psih;(\uvh,\psh),(\vvh,\qsh) \big)$ with $(\vvh,\qsh)=(-\wvI,0)$. In this sequel, first employing} the continuous inf-sup condition for any
  $\psh\in\Qsh$ there exists $\wv\in\Vv$ such that
  \begin{align}
    \int_\Omega \big(\nabla\cdot\wv\big)\psh\,d\Omega
    =    \theta \NORM{\psh}{\Omega}^2 \qquad\text{and}\,\,\qquad\NORM{\nabla\wv}{\Omega}
    \leq \frac{\theta}{\beta_0} \NORM{\psh}{\Omega}.
    \label{vstab5}
  \end{align}
  Let $\wvI\in\Vvh{}$ represent the virtual interpolant of
  $\wv\in\Vv$.
  Applying Lemma \ref{lemmaproj2}, we obtain
  \begin{align}
    \NORM{\nabla\wvI}{\Omega}
    \leq \frac{\CclemHat{}\theta}{\beta_0}\NORM{\psh}{\Omega},
    \label{vstab6}
  \end{align}
  where $\CclemHat{}:=1+\Cclem{}$.
  Applying $(\vvh,\qsh)=(-\wvI,0)$ in the definition of
  $\calA_{\PSPG{},\hh}$, gives
  \begin{align}
    \calA_{\PSPG{},\hh}\big(\psih;(\uvh,\psh),(-\wvI,0) \big)
    &=  -\asVh(\uvh, \wvI) + {b_h(\wvI,\psh)} - \csh(\psih; \uvh, \wvI) - \calLss{3,h}(\uvh, \wvI)\nonumber\\
    &=: -\ass{1}+ \ass{2} -\ass{3} -\ass{4}.
    \label{vstab7}
  \end{align}
 {The estimate of the above four terms are given as follows:} \newline
  \medskip
  \noindent
  $\bullet$  We use the {inequalities~\eqref{est1}} and \eqref{vstab6}:
  \begin{align}
    \ass{1}
    \leq \alpha^\ast\mu\NORM{\nabla\uvh}{\Omega}\,\NORM{\nabla\wvI}{\Omega}
    \leq \frac{\mu\alpha^\ast\CclemHat{}\theta}{\beta_0}\,\NORM{\nabla\uvh}{\Omega}\,\NORM{\psh}{\Omega}.
    \label{vstab8}
  \end{align}

  \medskip
  \noindent
  $\bullet$ Concerning $\ass{2}$, we proceed as follows
  \begin{align}
    \ass{2}
    &= \sum_{\P\in\Th} \int_{\P} \PizP{k-1} \big(\nabla\cdot\wvI\big)\PizP{k}\psh\dE \nonumber\\
    &= \sum_{\P\in\Th} \Big(
    \int_{\P} \nabla\cdot\wvI(\PizP{k-1}\psh-\psh)\dE +
    \int_{\P} \nabla\cdot(\wvI - \wv)\psh\dE +
    \int_{\P} \nabla\cdot\wv\psh\dE
    \Big).
    \nonumber
    %% -------------------------------------------------------------
    \intertext{Regarding the second integral, we use integration by
      parts and {the continuity property of $\wv\in\Vv$, $\wvI \in \Vv^h$, and
      $\psh \in Q_h$:}}
    %% -------------------------------------------------------------
    \ass{2}
    &= \sum_{\P\in\Th} \Big(
    \int_{\P} \nabla\cdot\wvI (\PizP{k-1}\psh -\psh)\dE -
    \int_{\P} (\wvI - \wv)\cdot\nabla\psh\dE +
    \int_{\P} \nabla\cdot\wv\psh\dE
    \Big)\nonumber\\
    %% -------------------------------------------------------------
    &= \sum_{\P\in\Th} \Big(
    \int_{\P} \nabla\cdot\wvI (\PizP{k-1}\psh -\psh)\dE +
    \int_{\P} (\wvI - \wv)\cdot(\nabla\psh - \PvzP{k-1}\nabla\psh)\dE +
    \int_{\P} (\wvI - \wv)\cdot\PvzP{k-1}\nabla\psh\dE \nonumber\\
    %% -------------------------------------------------------------
    & \qquad
    +\int_{\P} \nabla\cdot\wv\psh\dE \Big) \nonumber\\
    &=: \ass{21} + \ass{22} + \ass{23} + \ass{24}.
    \label{vstab9}
  \end{align}
  Applying the bound \eqref{vstab6} and Lemma \ref{lemmaproj1}, it
  holds
  \begin{align}
    \ass{21}
    &
    \leq \sum_{\P\in\Th} \NORM{\nabla\wvI}{\P} \NORM{(I - \PizP{k-1})\psh}{\P}
    =    \sum_{\P\in\Th} \NORM{\nabla\wvI}{\P} \NORM{(I - \PizP{k-1})\,(I - \PinP{k-1}) \psh}{\P} \nonumber \\ 
    & \leq \sum_{\P\in\Th} \hP \NORM{\nabla\wvI}{\P} \NORM{\nabla(I - \PinP{k-1})\psh}{\P} \nonumber \\ 
    & \leq \frac{\CclemHat{}\theta}{\beta_0} \max_{\P\in\Th} \frac{\hP}{\tau_{\P}^{1/2}}\NORM{\psh}{\Omega}\calLss{1,h}^{1/2}(\psh,\psh).
    \label{vstab10}
  \end{align}
  We use Lemma \ref{lemmaproj2}, the property of the projectors, and
  \eqref{vstab6}:
  \begin{align}
    \ass{22}
    &
    \leq \sum_{\P\in\Th} \NORM{\wvI - \wv}{\P}\,\NORM{\nabla\psh - \PvzP{k-1}\nabla\psh}{\P}
    \leq \sum_{\P\in\Th} \hP\NORM{\nabla\wv}{\P} \,\NORM{\nabla( I- \PvnP{k-1}) \psh}{\P} \nonumber \\
    &
    \leq \frac{\CclemHat{} \theta}{\beta_0}
    \max_{\P\in\Th} \frac{\hP}{\tau_{\P}^{1/2}}
    \NORM{\psh}{\Omega}
    \calLss{1,h}^{1/2}(\psh,\psh).
    \label{vstab11}
  \end{align}
  Following the estimation of $\ass{22}$, we infer
  \begin{align}
    \ass{23}
    \leq \frac{\CclemHat{}\theta}{\beta_0}
    \max_{\P\in\Th}\frac{\hP}{\tau_{\P}^{1/2}}
    \NORM{\psh}{\Omega}
    \calLss{1,h}^{1/2}(\psh,\psh).
    \label{vstab12} 
  \end{align}
  Concerning $\ass{24}$, we utilize \eqref{vstab5}. Combining the
  above bounds, we obtain
  \begin{align}
    \ass{2}
    \geq \theta\NORM{\psh}{\Omega}^2 -
    \frac{3\CclemHat{}\theta}{\beta_0}
    \max_{\P\in\Th}\frac{\hP}{\tau_{\P}^{1/2}}
    \NORM{\psh}{\Omega}
    \calLss{1,h}^{1/2}(\psh,\psh).
    \label{vstab13}
  \end{align}

  \medskip
  \noindent
  $\bullet$ Concerning $\ass{3}$, we use the Sobolev embedding
  theorem, \eqref{stabc1}{,} and \eqref{vstab6}:
  \begin{align}
    \ass{3}
    \leq \frac{\mu\alpha^\ast\CclemHat{}\theta}{2\beta_0}
    \NORM{\nabla\uvh}{\Omega}\,
    \NORM{\psh}{\Omega}.
    \label{vstab14}
  \end{align}

  \medskip
  \noindent
  $\bullet$ Employing the bounds \eqref{vem-a} and \eqref{vstab6}, it
  holds
  \begin{align}
    \ass{4}
    \leq \alpha^\ast\max_{\P\in\Th}\delta_E
    \NORM{\nabla\uvh}{\Omega}\,
    \NORM{\nabla\wvI}{\Omega}
    \leq \frac{\alpha^\ast\CclemHat{}\theta}{\beta_0}
    \max_{\P\in\Th} \delta_E
    \NORM{\nabla\uvh}{\Omega}\,
    \NORM{\psh}{\Omega}.
    \label{vstab15}
  \end{align}
  Combining the bounds \eqref{vstab8}, \eqref{vstab13},
  \eqref{vstab14}{,} and \eqref{vstab15}, and {exploiting the assumption
  \eqref{stabc2}, the Young inequality, and $ab \leq a^2 + \frac{b^2}{4}$ with $a,b>0$, we deduce that}
  \begin{align}
    & {\calA_{\PSPG{},\hh}\big(\psih;(\uvh,\psh),(-\wvI,0) \big)}
    \nonumber \\ & \quad
    %% -----------------------------------------------
    \geq
    - \Big(\frac{3\mu\alpha^\ast\CclemHat{}\theta}{2\beta_0}
    + \frac{\alpha^\ast\CclemHat{}\theta}{\beta_0}\Big)\,\NORM{\nabla\uvh}{\Omega}\,\NORM{\psh}{\Omega}
    + \theta\NORM{\psh}{\Omega}^2
    - \frac{3\CclemHat{}\lambda_2^\ast\theta}{\beta_0}\,\NORM{\psh}{\Omega}\calLss{1,h}^{1/2}(\psh,\psh)
    %% -----------------------------------------------
    \nonumber \\ & \quad
    \geq
    - \Big( \frac{3\mu\alpha^\ast\CclemHat{}}{4\beta_0} + \frac{\alpha^\ast\CclemHat{}}{\beta_0} \Big)\,\NORM{\nabla\uvh}{\Omega}^2
    + \Big( \theta^{-1}-\frac{3\mu\alpha^\ast\CclemHat{}}{4\beta_0}
    - \frac{\alpha^\ast\CclemHat{}}{4\beta_0}
    - \frac{3\CclemHat{}\lambda_2^\ast }{4\beta_0} \Big)\Ys^2 
    %% -----------------------------------------------
    \nonumber \\ & \quad\,\,\phantom{\geq}
    - \frac{3\CclemHat{}\lambda_2^\ast}{\beta_0} \calLss{1,h}(\psh,\psh)
    %% -----------------------------------------------
    \nonumber \\ & \quad
    \geq
    - \frac{\big(4\alpha^\ast+3\alpha^\ast\mu+12\lambda_2^\ast)\CclemHat{}}{4\beta_0}\,\normiii{\uvh}^2_\ast
    + \frac{(3\mu\alpha^\ast+\alpha^\ast+3\lambda_2^\ast)\CclemHat{}}{4\beta_0}\Ys^2,
    \nonumber 
  \end{align}
  where the last line is obtained using the definition of $\theta$.

  \medskip\noindent
  \textbf{Step 3.} {This step focuses on the construction of at least one $(\vvh,\qsh) \in \Vvh{} \times \Qsh$ for any given $(\uvh, \psh) \in \Vvh{}\times \Qsh$.}  To this end, we now choose $\Theta>0$ be an arbitrary constant{,} and set
  $(\vvh,\qsh)=(\uvh-\Theta\wvI,\psh)$.
  Under these settings, we obtain the following
   \begin{align}
     &\calA_{\PSPG{},\hh}\big(\psih;(\uvh,\psh),(\vvh,\qsh)\big)
     =    \calA_{\PSPG{},\hh}\big(\psih;(\uvh,\psh),(\uvh-\Theta\wvI,\psh)\big)
     %% --------------------------------------------------------------------
     \nonumber \\[0.5em] & \qquad
     \geq \Big(
     \frac{\alpha_\ast\mu}{8} -
     \frac{\Theta\alpha^\ast(4 + 3\mu + 12\lambda^\ast_2)\CclemHat{}}{4\beta_0}\Big)\,\normiii{ \uvh}^2_\ast
     +  \frac{\Theta(3\mu\alpha^\ast + \alpha^\ast + 3\lambda_2^\ast)\CclemHat{}}{4\beta_0} \Ys^2
     %% --------------------------------------------------------------------
     \nonumber \\ & \qquad
     \geq
     \min\Big\{\frac{\alpha_\ast\mu}{16},\frac{\Theta}{2}\Big\}\,\normiii{(\uvh,\psh)}^2,
     \label{vstab0}
   \end{align}
   where we have invoked 
   $\Theta={\mu\beta_0}\slash{4(4+3\mu+12\lambda_2^\ast)\CclemHat{}}$ in the last line.
   
   \medskip\noindent
   \textbf{Step 4.} Following \cite[Lemma 10,
     cf. (87)]{mishra2025equal}, we can easily show that there exist a
   constant $\alpha_0 >2${,} independent of the mesh size such that{,}
   $\normiii{(\vvh,\qsh)}^2\leq\alpha_0\normiii{(\uvh,\psh)}^2$.
   Combining this with \eqref{vstab0}, we obtain the result
   \eqref{d1well1} with
   $\beta_\ast=\frac{1}{\sqrt{\alpha}_0}\min\Big\{\frac{\alpha_\ast\mu}{16},\frac{\Theta}{2}\Big\}$.
   Thus, Problem~\eqref{dsvem1} has a unique solution. Moreover, it also
   satisfies
   \begin{align}
     \beta_\ast\normiii{(\uvh,\psh)}\,\normiii{(\vvh,\qsh)} 
     &
     \leq \Css{P}\Big( \NORM{\fv}{\Omega}
     + \East\NORM{\gs}{\Omega}
     + \Css{P}\East\kappa^\ast\NORM{\nabla\psih}{\Omega}\Big)\,\NORM{\nabla\vvh}{\Omega}\, \nonumber
     %% ---------------------------------------------------------------------------------
     \\ \nonumber &
     \phantom{\leq}\quad
     + \sum_{\P\in\Th}\tau_{\P}^{1/2}\Big( \NORM{\fv}{\P}
     + \East \NORM{g}{\P}
     + \East \kappa^\ast\NORM{\psih}{\P} \Big),\normiii{(\vvh,\qsh)}_E
     %% ---------------------------------------------------------------------------------
     \\  &
     \leq \big( \Css{P} + \frac{\abs{\Omega}^{1/2}}{4\Css{inv}}\big)
     \Big( \NORM{\fv}{\Omega}
     + \East\NORM{\gs}{\Omega}
     + \Css{P}\East\kappa^\ast\NORM{\nabla\psih}{\Omega}\Big)\,\normiii{(\vvh,\qsh)}. \label{uq2}
   \end{align}
   Thus, the result \eqref{dstabv} {follows from the bounds~\eqref{uq2} and~\eqref{stabc1}.}
\end{proof}

%% \MGT{
\begin{remark}\label{rem:coercivity}
  The coercivity estimate \eqref{vstab4} in the proof of
  Theorem~\ref{dwell1} relies on the identity
  $\min\big\{\frac{\mu\alpha_\ast}{2},1\big\} =
  \frac{\mu\alpha_\ast}{2}$, which holds under the condition
  \begin{align}
    \mu\,\alpha_\ast \leq 2.
    \label{coer_cond}
  \end{align}
  Since $\alpha_\ast = \min\{1,\lambda_{1\ast}\} \leq 1$ by
  definition, condition~\eqref{coer_cond} is automatically satisfied
  whenever $\mu \leq 2$.
  For the parameter regimes that we will consider in the numerical
  experiments of Section~\ref{sec-6}, where $\mu = 1$ (Examples~1
  and~2) and $\mu = 0.1$ (Example~3), this condition holds with a
  comfortable margin.
  We note that even when $\mu > 2$, the proof can be adapted by
  retaining the minimum in the coercivity constant, which leads to
  $\min\big\{\frac{\mu\alpha_\ast}{2},1\big\}$ replacing
  $\frac{\mu\alpha_\ast}{2}$ in~\eqref{vstab4} and all subsequent
  estimates that depend on it.
  A thourough discussion about this topic is also reported in
  Section~\ref{sec:smallness}.
\end{remark}
%% }

Hereafter, we discuss the existence and uniqueness of the discrete
solution to the potential transport equation \eqref{dsvem2}.
Following the continuous case Lemma \ref{wellp}, the nonlinear form
$\mathcal{B}_h(\uvh;\cdot,\cdot)$ satisfies the strongly monotone
property (using the strongly monotone property of $\kappa(\cdot)$) and
Lipschitz continuity.
Here, we use the Brouwer fixed-point theorem to show the
well-posedness of the problem \eqref{dsvem2}.

\begin{theorem}\label{dwell2}
  For any fixed $\uvh\in\Vvh{}$, the discrete problem
  \eqref{dsvem2} has a unique solution $\psih \in \mathlarger
  \Phi_{h}$, satisfying
  \begin{align}
    \NORM{\nabla\psih}{\Omega}
    \leq \frac{\Css{P}}{\epsilon\gamma_\ast}\,\NORM{\gs}{\Omega}.
    \label{dstabp}
  \end{align}
\end{theorem}
\begin{proof}
  The proof of Theorem \ref{dwell2} follows from the Brouwer
  fixed-point theorem on the finite-dimensional spaces with some
  modification in the proof of \cite[Theorem 5.5]{vem28m}.
  In addition, it also holds
  \begin{align*}
    {\epsilon\gamma_\ast\NORM{\nabla\psih}{\Omega}^2}
    \leq \sum_{\P\in\Th}\NORM{\gs}{\P}\,\NORM{\psih}{\P}
    \leq \Css{P}\NORM{\gs}{\Omega}\,\NORM{\nabla\psih}{\Omega}.
  \end{align*}
  Thus, the stability estimate \eqref{dstabp} readily follows from the
  above analysis.
\end{proof}

%%%%%%%%%%%%%%%%%%%%%%%%%%%%%%%%%%%%%%%%%%%%%%%%%%%%%%%%%%%%%%%%%%%%%%%%%%%%%%%%%%%%%%%%%
%%%%%%%%%%%%%%%%%%%%%%%%%%%%%%%%%%%%%%%%%%%%%%%%%%%%%%%%%%%%%%%%%%%%%%%%%%%%%%%%%%%%%%%%%

\subsection{Well-posedness of the stabilized virtual element problem}
We are now in a position to demonstrate the well-posedness of the
stabilized problem \eqref{svem}.
To do this, we adopt the fixed-point approach utilized in the
continuous case to reformulate the problem \eqref{svariation} into an
equivalent fixed-point problem.
In view of Theorem \ref{dwell1}, we introduce a well-defined discrete
operator
$\mathcal{S}^h_{flow}:{\mathlarger \Phi_h} \rightarrow\Vvh{}\times\Qsh$
such that
\begin{align}
  \psih\rightarrow\mathcal{S}^h_{flow}
  := (\mathcal{S}^h_{1,flow}(\psih),\mathcal{S}^h_{2,flow}(\psih))
  =: (\uvh,\psh)
  \qquad \text{for~all~}\,\,\psih\in {\mathlarger \Phi_h},
\end{align}
where $(\uvh,\psh)\in\Vvh{}\times\Qsh$ is the unique solution of
\eqref{dsvem1} with given $\psih\in{\mathlarger \Phi_h}$ satisfying the
assumption of Theorem \ref{dwell1}.
Following Theorem \ref{dwell2}, we introduce a well-defined map
$\mathcal{N}^h_{elec}:\Vvh{}\rightarrow{\mathlarger \Phi_h}$ such that
\begin{align}
  \uvh\rightarrow\mathcal{N}^h_{elec}(\uvh) = \psih
  \qquad\text{for~all}\,\,\uvh\in\Vvh{},
\end{align}
where $\psih\in{\mathlarger \Phi_h}$ is the unique solution of
\eqref{dsvem2} with given $\uvh\in\Vvh{}$.
Mimicking the continuous case, we introduce an operator
$\calTh:{\mathlarger \Phi_h}\rightarrow{\mathlarger \Phi_h}$ such
that
\begin{align}
  \calTh(\psih)
  = \mathcal{N}^h(\mathcal{S}^h_{1,flow}(\psih))
  \qquad\text{for~all}\,\,\psih\in{\mathlarger \Phi_h}. 	
\end{align}
Finally, we introduce the equivalent discrete fixed-point problem as
follows:
\emph{Find {$\mathbf{\psi}_h\in\mathlarger \Phi_h$} such that
$\calTh(\psih)=\psih$}.
Therefore, seeking the fixed point of $\calTh$ will be
equivalent to solving the stabilized problem \eqref{svem}.
\begin{lemma} \label{set2}
  We define
  $\widehat{\mathlarger\Phi}_h:=\big\{\phih\in{\mathlarger \Phi_h}\,\,\text{such~that}\,\,\NORM{\nabla\phih}{\Omega}\leq\rs\big\}$
  with an arbitrary constant $r$.
  Let the data of the problem be such that
  \begin{align}
    \frac{\Css{P}}{\epsilon\gamma_\ast}\NORM{\gs}{\Omega}
    \leq \frac{\alpha_\ast\mu}{2\East\Css{1\hookrightarrow4}^2}
    \leq \rs.
  \end{align}
  Then, $\calTh(\widehat{\mathlarger\Phi}_{\hh}) \subset
  \widehat{\mathlarger\Phi}_{\hh}$.
\end{lemma}
\begin{proof}
  The proof follows from Lemma~\ref{selfc}.
\end{proof}

\begin{lemma}\label{dswell}
  Under the assumption of Lemma \ref{set2}, for any $\psih,
  \psihHat\in\widehat{\mathlarger\Phi}_h$, the discrete
  operator $\calTh$ satisfies the following
  \begin{align}
    \snorm{\calTh(\psih) -\calTh(\psihHat)}{1,\Omega}
    \leq
    &
    \frac{\rs\East\Css{1\hookrightarrow4}^2}{\gamma_\ast \epsilon \beta_\ast}\,
    \bigg(
    \Css{P}^2\kappa^\ast
    + \frac{\Css{P}\kappa^\ast\abs{\Omega}^{1/2}}{4\Css{inv}}
    + \frac{5}{4}\Css{1\hookrightarrow4}^2\,\snorm{\calSh{1,\text{flow}}(\widehat{\psi})}{1,\Omega}
    \bigg)\times
    \nonumber\\
    &
    \NORM{\nabla(\psih-\psihHat)}{\Omega}.
    \label{maincntsp}
  \end{align}
\end{lemma}
\begin{proof}
  The proof of Lemma \ref{dswell} follows from the following three
  steps.
  
  \medskip
  \noindent
  \textbf{Step 1.} {We first establish the Lipschitz continuity of the operator $\calSh{flow}$. To obtain this,} we choose $\psih$,
  $\psihHat\in\widehat{\mathlarger \Phi}_h$ such that
  \begin{itemize}
  \item $\calSh{flow}(\psih)= (\calSh{1,flow}(\psih),
  \calSh{2,flow}(\psih))=:(\uvh,\psh)$;
  %% ---
  \item $\calSh{flow}(\psihHat):=
  (\calSh{1,flow}(\psihHat),
  \calSh{2,flow}(\psihHat))=:(\widehat{\uv}_h,
  \widehat{p}_h)$.
  \end{itemize}
  %% %%
  Employing the discrete inf-sup inequality
  \eqref{d1well1} {and the problem~\eqref{dsvem1}}, it gives
     \begin{align}
       &\beta_\ast
       \normiii{\calSh{flow}(\psih) - \calSh{flow}(\psihHat)}  \,\normiii{(\vvh,\qsh)}
       \leq \calA_{\PSPG{},\hh}\big( \psih; (\uvh -  \uvhHat, \psh -  \pshHat), (\vvh, \qsh) \big)
       \nonumber \\
       & \qquad 
       =  F^{\psih}_{\text{PSPG},h}(\vvh,\qsh) - \calA_{\PSPG{},\hh}\big( \psih;( \uvhHat,\pshHat),(\vvh, \qsh) \big) \nonumber \\
       &\qquad
       = F^{\psih}_{\text{PSPG},h}(\vvh,\qsh)
       - F^{\psihHat}_{\text{PSPG},h}(\vvh,\qsh)
       + \calA_{\PSPG{},\hh}\big(\psihHat-\psih;( \uvhHat,\pshHat),(\vvh, \qsh) \big).
       \label{eq:step1a}
     \end{align}
     Expanding the terms on the right, we obtain
     %%
     %% Second chain (eq:step1b): Expanding the terms explicitly and applying initial estimates
     \begin{align}
       &F^{\psih}_{\text{PSPG},h}
       (\vvh,\qsh) - F^{\psihHat}_{\text{PSPG},h}(\vvh,\qsh)
       + \calA_{\PSPG{},\hh}\big(
       \psihHat - \psih; ( \uvhHat, \pshHat), (\vvh, \qsh)
       \big) \nonumber \\
       %% --------------------------------------------------------------
       &\hspace{2cm}
       = \sum_{\P\in\Th} \bigg(
       \int_{\P} \big(\kappa(\PizP{k}\psih) - \kappa(\PizP{k}\psihHat)\big)\,\Ev\cdot\PvzP{k}\vvh\dE
       \nonumber \\
       %% --------------------------------------------------------------
       &\hspace{3.5cm}
       + \int_{\P} \tau_{\P}
       \big(\kappa(\PizP{k}\psih) - \kappa(\PizP{k}\psihHat)\big)\,\Ev\cdot\PvzP{k-1}\nabla\qsh\dE
       \nonumber \\
       %% --------------------------------------------------------------
       &\hspace{3.5cm}
       + \int_{\P} \PvzP{k}\uvh\cdot\PvzP{k-1}\nabla(\psihHat - \psih)\,\Ev\cdot\PvzP{k}\vvh\dE
       \nonumber \\
       %% --------------------------------------------------------------       
       &\hspace{3.5cm}
       + \int_{\P} \tau_{\P}\big( \PvzP{k}\uvh\cdot\PvzP{k-1}\nabla(\psihHat-\psih) \big)\,\Ev\cdot\PvzP{k-1}\nabla\qsh\dE
       \bigg)
       \nonumber\\
       %% --------------------------------------------------------------
       &\hspace{2cm}
       \leq \sum_{\P\in\Th} \bigg(
       \East\kappa^\ast\NORM{\psih-\psihHat}{\P}\,\NORM{\vvh}{\P} +
       \East\kappa^\ast\tau_{\P}\NORM{\psih-\psihHat}{\P}\,\NORM{\PvzP{k-1}\nabla\qsh}{\P}
       \nonumber \\
       %% --------------------------------------------------------------
       &\hspace{3.5cm}
       + \East\NORM{\nabla(\psih-\psihHat)}{\P}\,\NORM{\uvh}{0,4,\P}\,\NORM{\vvh}{0,4,\P}
       \nonumber \\
       %% --------------------------------------------------------------
       &\hspace{3.5cm}
       + \East\tau_{\P}\NORM{\PvzP{k-1}\nabla(\psih-\psihHat)}{0,4,\P}\,\NORM{\uvh}{0,4,\P}\,\NORM{\PvzP{k-1}\nabla\qsh}{\P}
       \bigg).
       \label{eq:step1b}
     \end{align}
     Combining the estimates and using the bounds from Remark
     \ref{qstab}, inequality \eqref{bound1}, the Sobolev embedding
     theorem, and $\tau_{\P}\leq \frac{\hP^2}{16 \Css{inv}}$, we
     deduce
     %%
     %% Third chain (eq:step1c): Final bounds using the referenced results
     \begin{align}
       & \sum_{\P\in\Th} \bigg(
       \East \kappa^\ast\NORM{\psih-\psihHat}{\P} \NORM{\vvh}{\P} +
       \East \kappa^\ast\tau_{\P}\NORM{\psih-\psihHat}{\P}\,\NORM{\PvzP{k-1}\nabla\qsh}{\P}
       \nonumber \\
       %% -------------------------------------------------------------
       & \quad
       + \East\NORM{\nabla(\psih-\psihHat)}{\P}\,\NORM{\uvh}{0,4,\P}\,\NORM{\vvh}{0,4,\P}
       + \East\tau_{\P}\NORM{\PvzP{k-1}\nabla(\psih-\psihHat)}{0,4,\P}\,\NORM{\uvh}{0,4,\P}\,\NORM{\PvzP{k-1}\nabla\qsh}{\P}
       \bigg)
       \nonumber \\
       %% -------------------------------------------------------------
       & \quad\leq
       \East \bigg(
       \Css{P}^2\kappa^\ast +
       \frac{\Css{P}\kappa^\ast\abs{\Omega}^{1/2}}{4\Css{inv}} +
       \big(\Css{1\hookrightarrow4}^2 + \frac{1}{4}\Css{1\hookrightarrow4}^2 \big)
       \NORM{\nabla\uvhHat}{\Omega}
       \bigg)\,\NORM{\nabla(\psih-\psihHat)}{\Omega}\,\normiii{(\vvh,\qsh)}. \label{eq:step1c}
     \end{align}
   %% }
   %%
   Applying the definition of
   $\calSh{{flow}}$, it holds
   \begin{align}
     \normiii{\calSh{flow}(\psih) - \calSh{{flow}}(\psihHat)}
     \leq
     \frac{\East}{\beta_\ast} \bigg(
     \Css{P}^2\kappa^\ast
     + \frac{\Css{P}\kappa^\ast\abs{\Omega}^{1/2}}{4\Css{inv}}
     + \frac{5}{4}\Css{1\hookrightarrow4}^2\snorm{\calSh{1,{flow}}(\psiHat)}{1,\Omega}\bigg)\,\NORM{\nabla(\psih-\psihHat)}{\Omega}.
     \label{ds1}
   \end{align}

   \medskip
   \noindent
   \textbf{Step 2.} {In this step, we investigate the Lipschitz continuity of the operator $\calNh{{elec}}$.} Let $\uvh$, $\uvhHat\in\Vvh{}$ such that
   $\calNh{{elec}}(\uvh)=\psih$ and $\calNh{{elec}}(\uvhHat)=\psihHat$.
   We  exploit the definition use the stability property of $\asph$ {(cf. Lemma \ref{estimate}, the bound~\eqref{est3})}
   and the strongly monotone property of $\kappa(\cdot)$:
   \begin{align*}
     &\gamma_\ast\epsilon\snorm{\calNh{{elec}}(\uvh)- \calNh{{elec}}(\uvhHat)}{1,\Omega}^2
     =    \gamma_\ast \epsilon\NORM{\nabla(\psih-\psihHat)}{\Omega}^2\\
     &\qquad
     \leq \calBh(\uvh; \psih, \psih-\psihHat) - \calBh(\uvh; \psihHat, \psih-\psihHat)
     =    \calBh(\uvhHat; \psihHat, \psih-\psihHat) - \calBh(\uvh; \psihHat, \psih-\psihHat)\\
     &\qquad
     =    \cs_{p,h}^{skew}(\uvhHat-\uvh; \psihHat, \psih-\psihHat)
     \leq  \Css{1\hookrightarrow4}^2\NORM{\nabla\psihHat}{\Omega}\,\NORM{\nabla(\uvh-\uvhHat)}{\Omega}\,\NORM{\nabla(\psih-\psihHat)}{\Omega},
   \end{align*} 
   where we have exploited the H\"older inequality and the Sobolev embedding theorem in the last line.  Thus, we arrive at
   \begin{align}
     |\calNh{{elec}}(\uvh)- \calNh{{elec}}(\uvhHat)|_{1,\Omega} 
     &\leq \frac{C_{1\hookrightarrow4}^2} {\gamma_\ast \epsilon} |\calNh{\text{elec}}(\uvhHat)|_{1,\Omega} \NORM{\nabla(\uvh-\uvhHat)}{\Omega}. \label{ds2}
   \end{align}

   \medskip
   \noindent
   \textbf{Step 3.} {We are now in position to show the continuity of the map $\calTh$. To this end, we recall} \eqref{ds1} and \eqref{ds2}, and we proceed as follows
   \begin{align}
     \snorm{\calTh(\psih) -\calTh(\psihHat)}{1,\Omega}
     &=    \snorm{ \calNh{{elec}}(\calSh{1,{flow}}(\psih))- \calNh{{elec}}(\calSh{1,{flow}}(\psihHat)) }{1,\Omega} \nonumber\\
     &\leq \frac{\Css{1\hookrightarrow4}^2}{\gamma_\ast\epsilon} \snorm{\calTh(\psihHat)}{1,\Omega}\,\snorm{\calSh{1,{flow}}(\psih) - \calSh{1,{flow}}(\psihHat) }{1,\Omega}.\nonumber
     %&\leq \frac{r\East C_{1\hookrightarrow4}^2} {\gamma_\ast \epsilon \beta_\ast}   \Big[ \Css{P}^2  \kappa^\ast  + \frac{\Css{P} \kappa ^\ast |\Omega|^{1/2}}{4\Css{inv}} + \frac{5}{4} C_{1\hookrightarrow4}^2 | \calSh{1,\text{flow}}(\psiHat)|_{1,\Omega}\big] \NORM{\nabla(\psih-\psihHat)}{\Omega}. \nonumber
   \end{align}
   Thus, \eqref{maincntsp} can be easily obtained from
   \eqref{ds1}.
 \end{proof}
Following the continuous case, the discrete operator $\calTh$ has at
least one fixed point using the Brouwer fixed-point theorem
(cf. Lemmas \ref{set2} and \ref{dswell}).
Furthermore, the uniqueness of the fixed-point of $\calTh$ can be
demonstrated by introducing the necessary condition on the data of the
problem $(P)$. We summarize the existence and uniqueness of the
fixed-point problem in the following Proposition.
\begin{proposition} \label{uni2}
  Under the assumption of Lemma \ref{set2}, the operator
  $\calTh$ has at least one fixed-point.
  Thus, the stabilized virtual element problem \eqref{svem} has at
  least one solution $(\uvh,\psh,\psih)\in\Vvh{}\times\Qsh\times
  \mathlarger \phih$ with $\psi\in\widehat{\mathlarger \Phi}_h$ such
  that there holds
  \begin{align}
    \normiii{(\uvh, \psh)}
    \leq
    \frac{1}{\beta_\ast} \bigg( \Css{P} +
    \frac{\alpha_\ast \abs{\Omega}^{1/2}}{4\Css{inv}}
    \bigg)\,
    \bigg(
    \NORM{\fv}{\Omega}
    + \East \NORM{\gs}{\Omega}
    + \frac{\alpha_\ast\Css{P}\kappa^\ast\mu}{2\Css{1\hookrightarrow 4}^2}
    \bigg),
    \qquad
    \NORM{\nabla\psi}{\Omega}
    \leq \frac{\Css{P}}{\gamma_\ast\epsilon}\NORM{\gs}{\Omega}. \label{vvelv2}
  \end{align}
  Moreover, if the data of {Problem~\eqref{P}} satisfies the following:
  \begin{align}
    \frac{\rs\East\Css{1\hookrightarrow4}^2}{\gamma_\ast\epsilon\beta_\ast}\bigg(
      \Css{P}^2\kappa^\ast
      + \frac{\Css{P}\kappa^\ast\abs{\Omega}^{1/2}}{4\Css{inv}}
      + \frac{5\Css{1\hookrightarrow4}^2}{4\beta_\ast} \Big(
      \Css{P} + \frac{\alpha_\ast\abs{\Omega}^{1/2}}{4\Css{inv}}\Big)
      \Big(\NORM{\fv}{\Omega} + \East\NORM{\gs}{\Omega} + \frac{\alpha_\ast\Css{P}\kappa^\ast \mu}{2C^2_{1 \hookrightarrow 4}}\Big)
      \bigg) < 1.
  \end{align}
  Then, the operator $\calTh$ has a unique fixed-point.
  Equivalently, the problem \eqref{svem} has a unique solution.
\end{proposition}

%%%%%%%%%%%%%%%%%%%%%%%%%%%%%%%%%%%%%%%%%%%%%%%%%%%%%%%%%%%%%%%%%%%%%%%%%%%%%%%%%%%%%%%%%%%%%%%%%%%%%%%
%                 Error estimates
%%%%%%%%%%%%%%%%%%%%%%%%%%%%%%%%%%%%%%%%%%%%%%%%%%%%%%%%%%%%%%%%%%%%%%%%%%%%%%%%%%%%%%%%%%%%%%%%%%%%%%%

\subsection{A priori error estimate}
\label{theo}
In this section, we derive a priori error estimates in the energy norm
with optimal convergence rates.
Throughout this section, the virtual interpolant of
$(\uv,\ps,\psi)\in\Vv\times\Qs\times\mathlarger\Phi$ is denoted by
$(\uvI,\psI,\psiI)\in\Vvh{}\times\Qsh\times{\mathlarger \Phi_h}$.
We first introduce the following regularity condition to derive the
error estimate:
%%
%Let us introduce the following notations
%\begin{align*}
%e^\uv&:=\uv- \uvh, & e^p&:=p-\psh, & e^\theta&:=\theta- \theta_h, &
%e^\uv_I &:= \uv-\uv_I, &  e^\psI &:= p-\psI,  &
% e^\theta_I&:=\theta - \theta_I,\\
% e^\uvh &:= \uvh-\uv_I,&  e^\psh &:= \psh - \psI,  & e^\theta_h&:= \theta_h - \theta_I.
%\end{align*}

\medskip
\noindent \textbf{(A2)} Regularity assumptions:
\begin{align*}
  &  \uv\in\Vv\cap\big[\HS{k+1}(\Th)\big]^2,
  && \ps\in\Qs\cap\HS{k}(\Th),  
  && \psi\in\mathlarger\Phi\cap\HS{k+1}(\Th),  
  && \fv\in[\HS{k-1}(\Th)]^2, \\ 
  &  {\gs}\in\HS{k-1}(\Th),  
  && \Ev\in[\WS{k-1,\infty}{}(\Th)]^2,  
  && \kappa\in\HS{k-1}(\Th).
\end{align*}
Hereafter, we use $\NORM{\nabla\uv}{\cdot}\leq\Css{u}$ {(cf. Theorem~\ref{wellv}, the bound
\eqref{velv})} in the subsequent analysis.
Furthermore, the constant $\Cs$ is independent of the mesh size and
also depends on the data of the problem \eqref{P}, the mesh regularity
constant, and the stability constants.
%%
% In addition, under the assumption \textbf{(A2)}, we have
%
% \begin{align}
%	\NORM{\nabla(I - \PinP{k}) \theta_I}{0,\P} &\leq \NORM{\nabla\theta_I - \PvzP{k-1} \nabla\theta_I}{0,\P} \nonumber \\
%	& \leq \NORM{\nabla\theta-\nabla\theta_I}{0,\P} + \NORM{\nabla\theta - \PvzP{k-1}  \nabla\theta}{0,\P} + \NORM{\PvzP{k-1}  (\nabla\theta - \nabla\theta_I)}{0,\P} \nonumber \\
%	& \leq h^k_E \|\theta\|_{k+1,\P}. \label{err0}
% \end{align}
%%
% Hereafter, we will first derive the following three lemmas to
% establish the convergence estimates in the energy norm, assuming
% that the assumptions \textbf{(A0)} and \textbf{(A1)} are valid.

%%%%%%%%%%%%%%%%%%%%%%%%%%%%%%%%%%%%%%%%%%%%%%%%%%%%%%%%%%%%%%%%%%%%%%%%%%%%%%%%%%%%%%%%%%%%%%%%%%%%%%%%%%%%%%
%      VELOCITY LEMMA: stabilized Stokes
%%%%%%%%%%%%%%%%%%%%%%%%%%%%%%%%%%%%%%%%%%%%%%%%%%%%%%%%%%%%%%%%%%%%%%%%%%%%%%%%%%%%%%%%%%%%%%%%%%%%%%%%%%%%%

\begin{lemma}
  \label{velocity}
  With the assumption \textbf{(A2)} and the hypothesis of Theorems
  \ref{dwell1} and \ref{wellv}, we consider
  $(\uv_I,\psI)\in\Vvh{}\times\Qsh$ denotes the virtual interpolant of
  $(\uv,p)\in\Vv\times\Qs$.
  Furthermore, for given $\psi\in{\mathlarger\Phi}$ and
  $\psih\in{\mathlarger\Phi}_h$, let $(\uv,\ps)\in\Vv\times\Qs$ and
  $(\uvh,\psh)\in\Vvh{}\times\Qsh$ be the solution of problems
  \eqref{variation1} and \eqref{dsvem1}, respectively.
  Then, for all $(\vvh,\qsh)\in\Vvh{}\times\Qsh$, the following holds
  \begin{align}
    \calA_{\PSPG{}}\big( \psi;(\uv,\ps),(\vvh,\qsh) \big) -
    &
    \calA_{\PSPG{},\hh}\big( \psih;(\uvI,\psI),(\vvh,\qsh) \big)
    \leq \Cs\hh^k \bigg(
    \NORM{\uv}{k+1,\Omega} +
    \NORM{\ps}{k,\Omega} +
    \NORM{\psi}{k+1,\Omega} \nonumber \\
    &
    + \max_{\P\in\Th} \NORM{\Ev}{k-1,\infty,\P}\,\NORM{\uv}{k+1,\Omega}\,\NORM{\psi}{k+1,\Omega}
    + \NORM{\psi}{2,\Omega}\,\NORM{\uv}{k+1,\Omega}\bigg)\,
    \normiii{(\vvh,\qsh)} \nonumber \\ 
    &\quad
    + \frac{ 5\East\Css{1\hookrightarrow 4}^2}{4}\,
    \NORM{\nabla\uv}{\Omega}\,\NORM{\nabla(\psiI-\psih)}{\Omega}\,\normiii{(\vvh,\qsh)}.
    \label{verror}
  \end{align}
\end{lemma}
\begin{proof} Using the definition of $\mathcal{A}_{\PSPG{}}$ and $\calA_{\PSPG{},\hh}$, we have
  \begin{align}
    \calA_{\PSPG{}}\big(\psi;&(\uv,\ps),(\vvh,\qsh)\big) -
    \calA_{\PSPG{},\hh}\big(\psih;(\uvI,\psI),(\vvh,\qsh)\big)
    \nonumber \\
    &=
    \ass{V}(\uv,\vvh)   - \asVh(\uvI,\vvh) +
   { b_h(\vvh,\psI) - b(\vvh,\ps) +
    b(\uv,\qsh)   - b_h(\uvI,\qsh) \,+}
    \nonumber \\
    &\qquad
    {c(\psi;\uv,\vvh) - \csh(\psih;\uvI,\vvh)} +
    \calLss{1}(\ps,\qsh)   - \calLss{1,h}(\psI,\qsh)\,
    + \nonumber \\
    & \qquad
    \calLss{2}  \big(\psi;\uv,\qsh\big) -
    \calLss{2,h}\big(\psih;\uvI,\qsh\big) -
    \calLss{3,h}\big(\uvI,\vvh\big).
    \nonumber \\
    & =:
    \eta_{1}
    + \eta_{2}
    + \eta_{3}
    + \eta_{4}
    + \eta_{5}
    + \eta_{6}
    -\eta_{7}.
    \label{err05}
  \end{align}
  The estimation of the above involves the following steps:

  \medskip\noindent
  \textbf{Step 1.} {Estimate of $\eta_1, \eta_2$, and $\eta_3$.} Following \cite[Lemma~7]{mishra2024unified}, we
  obtain the following result
  \begin{align}
    \abs{\eta_1 + \eta_2 + \eta_3}
    &\leq  \abs{\ass{V}(\uv,\vvh) - \asVh(\uv_I, \vvh)}
    + {\abs{ b_h(\vvh,\psI) - b(\vvh,\ps) }
    + \abs{ b(\uv,\qsh)   - b_h(\uvI,\qsh) }}
    \nonumber \\
    &\leq \Cs\hh^k\big(
    \NORM{\uv}{k+1,\Omega} +
    \NORM{\ps}{k,\Omega}\big)\,\normiii{(\vvh,\qsh)}.
    \label{err1}
  \end{align}

  \medskip\noindent
  \textbf{Step 2.} {Estimate of $\eta_4$.} Adding and subtracting suitable terms, we infer
  \begin{align}
    \eta_4
    &= \sum_{\P\in\Th} \bigg(
    \int_{\P} \big(\uv\cdot\nabla\psi\big)\,\Ev\cdot\vvh\dE -
    \int_{\P} \big(\PvzP{k}\uvI\cdot\PvzP{k-1}\nabla\psih)\,\Ev\cdot\PvzP{k}\vvh\dE
    \bigg) \nonumber \\
    &= \sum_{\P\in\Th} \bigg(
    \int_{\P} \big(\uv\cdot\nabla\psi\big)\Ev\cdot\big(\vvh-\PvzP{k}\vvh\big)\dE +
    \int_{\P} \big((\uv-\PvzP{k}\uvI)\cdot\nabla\psi\big)\,\Ev\cdot\PvzP{k}\vvh\dE
    \nonumber \\
    & \qquad
    + \int_{\P} \PvzP{k}\uvI\cdot\big(\nabla\psi - \PvzP{k-1}\nabla\psiI\big)\,\Ev\cdot\PvzP{k}\vvh\dE
    + \int_{\P} \PvzP{k}\uvI\cdot\PvzP{k-1}\big(\nabla\psiI -  \nabla\psih\big)\,\Ev\cdot\PvzP{k}\vvh\dE
    \bigg)
    \nonumber \\
    &=: \eta_{41} + \eta_{42} +\eta_{43} + \eta_{44}.
    \label{err2}
  \end{align} 

  \medskip\noindent
  $\bullet$ We recall the orthogonality property of the projection
  operator, Lemma \ref{lemmaproj1}, and regularity assumption
  \textbf{(A2)}:
  \begin{align}
    \eta_{41}
    &= \sum_{\P\in\Th} \int_{\P} \Big(
    \big(\uv\cdot\nabla\psi\big)\Ev - \PvzP{k-1}
    \big(\uv\cdot\nabla\psi\Ev\big)
    \Big)\cdot\big(\vvh-\PvzP{k}\vvh\big)\dE
    \nonumber \\
    %% ------------------------------------------
    &\leq \Cs\sum_{\P\in\Th}
    \NORM{\Ev}{k-1,\infty,\P}\,\hP^{k}\,
    \snorm{\uv\cdot\nabla\psi}{k-1,\P}\,
    \NORM{\nabla\vvh}{\P}
    \nonumber \\
    %% ------------------------------------------
    &\leq \Cs\hh^k \max_{\P\in\Th}
    \NORM{\Ev}{k-1,\infty,\P}\,
    \snorm{\uv}{k-1,4,\Omega}\,
    \snorm{\nabla\psi}{k-1,4,\Omega}\,
    \NORM{\nabla\vvh}{\Omega}
    \nonumber \\
    %% ------------------------------------------
    &\leq \Cs\hh^k \max_{\P\in\Th}
    \NORM{\Ev}{k-1,\infty,\P}\,
    \NORM{\uv}{k+1,\Omega}\,
    \NORM{\psi}{k+1,\Omega}\,
    \normiii{( \vvh,\qsh)}.
    \label{err3}
  \end{align}

  \medskip\noindent
  $\bullet$ Concerning $\eta_{42}$, we proceed as follows
  \begin{align}
    \eta_{42}
    & \leq \Cs\sum_{\P\in\Th}
    \NORM{(\uv-\PvzP{k}\uvI)}{\P}\,\NORM{\nabla\psi}{0,4,\P}\,\NORM{\vvh}{0,4,\P}
    \nonumber \\
    %% --------------------------------------------------------------------------
    & \leq \sum_{\P\in\Th}\hP^{k+1}
    \NORM{\uv}{k+1,\P}\,\NORM{\nabla\psi}{0,4,\P}\,\NORM{\vvh}{0,4,\P}
    \nonumber \\
    %% --------------------------------------------------------------------------
    &\leq Ch^{k+1} \NORM{\uv}{k+1,\Omega}\,\NORM{\psi}{2,\Omega}\,\normiii{(\vvh,\qsh)}.
    \label{err4}
  \end{align}
  %Note that the last line is obtained using the boundedness of $\psi$ (cf. Theorem \ref{wellv}).  \newline

  \medskip\noindent
  $\bullet$ Applying Remark \ref{qstab}, the Sobolev embedding
  theorem, and Lemmas \ref{lemmaproj1} and \ref{lemmaproj2}, it gives
  \begin{align}
    \eta_{43}
    &\leq \sum_{\P\in\Th} \NORM{\uvI}{0,4,\P}\,\NORM{\nabla\psi-\PvzP{k-1}\nabla\psiI}{\P}\,\NORM{\vv}{0,4,\P}
    \nonumber \\
    %% --------------------------------------------------------------------------
    &\leq \Cs\hh^k \NORM{\uvI}{0,4,\Omega}\,\NORM{\psi}{k+1,\Omega}\,\NORM{\vv}{0,4,\Omega}
    \nonumber \\
    %% --------------------------------------------------------------------------
    &\leq \Cs\hh^k \NORM{\psi}{k+1,\Omega} \normiii{(\vv,q)}.
    \label{err5}
  \end{align} 

  \medskip\noindent
  $\bullet$ Using the stability property of $L^2$-projectors, the
  Sobolev embedding theorem and H\"older inequality, we obtain
  \begin{align}
    \eta_{44}
    &\leq \East \sum_{\P\in\Th} \NORM{\uvI}{0,4,\P}\,\NORM{\nabla\psiI-\nabla\psih}{\P}\,\NORM{\vvh}{0,4,\P}
    \nonumber \\
    %% ---------------------------------------------------------------------------------------------------
    &\leq \East \Css{1\hookrightarrow 4}^2
    \NORM{\nabla\uvI}{\Omega}\,
    \NORM{\nabla\psiI-\nabla\psih}{\Omega}\,
    \NORM{\nabla\vvh}{\Omega}
    \nonumber \\
    %% --------------------------------------------------------------------------
    &\leq \East \Css{1\hookrightarrow 4}^2
    \NORM{\nabla\uv }{\Omega}\,
    \NORM{\nabla\psiI-\nabla\psih}{\Omega}\,
    \normiii{(\vvh,\qsh)}.
    \label{err6}
  \end{align}
  Combining \eqref{err3}--\eqref{err6}, we have the following estimate
  \begin{align}
    \eta_4
    &\leq \Cs\hh^k\big(
    \max_{\P\in\Th}
    \NORM{\Ev}{k-1,\infty,\P}\,\NORM{\uv}{k+1,\Omega}\,\NORM{\psi}{k+1,\Omega} +
    \NORM{\uv}{k+1,\Omega}\,\NORM{\psi}{2,\Omega} +
    \NORM{\psi}{k+1,\Omega}\big)\,
    \normiii{( \vvh,\qsh)}
    \nonumber \\
    %% --------------------------------------------------------------------------
    & \qquad
    +\East\Css{1\hookrightarrow 4}^2
    \NORM{\nabla\uv}{\Omega}\,
    \NORM{\nabla(\psiI-\psih)}{\Omega}\,
    \normiii{( \vvh,\qsh)}.
    \label{err7}
  \end{align}

  \medskip\noindent
  \textbf{Step 3.} {Estimate of $\eta_5$.} We apply the orthogonality of the projectors, the
  bound \eqref{vem-b} and Lemma \ref{lemmaproj1}:
  \begin{align}
    \eta_5
    &= \sum_{\P\in\Th} \bigg(
    \int_{\P} \tau_{\P}\nabla\ps\cdot\nabla\qsh\dE -
    \int_{\P} \tau_{\P}\PvzP{k-1}\nabla\psI\cdot\PvzP{k-1}\nabla\qsh\dE
    - \tau_{\P} \SP_2\big( \psI - \PinP{k-1}\psI, \qsh - \PinP{k-1}\qsh \big)
    \bigg)
    \nonumber \\
    %% --------------------------------------------------------------------------
    &= \sum_{\P\in\Th} \bigg(
    \int_{\P} \tau_{\P}\big(\nabla\psI - \PvzP{k-1}\nabla\psI\big)\cdot\nabla\qsh\dE
    -\tau_{\P} \SP_2\big(\psI-\PinP{k-1}\psI, \qsh-\PinP{k-1}\qsh\big)
    \bigg)
    \nonumber \\
    & \leq \Cs\hh^k \NORM{\ps}{k,\Omega}\,\normiii{(\vvh,\qsh)}.
    \label{err8}
  \end{align}

  \medskip\noindent
  \textbf{Step 4.}  {Estimate of $\eta_6$.} Concerning $\eta_6$, we proceed as follows
  %% \MGT{
    \begin{align}
      \eta_6
      &= \sum_{\P\in\Th} \tau_{\P}\bigg(
      \int_{\P} \big( -\mu\Delta\uv + (\uv\cdot\nabla\psi)\Ev\big)\cdot\nabla\qsh\dE
      \nonumber \\
      %% --------------------------------------------------------------------------
      & \qquad -
      \int_{\P} \big(
      - \mu\nabla\cdot\PvzP{k-1}\nabla\uvI
      + (\PvzP{k}\uvI\cdot\PvzP{k-1}\nabla\psih)\Ev
      \big)\cdot\PvzP{k-1}\nabla\qsh\dE
      \bigg)
      \nonumber \\
      %% --------------------------------------------------------------------------
      &= \sum_{\P\in\Th} \tau_{\P}\bigg(
      \int_{\P} \mu\nabla\cdot\PvzP{k-1}\nabla(\uvI-\uv)\cdot\PvzP{k-1}\nabla\qsh\dE +
      \int_{\P} \mu\Delta\uv\cdot\big(\PvzP{k-1}\nabla\qsh-\nabla\qsh\big)\dE
      \nonumber \\
      & \qquad
      + \int_{\P} \big(\uv\cdot\nabla\psi\big)\Ev\cdot\big(\nabla\qsh-\PvzP{k-1}\nabla\qsh\big)\dE
      + \int_{\P} \big((\uv-\PvzP{k}\uvI)\cdot\nabla\psi\big)\Ev\cdot\PvzP{k-1}\nabla\qsh\dE
      \nonumber \\
      & \qquad
      + \int_{\P} \big(\PvzP{k}\uvI\cdot(\nabla\psi-\PvzP{k-1}\nabla\psiI)\big)\,\Ev\cdot\PvzP{k-1}\nabla\qsh\dE
      \nonumber\\
      & \qquad
      + \int_{\P} \big(\PvzP{k}\uvI\cdot\PvzP{k-1}\nabla(\psiI-\psih)     \big)\,\Ev\cdot\PvzP{k-1}\nabla\qsh\dE
      \bigg)
      \nonumber\\
      &=: \eta_{6,1} + \eta_{6,2} + \eta_{6,3} + \eta_{6,4} + \eta_{6,5} + \eta_{6,6}.
      \label{err9}
    \end{align}
  %% }

  \medskip\noindent
  $\bullet$ We use the inverse inequality, definition of
  $\normiii{\cdot}$, and Lemma \ref{lemmaproj2}:
  \begin{align}
    \eta_{6,1}
    \leq \sum_{\P\in\Th} \tau_{\P}\Css{inv}\hP^{-1}
    \NORM{\nabla(\uv-\uvI)}{\P}\,
    \NORM{\PvzP{k-1}\nabla\qsh}{\P}
    \leq \Cs\hh^k\NORM{\uv}{k+1,\Omega}\,\normiii{(\vvh,\qsh)}.
    \label{err10}
  \end{align}

  \medskip\noindent
  $\bullet$ Recalling the orthogonality of the projectors, definition
  of $\tau_{\P}$ and Lemma \ref{lemmaproj1}, we have
  \begin{align}
    \eta_{6,2}
    &
    \leq \sum_{\P\in\Th} \mu\tau_{\P}
    \NORM{\Delta\uv-\PvzP{k-1}\Delta\uv}{\P}\,
    \NORM{\PvzP{k-1}\nabla\qsh-\nabla\qsh}{\P}
    \leq \sum_{\P\in\Th} \mu\tau_{\P} \hP^{k-1}\NORM{\uv}{k+1,\P}\,\NORM{\nabla(I-\PinP{k-1})\qsh}{\P}
    \nonumber\\
    &
    \leq \Cs\hh^k\NORM{\uv}{k+1,\Omega}\,\normiii{(\vvh,\qsh)}.
    \label{err11}
  \end{align}

  \medskip\noindent
  $\bullet$ Following the estimation of $\eta_{6,2}$ and introducing
  the Sobolev embedding theorem, we arrive at
  \begin{align}
    \eta_{6,3}
    &\leq \sum_{\P\in\Th}
    \tau_{\P}
    \NORM{\big(\uv\cdot\nabla\psi\big)\Ev-\PvzP{k-1}\left(\big(\uv\cdot\nabla\psi\big)\Ev\right)}{\P}\,
    \NORM{\nabla\qsh-\PvzP{k-1}\nabla\qsh}{\P}
    \nonumber\\
    %% ---------------------------------------
    & \leq \sum_{\P\in\Th} \tau_{\P} \hP^{k-1}
    \NORM{\Ev}{k-1,\infty,\P}\,
    \snorm{\uv\cdot\nabla\psi}{k-1,\P}\,
    \NORM{\nabla(I-\PinP{k-1})\qsh}{\P}
    \nonumber\\
    %% ---------------------------------------
    & \leq \Cs\hh^k \max_{\P\in\Th}
    \NORM{\Ev}{k-1,\infty,\P}\,
    \NORM{\uv}{k+1,\Omega}\,
    \NORM{\psi}{k+1,\Omega}\,
    \normiii{(\vvh,\qsh)}.
    \label{err12}
  \end{align}

  \medskip\noindent
  $\bullet$ Following the estimation of $\eta_{42}$, we infer
  \begin{align}
    \eta_{6,4}
    & \leq \Cs\hh^k
    \NORM{\psi}{2,\Omega}\,
    \NORM{\uv}{k+1,\Omega}\,
    \normiii{(\vvh,\qsh)}.
    \label{err13}
  \end{align}

  \medskip\noindent
  $\bullet$ Applying the Sobolev embedding theorem, polynomial inverse
  inequality, and Lemma \ref{lemmaproj1}, it holds
  \begin{align}
    \eta_{6,5}
    &\leq \sum_{\P\in\Th} \tau_{\P}
    \NORM{\uvI}{0,4,\P}\,
    \NORM{\nabla\psi-\PvzP{k-1}\,\nabla\psiI}{\P}\,
    \NORM{\PvzP{k-1}\nabla\qsh}{0,4,\P}
    \nonumber\\
    %% --------------------------------------------
    &\leq \sum_{\P\in\Th} \tau_{\P}
    \NORM{\uvI}{0,4,\P}\big(
    \NORM{\nabla\psi-\PvzP{k-1}\nabla\psi}{\P} +
    \NORM{\nabla(\psi-\psiI)}{\P}\big) \hP^{-1/2}\NORM{\PvzP{k-1}\nabla\qsh}{\P}
    \nonumber \\
    %% --------------------------------------------
    & \leq \Cs\hh^k
    \NORM{\nabla\uv}{\Omega}\,
    \NORM{\psi}{k+1,\Omega}\,
    \normiii{(\vvh,\qsh)}
    \nonumber \\
    %% --------------------------------------------
    & \leq \Cs\hh^k \NORM{\psi}{k+1,\Omega}\,\normiii{(\vvh,\qsh)}.
    \label{err14}
  \end{align}

  $\bullet$ Recalling the bound \eqref{bound1}, Remark \ref{qstab} and
  the Sobolev embedding theorem, we obtain
  \begin{align}
    \eta_{6,6}
    &\leq \sum_{\P\in\Th} \East\tau_{\P}
    \NORM{\uvI}{0,4,\P}\,
    \NORM{\PvzP{k-1}\nabla(\psiI-\psih)}{0,4,\P}\,
    \NORM{\PvzP{k-1} \nabla\qsh}{\P}
    \nonumber\\
    &\leq \sum_{\P\in\Th} \East\tau_{\P}^{1/2} \Css{inv} \hP^{-1}
    \NORM{\uvI}{0,4,\P}\,
    \NORM{\psiI-\psih}{0,4,\P}\,
    \normiii{(\vvh,\qsh)}_{\P}
    \nonumber\\ 
    &\leq \frac{\East}{4}
    \NORM{\uvI}{0,4,\Omega}\,
    \NORM{\psiI-\psih}{0,4,\Omega}\,
    \normiii{(\vvh,\qsh)}
    \nonumber\\ 
    &\leq \frac{ \East \Css{1 \hookrightarrow 4}^2}{4}\,
    \NORM{\nabla\uv}{\Omega}\,
    \NORM{\nabla(\psiI-\psih)}{\Omega}\,
    \normiii{(\vvh,\qsh)}.
    \label{err15} 
  \end{align}
  Adding the estimates \eqref{err10}--\eqref{err15}, we arrive at
  \begin{align}
    \eta_6
    &\leq \Cs\hh^k\big(
    \NORM{\uv}{k+1,\Omega} +
    \max_{\P\in\Th}\NORM{\Ev}{k-1,\infty,\P}\,
    \NORM{\uv}{k+1,\Omega}\,
    \NORM{\psi}{k+1,\Omega} +
    \NORM{\psi}{2,\Omega}\,
    \NORM{\uv}{k+1,\Omega} + \NORM{\psi}{k+1,\Omega}
    \big)\,
    \normiii{(\vvh,\qsh)}
    \nonumber \\
    & \quad
    +\frac{ \East \Css{1 \hookrightarrow 4}^2}{4}\,
    \NORM{\nabla\uv}{\Omega}\,
    \NORM{\nabla(\psiI-\psih)}{\Omega}\,
    \normiii{(\vvh,\qsh)}.
    \label{err16}
  \end{align}
  
  \medskip\noindent
  \textbf{Step 5.} {Estimate of $\eta_7$.} Recalling the bound \eqref{vem-a}, the
  divergence-fee property of $\uv$ and Lemmas \ref{lemmaproj1} and
  \ref{lemmaproj2}:
  \begin{align}
    \eta_7
    &= \sum_{\P\in\Th} \delta_E \bigg(
    \int_{\P} \PizP{k-1}\nabla\cdot\uvI\PizP{k-1}\nabla\cdot\vvh\dE
    +\SP_1 \Big(\uvI - \PvnP{k}\uvI,\vvh - \PvnP{k}\vvh\Big)
    \bigg)
    \nonumber \\
    %% -----------------------------------------------------------------
    &= \sum_{\P\in\Th} \delta_E \bigg(
    \int_{\P} \PizP{k-1}(\nabla\cdot\uvI - \nabla\cdot\uv)\PizP{k-1}\nabla\cdot\vvh\dE
    +  \SP_1 \Big(\uvI - \PvnP{k}\uvI, \vvh - \PvnP{k}\vvh\Big)
    \bigg)
    \nonumber \\
    %% -----------------------------------------------------------------
    &\leq C \sum_{\P\in\Th} \delta^{1/2}_{\P}\bigg(
    \NORM{\PizP{k-1}(\nabla\cdot\uvI - \nabla\cdot\uv)}{0,\P} +
    \lambda_1^\ast \NORM{\nabla(\mathbf{I} - \PvnP{k})\uvI}{0,\P}
    \bigg)\,\normiii{(\vvh, \qsh)}_{\P}
    \nonumber \\
        %% -----------------------------------------------------------------
    &\leq \Cs\hh^k \NORM{\uv}{k+1,\Omega}\,\normiii{(\vvh, \qsh)}.
    \label{etaL-2}
  \end{align}
  Finally, combining the estimates \eqref{err1}, \eqref{err7},
  \eqref{err8}, \eqref{err16}, we obtain the estimate \eqref{verror}.
\end{proof}

\begin{lemma} \label{loadf}
  Under the hypothesis of Lemma \ref{velocity}, for
  $(\vvh,\qsh)\in\Vvh{}\times\Qsh$, the following holds
  \begin{align}
    \abs{ \Fs^{\psih}_{\PSPG{},\hh}(\vvh,\qsh) - \Fs^{\psi}_{\PSPG{}}(\vvh,\qsh) }
    &\leq \Cs\hh^k \bigg( \max_{\P\in\Th}\NORM{\Ev}{k-1,\infty,\P}
    \big(\NORM{g}{k-1,\Omega} + \NORM{\kappa}{k-1,\Omega}\big) + \NORM{\psi}{k+1,\Omega}
    \nonumber \\
    &\qquad
    \NORM{\fv}{k-1,\Omega}
    \bigg)\,\normiii{(\vvh,\qsh)}
    \nonumber \\
    & \qquad
    \Big( \Css{P}^2\East\kappa^\ast + \frac{ \Css{P}\abs{\Omega}^{1/2} \East\kappa^\ast}{4\Css{inv}} \Big)\,
    \NORM{\nabla(\psih - \psi_I)}{\Omega}\,
    \normiii{(\vvh,\qsh)}.
    \label{errf}
  \end{align}
\end{lemma}
%%

%% fin qui fin qui fin qui

\begin{proof}
  Let us recall the definition of the loads
  {$F^{\psih}_{\text{PSPG},h}(\cdot, \cdot)$ and $F^{\psi}_{\PSPG{}}(\cdot, \cdot)$},
  \begin{align}
    \eta_F :&= F^{\psih}_{\text{PSPG},h}(\vvh, \qsh) - F^{\psi}_{\PSPG{}}(\vvh, \qsh) \nonumber \\ 
    &= \sum_{\P\in\Th} \bigg( \int_{\P} \PvzP{k} \fv \cdot \boldsymbol{\Pi}^{0,\P}_{k}\vvh\dE - \int_{\P} \fv \cdot \vvh\dE + \int_{\P} \tau_{\P} \PvzP{k} \fv \cdot \PvzP{k-1}\nabla\qsh \,dE  - \int_{\P} \tau_{\P} \fv \cdot \nabla\qsh \,dE \, + \nonumber \\
      & \qquad \int_{\P}  \big(g-\kappa(\PizP{k}\psih)\big)\Ev \cdot \boldsymbol{\Pi}^{0,\P}_{k}\vvh \,dE - \int_{\P} \big(g-\kappa(\psi)\big)\Ev \cdot \vvh \,dE \, + \nonumber \\ & \qquad \int_{\P} \tau_{\P} \big(g-\kappa(\PizP{k}\psih)\big)\Ev \cdot \PvzP{k-1}\nabla\qsh \,dE - \int_{\P} \tau_{\P} \big(g-\kappa(\psi)\big)\Ev \cdot \nabla\qsh \,dE \bigg) \nonumber \\
    &=: \eta_{F,1} + \eta_{F,2} + \eta_{F,3} +\eta_{F,4}. \label{verr2}
 \end{align}
  Employing the orthogonality of the projection operators, regularity
  \textbf{(A2)} and Lemma \ref{lemmaproj1}, we infer
  \begin{align}
    \eta_{F,1} &= \sum_{\P\in\Th} \int_{\P} \big(\PvzP{k} \fv - \fv \big) \cdot \big(\vvh - \boldsymbol{\Pi}^{0,\P}_{k}\vvh \big)\dE \leq \sum_{\P\in\Th} \NORM{\PvzP{k} \fv - \fv}{\P} \NORM{\vvh - \boldsymbol{\Pi}^{0,\P}_{k}\vvh}{\P} \nonumber \\
    & \leq Ch^{k} \NORM{\fv}{k-1,\Omega} \normiii{(\vvh,\qsh)}. \label{verr3}
 \end{align}
  Following the estimation of $\eta_{F,1}$, and using the definition
  of $\tau_{\P}$ and $\normiii{\cdot}$, it holds
  \begin{align}
    \eta_{F,2} & \leq Ch^{k} \NORM{\fv}{k-1,\Omega} \normiii{(\vvh,\qsh)}. \label{verr4}
  \end{align}
  Concerning $\eta_{F,3}$, we add and subtract suitable terms, and use
  the property of the projectors and Lemma \ref{lemmaproj1}:
  \begin{align}
    \eta_{F,3} &= \sum_{\P\in\Th} \bigg( \int_{\P}  \big(g-\kappa(\PizP{k}\psih)\big)\Ev \cdot \boldsymbol{\Pi}^{0,\P}_{k}\vvh \,dE - \int_{\P} \big(g-\kappa(\psi)\big)\Ev \cdot \vvh \,dE\bigg) \nonumber \\
    &= \sum_{\P\in\Th} \bigg( \int_{\P} g \Ev \cdot \big( \boldsymbol{\Pi}^{0,\P}_{k}\vvh - \vvh \big)\,dE + \int_{\P}  \big( \kappa(\psi) - \kappa(\PizP{k}\psih) \big) \Ev \cdot  \boldsymbol{\Pi}^{0,\P}_{k}\vvh \,dE \nonumber \\
      & \qquad  + \int_{\P} \kappa(\psi)\Ev \cdot \big(\vvh - \boldsymbol{\Pi}^{0,\P}_{k}\vvh \big) \,dE\bigg) \nonumber \\
    &= \sum_{\P\in\Th} \bigg( \int_{\P} \big(g \Ev - \PvzP{k} (g \Ev) \big)\cdot \big( \boldsymbol{\Pi}^{0,\P}_{k}\vvh - \vvh \big)\,dE+ \int_{\P}  \big( \kappa(\psi) - \kappa(\PizP{k}\psih) \big) \Ev \cdot  \boldsymbol{\Pi}^{0,\P}_{k}\vvh \,dE \nonumber \\
      & \qquad  + \int_{\P} \big( \kappa(\psi)\Ev - \PvzP{k}(\kappa(\psi)\Ev) \big)\cdot \big(\vvh - \boldsymbol{\Pi}^{0,\P}_{k}\vvh \big) \,dE\bigg) \nonumber \\
    &\leq C \sum_{\P\in\Th} \bigg( h^k_E \NORM{\Ev}{k-1,\infty,\P} \NORM{g}{k-1,\P} \NORM{\nabla\vvh}{\P} + \East \kappa^\ast \NORM{\psi - \PizP{k}\psih}{\P} \NORM{\vvh}{\P} \nonumber \\
      & \qquad  +  h^k_E \NORM{\Ev}{k-1,\infty,\P} \NORM{\kappa}{k-1,\P} \NORM{\nabla\vvh}{\P}\bigg) \nonumber \\
    &\leq \Cs\hh^k \bigg( \max_{\P\in\Th} \NORM{\Ev}{k-1,\infty,\P} \big(\NORM{g}{k-1,\Omega}  + \NORM{\kappa}{k-1,\Omega}\big) + \NORM{\psi}{k+1,\Omega} \bigg)  \normiii{(\vvh,\qsh)} \nonumber \\
    & \qquad  + \Css{P}^2 \East \kappa^\ast \NORM{\nabla(\psih - \psi_I)}{\Omega} \normiii{(\vvh,\qsh)}. \label{verr5}
  \end{align}
  Following the estimation of $\eta_{F,3}$ and using the definition of
  $\tau_{\P}$, we arrive at
  \begin{align}
    \eta_{F,4}	&\leq \Cs\hh^k \bigg( \max_{\P\in\Th} \NORM{\Ev}{k-1,\infty,\P} \big(\NORM{g}{k,\Omega}  + \NORM{\kappa}{k-1,\Omega}\big) + \NORM{\psi}{k+1,\Omega} \bigg)  \normiii{(\vvh,\qsh)} \nonumber \\
    & \qquad  + \frac{ \Css{P} |\Omega|^{1/2} \East \kappa^\ast}{4 \Css{inv}} \NORM{\nabla(\psih - \psi_I)}{\Omega} \normiii{(\vvh,\qsh)}, \label{verr6}
  \end{align}
  which completes the proof.
  %% \qquad \qquad $\blacksquare$
\end{proof}

%%%%%%%%%%%%%%%%%%%%%%%%%%%%%%%%%%%%%%%%%%%%%%%%%%%%%%%%%%%%%%%%%%%%%%%%%%%%%%%%%%%%%%%%%%%%%%%%%%%%%%%%%%%%
% TEMPERATURE LEMMA
%%%%%%%%%%%%%%%%%%%%%%%%%%%%%%%%%%%%%%%%%%%%%%%%%%%%%%%%%%%%%%%%%%%%%%%%%%%%%%%%%%%%%%%%%%%%%%%%%%%%%%%%%%%%

\begin{lemma}\label{temp}
  Let the assumption \textbf{(A2)} and the hypothesis of Theorems
  \ref{wellp} and \ref{dwell2} hold true. For any given $\uv
 \in\Vv_{div}$ and $\uvh\in\Vvh{}$, let $\psi
 \in\mathlarger \Phi$ and $\psih\in\mathlarger \phih$ be the
  solution of problems \eqref{variation2} and \eqref{dsvem2},
  respectively. Then, the following holds
  \begin{align}
    \NORM{\nabla(\psi_I -\psih)}{\Omega}
    & \leq \Cs\hh^k \bigg(
    \NORM{\psi}{k+1,\Omega} +
    \NORM{\uv}{k+1,\Omega} +
    \NORM{\psi}{k+1,\Omega}\,\NORM{\uv}{k,\Omega} +
    \NORM{\kappa}{k-1,\Omega} +
    \NORM{g}{k-1,\Omega}
    \bigg)
    \nonumber \\
    & \qquad
    +\frac{\Css{1\hookrightarrow 4}^2}{\gamma_\ast\epsilon}
    \NORM{\nabla\psi}{\Omega}\,
    \NORM{\nabla(\uvI - \uvh)}{\Omega}.
    \label{errtemp}
  \end{align}
\end{lemma}

\begin{proof}
  We first set $\phihTld:=\psi_I -\psih$, and let us recall
  the stability properties of the bilinear form $\ass{p}(\cdot,\cdot)$ and
  $\asph(\cdot,\cdot)$ and the strongly monotone property of
  $\kappa(\cdot)$, we infer
  \begin{align}
    \gamma_\ast \epsilon \NORM{\nabla(\psi_I -\psih)}{\Omega}^2
    &\leq
    \calBh(\uvh;\psih,\phihTld) - \calB(\uvh;\psiI,\phihTld)
    \nonumber \\
    &=
    (\gsh,\phihTld) - (\gs,\phihTld) +
    \calB(\uv;\psi,\phihTld) - \calB(\uvh;\psiI,\phihTld)
    \nonumber \\
    & =
    (\gsh,\phihTld) - (\gs,\phihTld) +
    \ass{p}(\psi,\phihTld)
    - \asph(\psiI,\phihTld)
    + \cs^{skew}_{p}(\uv;\psi,\phihTld)
    - \cs^{skew}_{p,h}(\uvh;\psiI,\phihTld)
    \nonumber \\
    & \qquad
    +
   { \ds(\psi,\phihTld) -\ds_h(\psiI,\phihTld)}
    \nonumber \\
    &=: \bss{0} + \bss{1} + \bss{2} + \bss{3}.
    \label{perr1}
  \end{align}

  \medskip\noindent
  $\bullet$ Employing the orthogonality of $\PizP{k}$, the assumption
  \textbf{(A2)} and Lemma \ref{lemmaproj1}, we have
  \begin{align}
    \bss{0}
    \leq \sum_{\P\in\Th}
    \NORM{\gs - \PizP{k}\gs}{\P}\,
    \NORM{\phihTld - \PizP{k}\phihTld}{\P}
    \leq  \Cs\hh^k
    \NORM{\gs}{k-1,\Omega}\,
    \NORM{\nabla\phihTld}{\Omega}.
    \label{perr0}
  \end{align}

  \medskip\noindent
  $\bullet$ We use the triangle inequality, the property of
  $L^2$-projectors, the bound \eqref{vem-c} and Lemmas
  \ref{lemmaproj1} and \ref{lemmaproj2}:
  \begin{align}
    \bss{1}
    &=
    \sum_{\P\in\Th} \epsilon \bigg(
    \int_{\P} \nabla\psi\cdot\nabla\phihTld\dE -
    \int_{\P} \PvzP{k-1}\nabla\psiI\cdot\PvzP{k-1}\nabla\phihTld\dE -
    \SP_3 \big( \psiI - \Pi^{\nabla,\P}_k\psiI, \phihTld - \PinP{k}\phihTld\big)
    \bigg)
    \nonumber \\
    &=
    \sum_{\P\in\Th} \epsilon \bigg(
    \int_{\P} \big(\nabla\psi - \PvzP{k-1}\nabla\psiI \big)\cdot\PvzP{k-1}\nabla\phihTld\dE -
    \SP_3 \big( \psiI - \PinP{k}\psiI, \phihTld - \PinP{k}\phihTld\big)
    \bigg)
    \nonumber \\
    &\leq \Cs \sum_{\P\in\Th} \epsilon\bigg(
    \hP^k\NORM{\psi}{k+1,\P}\,\NORM{\nabla\phihTld}{\P} +
    \lambda_3^\ast \hP^k\NORM{\psi}{k+1,\P}\,\NORM{\nabla\phihTld}{\P}
    \bigg)
    \nonumber \\
    &\leq \Cs\hh^k
    \NORM{\psi}{k+1,\Omega}\,
    \NORM{\nabla\phihTld}{\Omega}.
    \label{perr2}
  \end{align}

  \medskip\noindent
  $\bullet$ Concerning $\bss{2}$, we proceed as follows
  \begin{align}
    \bss{2}
    &= \cs^{skew}_p(\uv;\psi,\phihTld) - \cs^{skew}_{p,h}(\uvh;\psiI,\phihTld)
    \nonumber\\
    &= \cs^{skew}_p(\uv;\psi,\phihTld) - \cs^{skew}_{p,h}(\uv;\psi,\phihTld)
    + \cs^{skew}_{p,h}(\uv-\uvh;\psi,\phihTld) + \cs^{skew}_{p,h}(\uvh;\psi-\psiI,\phihTld)
    \nonumber\\
    &=: \bss{2,1} + \bss{2,2} + \bss{2,3}.
    \label{perr3}
  \end{align}
  Following \cite[Lemma 4.3]{mvem9}, there holds that
  \begin{align}
    \abs{\bss{2,1}}
    &\leq \Cs\hh^k \bigg(
    \NORM{\uv}{k,\Omega}\,\NORM{\psi}{k+1,\Omega} +
    \NORM{\nabla\uv}{\Omega}\,\NORM{\psi}{k+1,\Omega} +
    \NORM{\uv}{k+1,\Omega}\,\NORM{\nabla\psi}{\Omega}
    \bigg) \NORM{\nabla\phihTld}{\Omega}
    \nonumber \\
    &\leq \Cs\hh^k \bigg(
    \NORM{\uv}{k,\Omega}\,\NORM{\psi}{k+1,\Omega} +
    \NORM{\psi}{k+1,\Omega} +
    \NORM{\uv}{k+1,\Omega}
    \bigg)\,\NORM{\nabla \phihTld}{\Omega}.
    \label{perr4}
  \end{align}
  We use the Sobolev embedding theorem and the estimation of
  $\eta_{43}$:
  \begin{align}
    \abs{\bss{2,2}}
    &\leq \Css{1\hookrightarrow 4}^2
    \NORM{\nabla(\uv-\uvh}{\Omega}\,
    \NORM{\nabla\psi}{\Omega}\,
    \NORM{\nabla\phihTld}{\Omega}.
    \label{perr6}
  \end{align}
  Following the estimation of $\eta_{43}$, we get
  \begin{align}
    \abs{\bss{2,3}}
    &\leq \Cs\hh^k
    \NORM{\nabla\uvh}{\Omega}\,
    \NORM{\psi}{k+1,\Omega}\,
    \NORM{\nabla\phihTld}{\Omega}
    \nonumber \\
    &\leq \Cs\hh^k
    \NORM{\psi}{k+1,\Omega}\,
    \NORM{\nabla\phihTld}{\Omega},
    \label{perr5}
  \end{align}
  where the last line is obtained using the boundedness of
  $\uvh$ in the $\HONE$-semi norm.

  \medskip\noindent
  Combining the estimates \eqref{perr4}--\eqref{perr6} and applying
  Lemma \ref{lemmaproj2}, we obtain the following
  \begin{align}
    \bss{2}
    &\leq \Cs\hh^k \bigg(
    \big(1+\NORM{\uv}{k,\Omega}\big)\,\NORM{\psi}{k+1,\Omega} +
    \NORM{\uv}{k+1,\Omega}
    \bigg)\,
    \NORM{\nabla\phihTld}{\Omega}
    \nonumber \\
    & \qquad
    +\Css{1 \hookrightarrow 4}^2
    \NORM{\nabla(\uvI-\uvh}{\Omega}\,
    \NORM{\nabla\psi}{\Omega}\,
    \NORM{\nabla\phihTld}{\Omega}.
    \label{perr7}
  \end{align}

  \medskip\noindent
  $\bullet$ Adding and subtracting appropriate terms and using the
  orthogonality property of $\Pi^{0, E}_k$, we infer
  \begin{align}
    \bss{3}
    & = \sum_{\P\in\Th} \bigg(
    \int_{\P} \kappa(\psi)\phihTld\dE -
    \int_{\P} \kappa(\PizP{k}\psiI)\PizP{k}\phihTld\dE
    \bigg)
    \nonumber \\
    & = \sum_{\P\in\Th} \bigg(
    \int_{\P} \big(\kappa(\psi)-\PizP{k}\kappa(\psi)\big)\,\big(\phihTld-\PizP{k}\phihTld\big)\dE -
    \int_{\P} \big(\kappa(\psi)-\kappa(\PizP{k}\psiI)\big)\PizP{k}\phihTld\dE
    \bigg)
    \nonumber \\	
    &\leq \Cs\sum_{\P\in\Th} \bigg(
    \hP^{k}\NORM{\kappa(\psi)}{k-1,\P}\,\NORM{\nabla \phihTld}{\P} +
    \kappa^\ast\NORM{\psi-\PizP{k}\psiI}{\P}\,\NORM{\phihTld}{\P}
    \bigg)
    \nonumber \\		
    & \leq \Cs\hh^k\bigg(
    \NORM{\kappa}{k-1,\Omega} +
    \NORM{\psi}{k+1,\Omega}
    \bigg)\,\NORM{\nabla\phihTld}{\Omega}.
    \label{perr8}
  \end{align}
  Finally, combining the estimates \eqref{perr0}, \eqref{perr2},
  \eqref{perr7} and \eqref{perr8}, we readily arrive at the result
  \eqref{errtemp}.
\end{proof}

%%%%%%%%%%%%%%%%%%%%%%%%%%%%%%%%%%%%%%%%%%%%%%%%%%%%%%%%%%%%%%%%%%%%%%%%%%%%%%%%%%%%%%%%%%%%%%%%%%%%%%%%%%%%%
%                            Convergence study
%%%%%%%%%%%%%%%%%%%%%%%%%%%%%%%%%%%%%%%%%%%%%%%%%%%%%%%%%%%%%%%%%%%%%%%%%%%%%%%%%%%%%%%%%%%%%%%%%%%%%%%%%%%%%

Hereafter, we present the convergence estimate in the energy norm:

\begin{proposition}
  \label{convergence}
  Under the assumption \textbf{(A0)} and the assumptions of
  Proposition \ref{uni1} and Proposition \ref{uni2}, let
  $(\uv,\ps,\psi)\in\Vv\times\Qs\times\mathlarger\Phi$ and
  $(\uvh,\psh,\psih)\in\Vvh{}\times\Qsh\times{\mathlarger \Phi_h}$ be the
  solution of the problems \eqref{variation1} and \eqref{svem}.
  Furthermore, we assume that the stabilization parameters are such
  that the assumption \eqref{stabc2} holds true.
  We assume that data of problem~\eqref{P} are such that
  \begin{align}	
    \beta_\ast -
    \frac{\mu}{2\gamma_\ast\epsilon}\Big(
    \frac{5\Css{1 \hookrightarrow 4}^2}{4} \Css{u} +
    \Css{P}^2\kappa^\ast +
    \frac{\Css{P}\abs{\Omega}^{1/2}\kappa^\ast}{4\Css{inv}}
    \Big) > 0.
    \label{cgtcon}
  \end{align}
  Then under the assumption \textbf{(A2)}, we have the following error
  estimate
  \begin{align}
    \normiii{(\uv-\uvh,\ps-\psh)} + \NORM{\nabla(\theta-\theta_h)}{\Omega}
    &\leq \Cs\hh^k\bigg(
    \NORM{\uv}{k+1,\Omega}  +
    \NORM{\ps}{k,\Omega}    +
    \NORM{\psi}{k+1,\Omega} +
    \NORM{\uv}{k+1,\Omega}\,\NORM{\psi}{k+1,\Omega}
    \nonumber \\
    & \qquad
    + \NORM{\psi}{2,\Omega}\,
    \NORM{\uv}{k+1,\Omega}\,
    \NORM{\gs}{k-1,\Omega}
    + \NORM{\kappa}{k-1,\Omega}
    + \NORM{\psi}{k+1,\Omega}
    + \NORM{\fv}{k-1,\Omega}
    \bigg),
    \label{converge0}
  \end{align}
  where the positive constant $C$ additionally depends on $\Ev$.
\end{proposition}
\begin{proof}
  Using the discrete inf-sup condition \eqref{d1well1}, it holds
  \begin{align}
    \beta_\ast \normiii{(\uvh-\uvI,\psh-\psI)}\,\normiii{(\vvh,\qsh)}
    &
    \leq \calA_{\PSPG{},\hh}\big(\psih;(\uvh-\uvI,\psh-\psI),(\vvh,\qsh)\big)
    \nonumber \\
    & =
    \Fs^{\psih}_{\PSPG{},h}(\vvh,\qsh) -
    \Fs^{\psi}_{\PSPG{}}(\vvh,\qsh)
    \nonumber \\
    & \quad
    + \calA_{\PSPG{}}\big(\psi;(\uv,\ps),(\vvh,\qsh)\big)
    - \calA_{\PSPG{},\hh}\big(\psih;(\uvI,\psI ),(\vvh,\qsh) \big)
    \nonumber \\
    & =: \eta_F + \eta_A. \label{cerr1}
  \end{align}
  Recalling Lemmas \ref{velocity} and \ref{loadf}, we have
  \begin{align}
    \beta_\ast \normiii{(\uvh-\uvI,\psh-\psI)}
    &\leq \Cs\hh^k \bigg(
    \NORM{\uv}{k+1,\Omega} +
    \NORM{\ps}{k,\Omega} +
    \NORM{\psi}{k+1,\Omega} +
    \NORM{\uv}{k+1,\Omega}\,\NORM{\psi}{k+1,\Omega} +
    \NORM{\psi}{2,\Omega}\,\NORM{\uv}{k+1,\Omega}
    \nonumber \\
    & \qquad
    \NORM{\gs}{k-1,\Omega} +
    \NORM{\kappa}{k-1,\Omega} +
    \NORM{\psi}{k+1,\Omega} +
    \NORM{\fv}{k-1,\Omega}
    \bigg)
    \nonumber \\
    & \qquad +
    \Big( \frac{5\East\Css{1 \hookrightarrow 4}^2}{4}\,\NORM{\nabla\uv}{\Omega} +
    \Css{P}^2 \East\kappa^\ast +
    \frac{\Css{P}\abs{\Omega}^{1/2}\East\kappa^\ast}{4\Css{inv}}
    \Big)\NORM{\nabla(\psih-\psiI)}{\Omega}.
  \end{align}
  Employing Lemma \ref{temp}, the Sobolev embedding theorem and the
  assumption \eqref{cgtcon}, yields
  \begin{align}
    \normiii{(\uvh-\uvI, \psh-\psI)}
    &\leq \Cs\hh^k \bigg(
    \NORM{\uv}{k+1,\Omega} +
    \NORM{\ps}{k,\Omega} +
    \NORM{\psi}{k+1,\Omega} +
    \NORM{\uv}{k+1,\Omega}\,\NORM{\psi}{k+1,\Omega} +
    \NORM{\psi}{2,\Omega}\,\NORM{\uv}{k+1,\Omega}
    \nonumber \\
    & \qquad
    \NORM{\gs}{k-1,\Omega} +
    \NORM{\kappa}{k-1,\Omega} +
    \NORM{\psi}{k+1,\Omega} +
    \NORM{\fv}{k-1,\Omega}
    \bigg).
    \label{cerr2}
\end{align}
Thus, the estimate \eqref{converge0} can be easily obtained by
combining the estimates \eqref{cerr2} and \eqref{errtemp} and the
interpolant estimate \eqref{projlll2}.
\end{proof}

%% \subsection{\MGT{Discussion of the smallness condition}}
\subsection{Discussion of the smallness condition}
\label{sec:smallness}

The convergence estimate of Proposition~\ref{convergence} requires
condition~\eqref{cgtcon}, which demands that the discrete inf-sup
constant $\beta_\ast$ be strictly larger than a quantity that depends
on the viscosity $\mu$, the permittivity $\epsilon$, the nonlinearity
parameter $\kappa^\ast$, the velocity bound $\Css{u}$, the Poincar\'e
constant $\Css{P}$, the Sobolev embedding constant
$\Css{1\hookrightarrow 4}$, and the inverse inequality constant
$\Css{inv}$.
In this subsection, we discuss the origin, the practical
restrictiveness, and the indirect numerical verification of this
condition.

\medskip
%% Origin
The error analysis of the coupled Stokes-Poisson-Boltzmann system
introduces a cross-coupling between the velocity-pressure error and
the potential error.
Specifically, the Stokes error bound (cf.~the estimate preceding
\eqref{cerr2}) contains a term proportional to $\NORM{\nabla(\psih -
  \psiI)}{\Omega}$, arising from the dependence of the Stokes bilinear
form on the potential field.
Through Lemma~\ref{temp}, this potential error in turn depends on the
velocity error $\NORM{\nabla(\uvI - \uvh)}{\Omega}$, creating a
feedback loop.
Condition~\eqref{cgtcon} ensures that the coupling coefficient is
strictly smaller than $\beta_\ast$, so that the feedback can be
resolved and the velocity-pressure error can be bounded purely in
terms of interpolation errors.

\medskip
%% Context in the literature
Smallness conditions of this type are a standard feature of the
analysis of nonlinearly coupled systems.
In the finite element analysis of the same Stokes-Poisson-Boltzmann
model, AlSohaim et al.~\cite{alsohaim2025analysis} require three
analogous conditions~\cite[conditions (2.11), (2.12),
  and~(3.3)]{alsohaim2025analysis}, involving generic Sobolev and
stability constants that are likewise not explicitly computable in
closed form.
Similar assumptions appear in the analysis of stabilized methods for
related coupled problems~\cite{mishra2024unified,mishra2025equal}.
In all these references, the conditions are stated as abstract
inequalities and verified only indirectly through the numerical
experiments.

\medskip
%% Qualitative interpretation
Although condition~\eqref{cgtcon} cannot be verified algebraically for
specific parameter values (due to the presence of generic constants
such as $\beta_\ast$ and $\Css{inv}$), its dependence on the physical
parameters can be interpreted qualitatively.
The subtracted quantity in~\eqref{cgtcon} scales as
$\mu/(\gamma_\ast\epsilon)$ times a combination of $\Css{u}$ and
$\kappa^\ast$.
Therefore, the condition becomes more restrictive when:
\begin{description}[nosep]
\item $(i)$ the ratio $\mu/\epsilon$ is large, corresponding to regimes where
viscous effects dominate over electrostatic screening;
\item $(ii)$ the nonlinearity strength $\kappa^\ast$ is large, which occurs
when the ionic activity coefficients $\alpha_0$ and $\alpha_1$ produce
strong charge densities;
\item $(iii)$ the velocity bound $\Css{u}$ is large, which by~\eqref{velv}
happens when the external forces $\fv$ and the source term $\gs$ are
large relative to the viscosity.
\end{description}
Conversely, the condition is more easily satisfied for moderate
physical parameters and small data.

\medskip
%% Indirect numerical verification
For the parameter regimes considered in the numerical experiments of
Section~\ref{sec-6}, several factors contribute to the expectation
that condition~\eqref{cgtcon} is satisfied with a comfortable margin.
In Example~\ref{case1}, the moderate parameter values $\mu = 1$,
$\kappa = 1$, $\alpha_0 = \alpha_1 = 1$, and $\Ev = [0,-1]^T$ yield a
regime where the coupling effects are mild.
In Example~\ref{case3}, although $\mu/\epsilon = 0.1/0.075 \approx
1.33$ is not particularly small, the very small ionic activity
coefficient $\alpha_0 = 0.001$ produces a correspondingly small
nonlinearity parameter $\kappa^\ast$, significantly reducing the
subtracted quantity in~\eqref{cgtcon}.
The numerical experiments provide two forms of indirect evidence that
the condition holds.
First, the fixed-point iteration converges consistently within
$7$--$8$ iterations across all mesh sizes and polynomial orders $k=1$
and $k=2$, with no sign of deterioration under mesh refinement.
Second, the computed convergence rates closely match the theoretical
prediction $\mathcal{O}(h^k)$ on all mesh types, including non-convex
polygons and meshes with hanging nodes.
If condition~\eqref{cgtcon} were violated or only marginally
satisfied, one would expect either divergence of the fixed-point
scheme, degradation of convergence rates on finer meshes, or
sensitivity to the mesh type, none of which is experimentally
observed.

% Local Variables:
% mode: latex
% End:
%% Hey Emacs, this is -*-latex-*-

\section{Numerical results}
\label{sec-5}

Hereafter, we investigate the practical behavior of the proposed
method \eqref{svem} through three numerical examples that validate the
theoretical estimates derived in Section \ref{theo}.
In these test cases, we consider convex and non-convex domains,
including meshes with hanging nodes, demonstrating the geometric
flexibility of the proposed method.
%%
%% We also demonstrate the flexibility of the proposed method by working
%% with meshes with hanging nodes.
%% %%
%% \begin{figure}[h]
%%   \centering
%%   \subfloat[$\Omega_1$]{\includegraphics[height=3cm, width=3cm]{./hex5.png}}~~~~~
%%   \subfloat[$\Omega_2$]{\includegraphics[height=3cm, width=3cm]{./ncv_5.png}}
%%   \subfloat[$\Omega_3$]{\includegraphics[height=3cm, width=4.5cm]{./Lshaped.png}}\\
%%   \subfloat[$\Omega_4$]{\includegraphics[height=4cm, width=4cm]{./mix15.png}}~~
%%   \caption{Convex and non-convex domains with general polygons including hanging nodes $\Omega_4$}
%%   \label{samp} 
%% \end{figure} 

\begin{figure}[h]
  \centering
  \subfloat[$\Omega_1$]{%
    \includegraphics[height=3cm, width=3cm]{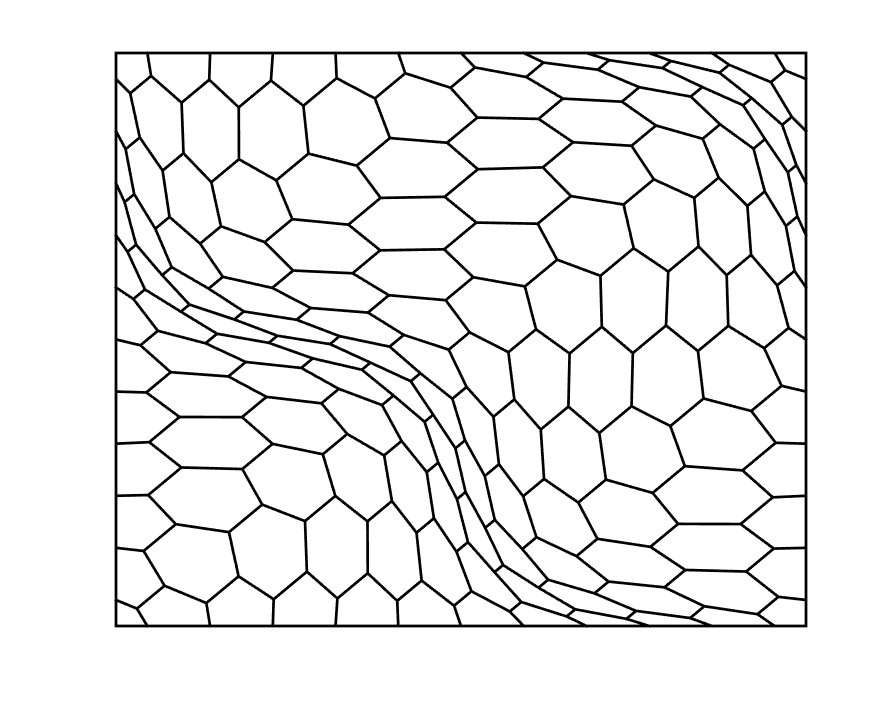}
  }\hspace{0.4cm}
  \subfloat[$\Omega_2$]{%
    \includegraphics[height=3cm, width=3cm]{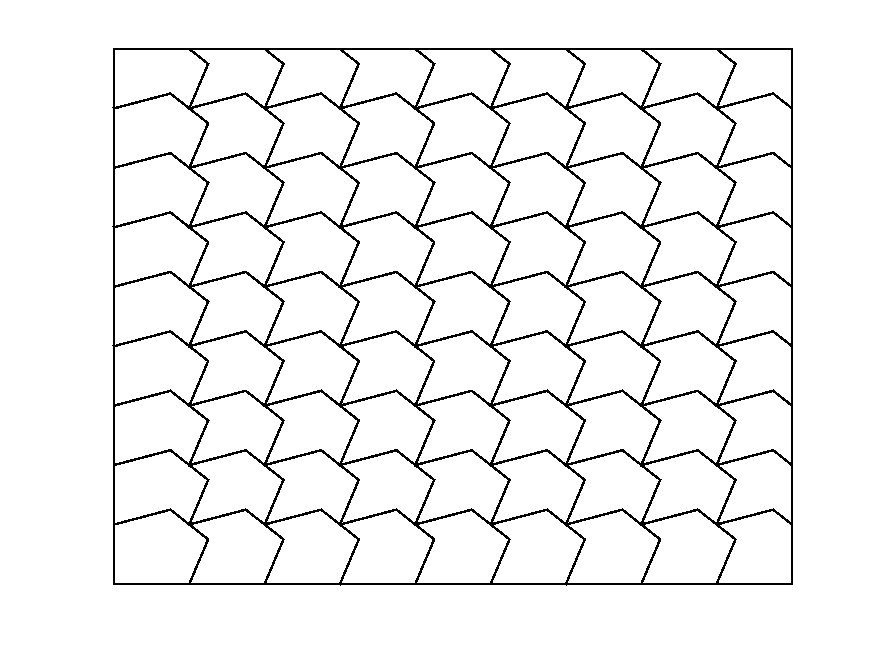}
  }%\hspace{0.05cm}
  \subfloat[$\Omega_3$]{%
    \includegraphics[height=3cm, width=4.cm]{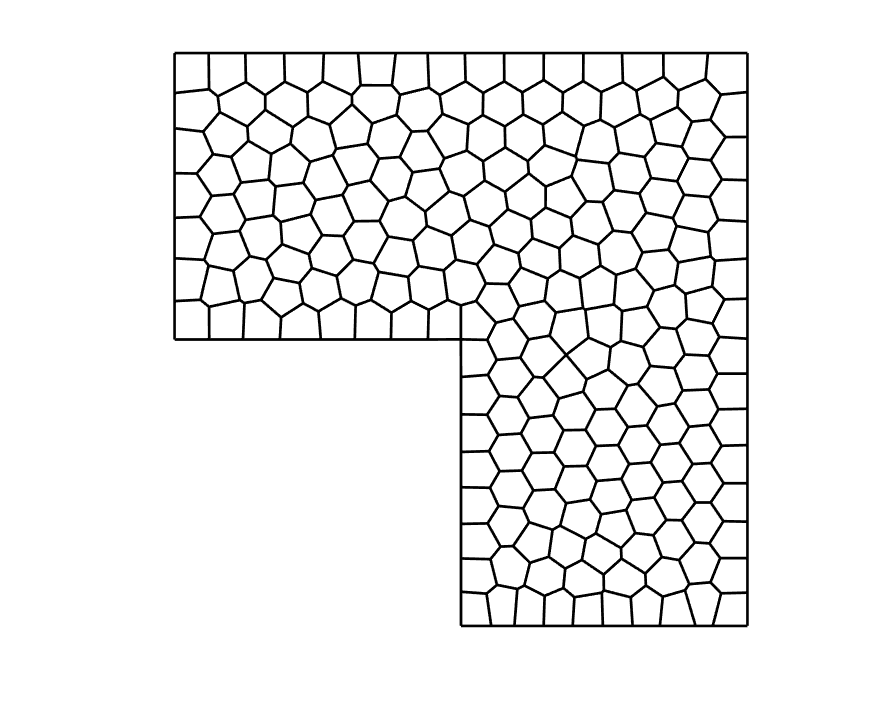}
  }%\hspace{0.05cm}
  \subfloat[$\Omega_4$]{%
    \includegraphics[height=3cm, width=3cm]{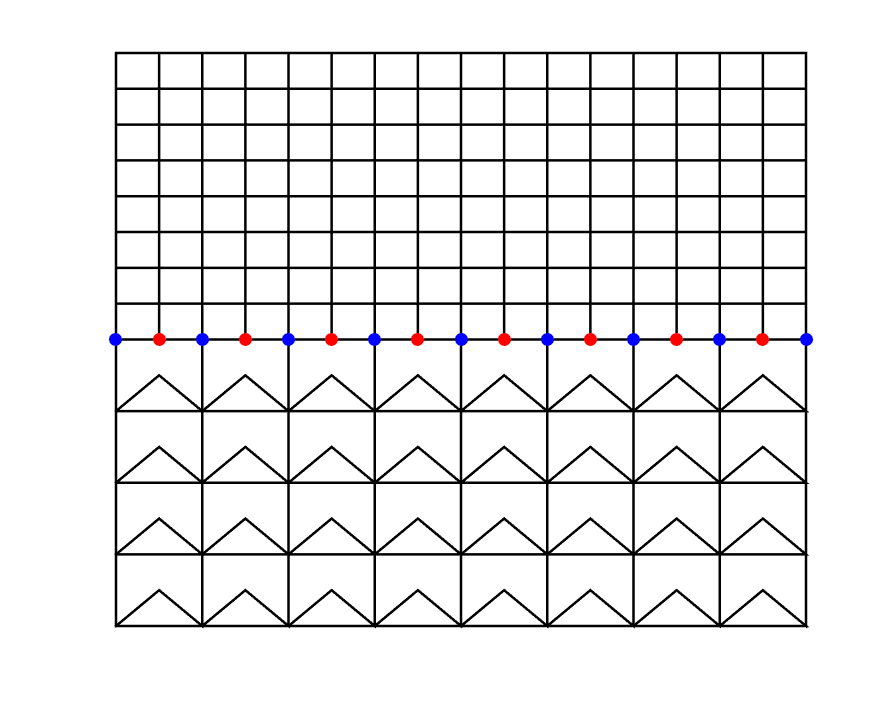}
  }
  \caption{Convex and non-convex domains with general polygons including hanging nodes $\Omega_4$.}
  \label{samp}
\end{figure}

%% \subsection{Preliminaries}
%% \label{sec-error}
%%
%% We also demonstrate the flexibility of the proposed method by working
%% with meshes with hanging nodes.
%%
All computations employ VEM orders $k=1$ and $k=2$ to confirm optimal
convergence for different polynomial degrees, and a fixed-point
iterative scheme for solving the nonlinear stabilized virtual element
problem \eqref{svem} with stopping threshold equal to $10^{-6}$.
%%
%% We present numerical results that validate our theoretical estimate
%% derived in Section \ref{theo} for convex and non-convex domains,
%% including the meshes having hanging nodes for VEM order $k=1$ and
%% $k=2$.
%%
Following Lemma \ref{dwell1}, stabilization parameters are chosen as
$\tau_{E} \sim h^2_E$ and $\delta_{E} \sim h_E$.
For domains $\Omega_1$, $\Omega_2$, and $\Omega_4$ on the unit square,
we employ mesh sizes $\hh=1\slash{\NMB}$, with $\NMB\in\{5, 10, 20,
40, 80\}$, while the L-shaped domain $\Omega_3$ uses the sequence
$\NMB\in\{4, 8, 16, 32, 64\}$.
A sample of computational meshes is shown in Figure \ref{samp}.
%%
%% We utilize the following mesh sizes:
%% %%
%% \begin{itemize}
%% \item For $\Omega_1$, $\Omega_2$, and $\Omega_4$, we use $h:=1/5,
%%   1/10, 1/20, 1/40, 1/80$;
%% \item For $\Omega_3$, we use $h:=1/4, 1/8, 1/16, 1/32, 1/64$;
%% \end{itemize}
%% %%
%% Following Lemma \ref{dwell1}, the stabilization parameters are chosen
%% as follows: $\tau_{E} \sim h^2_E$ and $\delta_{E} \sim h_E$.

We quantify the absolute errors in the $H^1(\Omega)$ semi-norm for the
velocity field and the $L^2$-norm for the pressure field as follows:
\begin{align}
  E^{\mathbf{u}}_{H^1}
  := {\sqrt{\sum\limits_{E\in\Omega_h}  \| \nabla (\mathbf{u} - \boldsymbol{\Pi}^{\nabla,E}_k \mathbf{u}_h ) \|_E^2}}, \qquad \qquad  E^p_{L^2} :=\sqrt{\sum\limits_{E\in\Omega_h} \| p- {\Pi}^{0,E}_k p_h\|^2_E}, \nonumber 
\end{align}
where the exact and discrete velocities are denoted by $\mathbf{u}$
and $\mathbf{u}_h$, and $p$ and $p_h$ for exact and discrete pressure
field, and $\boldsymbol{\Pi}^{\nabla,E}_k$ and
$\boldsymbol{\Pi}^{0,E}_k$ are the elliptic and $L^2$-orthogonal
projectors discussed in Section~\ref{subsec-32:projectors}.
Following the above notations, we use $E^{\psi}_{H^1}$ to represent
{the absolute error in the $H^1$ semi-norm for the potential field.}

%% SECTION 5.1
\subsection{Example 1: Convergence studies on convex domains}
\label{case1}

In this first example, we consider a manufactured solution with smooth
velocity, pressure, and potential fields to validate optimal
convergence rates on various polygonal meshes.
The parameters defining problem \ref{P} are:
%%
%% In this example, we choose the following parameters for
%% Problem~\ref{P}:
%%
\begin{align}
  \mu = 1,\qquad\kappa=1,\qquad\alpha_0=\alpha_1=1,\qquad\mathbf{E}=[0,-1]^T.
  \nonumber 
\end{align}
Let $\xi(x,y):= x^3y^3(1-x)^3(1-y)^3$ define an auxiliary function.
Then, we consider the exact solution fields:
\begin{empheq}[left= \empheqlbrace]{align}
  \uv (x,y) &= {\Big[\frac{d\xi}{dy},  -\frac{d\xi}{dx} \Big]^T},\nonumber \\
  \ps (x,y) &= \sin(\pi x) \cos(\pi x) +p_0, \nonumber\\
  \psi(x,y) &= x^2 y^2(x-1)(y-1),  \nonumber 
\end{empheq}
where $p_0$ is the pressure constant ensuring the zero-mean pressure
constraint according to~\eqref{heat-2}.
The load terms and the boundary conditions are derived accordingly.

%% then the loads given
%% by the exact solution of problem $(P)$:
%% %%
%% \begin{empheq}[left= \empheqlbrace]{align}
%%   \mathbf{u}(x,y) &=\Big[\frac{d\xi}{dy},  -\frac{d\xi}{dx} \Big],\nonumber \\
%%   p(x,y) &= \sin(\pi x) \cos(\pi x) +p_0, \nonumber\\
%%   \phi(x,y) &= x^2 y^2(x-1)(y-1),  \nonumber 
%% \end{empheq}
%% %%
%% where $p_0$ is a constant obtained by \eqref{heat-2}.

Tables \ref{ex1_tb1}--\ref{ex1_tb3} present numerical errors and
convergence rates on the unit square using computational meshes
$\Omega_1$ (distorted hexagons), $\Omega_2$ (non-convex polygons), and
$\Omega_4$ (composite mesh with non-convex polygons and hanging
nodes).
%%
%% We present our numerical results on the unit square using the
%% computational meshes $\Omega_1$, $\Omega_2$, and $\Omega_4$.
%% %%
%% The numerical errors and convergence rates for each field are shown in
%% Tables \ref{ex1_tb1}--\ref{ex1_tb3} for VEM order $k=1$ and $k=2$.
%%
These numerical results confirm that the stabilized VEM achieves the
optimal theoretical convergence rates derived in Section \ref{theo},
which are $\mathcal{O}(h)$ for $k=1$ and $\mathcal{O}(h^2)$ for $k=2$
for the three solution fields, i.e., velocity, pressure and potential.
%% for every solution field for VEM order $k=1$ and $k=2$, respectively.

Additionally, we remark that the meshes of the mesh family $\Omega_4$
are constructed through the composition of two non-matching meshes,
thus leading to hanging nodes at the shared interface.
Table \ref{ex1_tb3} shows that optimal convergence is maintained
without special treatment of these hanging nodes, confirming that this
method naturally handles possible mesh non-conformities.
%%
%% Additionally, we remark that to demonstrate the property of the
%% proposed method, we have investigated the performance of the proposed
%% method on the composition of two non-matching meshes, which leads to
%% hanging nodes.
%%
%% The obtained results are shown in Table \ref{ex1_tb3}, validating the
%% theoretical estimate. Thus, the proposed method easily incorporates
%% the hanging nodes.

\begin{table}[h]
  \centering
  \caption{[Example 1] Convergence studies on distorted hexagons
    $\Omega_1$ using proposed method}
  \label{ex1_tb1}
  \begin{tabular*}{.9\linewidth}{@{}lllllllll@{}}
    \toprule
    order & $h$ & $E^\mathbf{u}_{H^1}$ & rate & $E^p_{L^2}$ & rate & $E^{\psi}_{H^1}$ & rate & {itr} \\
    \midrule
    & 1/5  & 1.311720e-02 & --    & 2.449411e-02 & --    & 8.705857e-03 & --    & 7 \\
    & 1/10 & 6.428537e-03 & 1.029 & 1.221762e-02 & 1.004 & 4.577561e-03 & 0.927 & 7 \\
    {$k=1$} & 1/20 & 3.232437e-03 & 0.992 & 6.103906e-03 & 1.001 & 2.345729e-03 & 0.965 & 7 \\
    & 1/40 & 1.623956e-03 & 0.993 & 3.052641e-03 & 1.000 & 1.187052e-03 & 0.983 & 7 \\
    & 1/80 & 8.136887e-04 & 0.997 & 1.527061e-03 & 0.999 & 5.970562e-04 & 0.991 & 8 \\
    \midrule
    & 1/5  & 2.681120e-02 & --    & 4.503502e-02 & --    & 8.103336e-04 & --    & 7 \\
    & 1/10 & 6.049456e-03 & 2.418 & 1.181225e-02 & 1.931 & 2.221737e-04 & 1.867 & 7 \\
    {$k=2$} & 1/20 & 1.533453e-03 & 1.980 & 3.019989e-03 & 1.968 & 5.772350e-05 & 1.945 & 7 \\
    & 1/40 & 3.893513e-04 & 1.978 & 7.607732e-04 & 1.989 & 1.468018e-05 & 1.975 & 7 \\
    & 1/80 & 9.810407e-05 & 1.989 & 1.910031e-04 & 1.994 & 3.699626e-06 & 1.988 & 8 \\
    \bottomrule
  \end{tabular*}
\end{table}

\begin{table}[h]
  \centering
  \caption{[Example 1] Convergence studies on non-convex mesh
    $\Omega_2$ using proposed method}
  \label{ex1_tb2}
  \begin{tabular*}{.9\linewidth}{@{}lllllllll@{}}
    \toprule
    order & $h$ & $E^\mathbf{u}_{H^1}$ & rate & $E^p_{L^2}$ & rate & $E^{\psi}_{H^1}$ & rate & {itr} \\
    \midrule
    & 1/5  & 2.002868e-02 & --    & 3.248449e-02 & --    & 1.051645e-02 & --    & 7 \\
    & 1/10 & 8.161245e-03 & 1.295 & 1.427017e-02 & 1.187 & 5.393170e-03 & 0.963 & 7 \\
    {$k=1$} & 1/20 & 3.621136e-03 & 1.172 & 6.640047e-03 & 1.106 & 2.730142e-03 & 0.982 & 7 \\
    & 1/40 & 1.722770e-03 & 1.072 & 3.201921e-03 & 1.052 & 1.373356e-03 & 0.991 & 7 \\
    & 1/80 & 8.462576e-04 & 1.026 & 1.527061e-03 & 1.024 & 6.887283e-04 & 0.996 & 8 \\
    \midrule
    & 1/5  & 1.434075e-02 & --    & 2.543417e-02 & --    & 1.241960e-03 & --    & 7 \\
    & 1/10 & 3.361010e-03 & 2.093 & 6.199488e-03 & 2.034 & 3.220434e-04 & 1.947 & 7 \\
    {$k=2$} & 1/20 & 8.330459e-04 & 2.012 & 1.548879e-03 & 2.001 & 8.184041e-05 & 1.976 & 7 \\
    & 1/40 & 2.077321e-04 & 2.004 & 3.889136e-04 & 1.994 & 2.061589e-05 & 1.989 & 7 \\
    & 1/80 & 5.180141e-05 & 2.004 & 9.781639e-05 & 1.991 & 5.172670e-06 & 1.995 & 8 \\
    \bottomrule
  \end{tabular*}
\end{table}

\begin{table}[h]
  \centering
  \caption{[Example 1] Convergence studies {on a configuration with hanging nodes $\Omega_4$} using proposed method}
  \label{ex1_tb3}
  \begin{tabular*}{.9\linewidth}{@{}lllllllll@{}}
    \toprule
    order & $h$ & $E^\mathbf{u}_{H^1}$ & rate & $E^p_{L^2}$ & rate & $E^{\psi}_{H^1}$ & rate & {itr} \\
    \midrule
    & 1/5  & 1.700467e-02 & --    & 3.324558e-02 & --    & 7.411002e-03 & --    & 7 \\
    & 1/10 & 6.418063e-03 & 1.406 & 8.366519e-03 & 1.991 & 3.720487e-03 & 0.994 & 7 \\
    {$k=1$} & 1/20 & 2.918559e-03 & 1.137 & 3.066089e-03 & 1.448 & 1.861992e-03 & 0.999 & 7 \\
    & 1/40 & 1.411469e-03 & 1.048 & 1.542844e-03 & 0.991 & 9.312391e-04 & 1.000 & 7 \\
    & 1/80 & 6.982803e-04 & 1.015 & 7.846744e-04 & 0.975 & 4.656609e-04 & 1.000 & 8 \\
    \midrule
    & 1/5  & 1.853717e-02 & --    & 3.539204e-02 & --    & 6.973699e-04 & --    & 7 \\
    & 1/10 & 4.086304e-03 & 2.182 & 7.855160e-03 & 2.172 & 1.755867e-04 & 1.990 & 7 \\
    {$k=2$} & 1/20 & 9.677780e-04 & 2.078 & 1.888270e-03 & 2.057 & 4.397411e-05 & 1.998 & 7 \\
    & 1/40 & 2.349881e-04 & 2.042 & 4.751719e-04 & 1.991 & 1.099864e-05 & 1.999 & 7 \\
    & 1/80 & 5.806576e-05 & 2.017 & 1.270720e-04 & 1.903 & 2.750029e-06 & 2.000 & 8 \\
    \bottomrule
  \end{tabular*}
\end{table}

\begin{table}[h]
  \centering
  \caption{[Example 2] Convergence studies on L-shaped domain with
    Voronoi mesh using proposed method}
  \label{ex2_tb1}
  \begin{tabular*}{.9\linewidth}{@{}lllllllll@{}}
    \toprule
    order & $h$ & $E^\mathbf{u}_{H^1}$ & rate & $E^p_{L^2}$ & rate & $E^{\psi}_{H^1}$ & rate & {itr} \\
    \midrule
    & 1/4  & 5.412799e-03 & --    & 1.761315e-02 & --    & 1.240271e-02 & --    & 5 \\
    & 1/8  & 2.089460e-03 & 1.438 & 4.781077e-03 & 1.970 & 6.441351e-03 & 0.990 & 6 \\
    {$k=1$} & 1/16 & 7.826530e-04 & 1.485 & 1.341547e-03 & 1.921 & 3.300835e-03 & 1.011 & 6 \\
    & 1/32 & 3.606618e-04 & 1.014 & 5.436234e-04 & 1.182 & 1.474711e-03 & 1.055 & 6 \\
    & 1/64 & 1.644670e-04 & 1.123 & 2.656548e-04 & 1.030 & 7.382761e-04 & 0.995 & 6 \\
    \midrule
    & 1/4  & 9.673890e-03 & --    & 2.603108e-02 & --    & 1.745202e-03 & --    & 5 \\
    & 1/8  & 1.324590e-03 & 3.005 & 6.156831e-03 & 2.179 & 4.313425e-04 & 2.112 & 6 \\
    {$k=2$} & 1/16 & 2.939265e-04 & 2.276 & 1.583476e-03 & 2.053 & 1.094255e-04 & 2.074 & 6 \\
    & 1/32 & 5.019257e-05 & 2.314 & 3.452908e-04 & 1.994 & 2.263689e-05 & 2.062 & 6 \\
    & 1/64 & 1.248291e-05 & 2.002 & 9.124393e-05 & 1.914 & 5.360583e-06 & 2.072 & 7 \\
    \bottomrule
  \end{tabular*}
\end{table}

\begin{figure}[h]
  \centering
  \subfloat[]{\includegraphics[height=5.33cm, width=4cm]{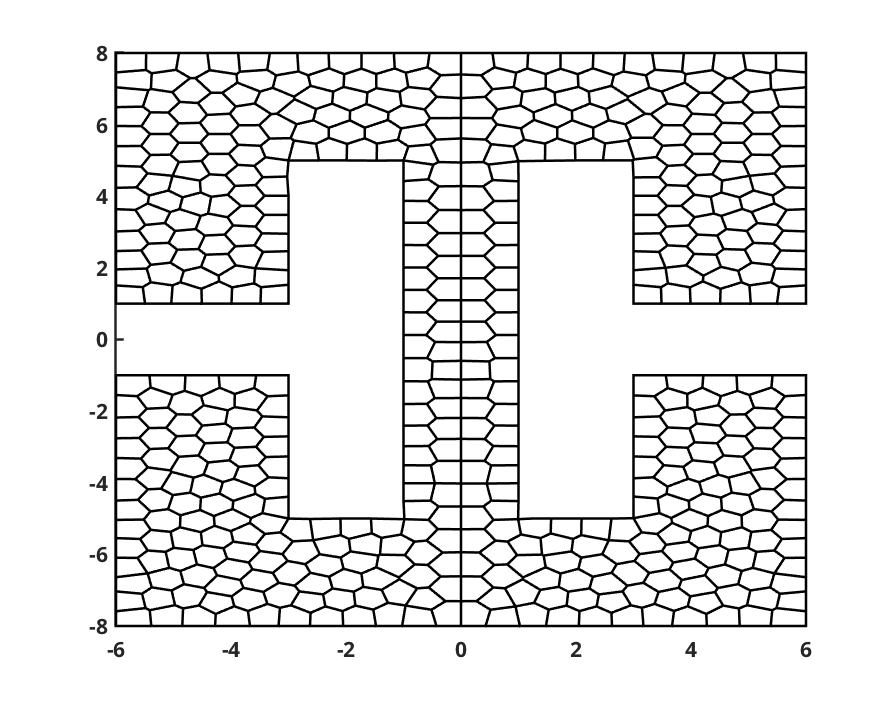}}
  \subfloat[]{\includegraphics[height=4.66cm, width=4cm]{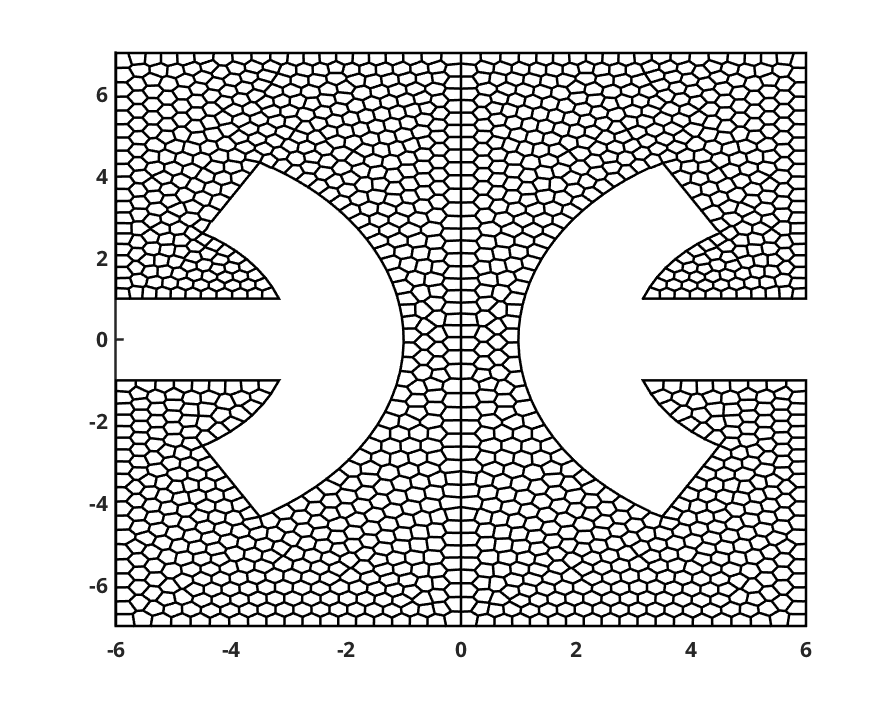}}
  \caption{[Example 3] Computational domains with their geometry.}
  \label{samp2} 
\end{figure}

%% SECTION 5.2
\subsection{Example 2: Convergence studies on non-convex domain}
\label{case2}
In this example, we investigate the convergence behavior of the
proposed method on the non-convex domain with re-entrant {corner} for a
problem with a known solution.
We introduce an L-shaped non-convex domain given by
$(0,1)^2\backslash(0,0.5)^2$, discretized with regular Voronoi mesh,
as shown in Figure \ref{samp}(c).
We consider the same manufactured solution of Example~1.
The problem parameters are the same, with the exception of the
pressure $p_0$, which must be recalculated for the L-shaped domain by
using \eqref{heat-2}.

Table \ref{ex2_tb1} confirms optimal convergence rates
$\mathcal{O}(h)$ for $k=1$ and $\mathcal{O}(h^2)$ for $k=2$ for the
approximation of all fields, velocity, pressure and potential, showing
that the geometric non-convexity does not degrade the theoretical
convergence behavior.
In fact, the Voronoi tessellation maintains the mesh regularity
assumption \textbf{(A1)} that we introduced at the beginning of
Section~\ref{sec-3}.
%% The numerical error and convergence characteristic of the proposed
%% method on the L-shaped domain is demonstrated in Table \ref{ex2_tb1}
%% for VEM orders $k=1$ and $k=2$.
%% %%
%% Once again, we observe that the proposed method exhibits the expected
%% convergence rates as derived theoretically.

%%%%%%%%%%%%%%%%%%%%%%%%%%%%%%%%%%%%%%%%%%%%%%%%%%%%%%%%%%%%%%%%%%%%%%%%%%%%%%%%%%%%%%
%%%%%%%%%%%%%%%%%%%%%%%%%%%%%%%%%%%%%%%%%%%%%%%%%%%%%%%%%%%%%%%%%%%%%%%%%%%%%%%%%%%%%%

%% SECTION 5.3
\subsection{Example 3: Flows in nanopore sensor with obstacles}
\label{case3}

The final example demonstrates the VEM formulation's capability for
simulating realistic electrically charged flows in a nanopore sensor
with T-shaped and curved-shaped obstacles, which is the same
configuration considered
in~\cite{mitscha2017adaptive,alsohaim2025analysis}.
%%
%% Finally, we end our numerical simulation by investigating the
%% electrically charged flows in a nanopore sensor with T-shaped and
%% curved-shaped obstacles following
%% \cite{mitscha2017adaptive,alsohaim2025analysis}.
%%
The geometry and computational meshes are shown in Figure \ref{samp2}
with the sizes 12 nm $\times$ 16 nm (for T-shaped obstacles: (a)) and
12 nm $\times$ 14 nm (for curved shaped obstacles: (b)).
%%
%% We use the same flow configurations, which have been discussed in
%% \cite{alsohaim2025analysis, mitscha2017adaptive}.
%%
An ionic current is driven by an imposed electric potential difference
across the nanopore.
The electrostatic potential \(\psi\) is prescribed on the inflow
(bottom) and outflow (top) boundaries as $\psi_{in}=0$ and
$\psi_{out}=2$, and zero-flux conditions on all other boundaries.
%%
%% On the remaining boundary, we impose a zero-flux condition for the
%% potential.

For the velocity field, on the top boundary we impose the parabolic
 {profile, given by}
\begin{equation*}\label{eq:uin}
  \mathbf{u}_{{in}}(x) \;=\; \Big[ 0, -0.1\,\tanh\!\bigl(40(6-x)^2\bigr) \Big]^T,
  %	\quad\text{on the top boundary},
\end{equation*}
and $\mathbf{u}\cdot \mathbf{n} =0$ on the lateral (right and left)
boundaries, zero-traction conditions at {the bottom}, and no-slip
conditions on the obstacle surfaces and channel walls.
%%
%% On the outer right and left boundaries, we impose $\mathbf{u}\cdot
%% \mathbf{n} =0$. We apply zero-traction boundary condition $(\mu
%% \nabla \mathbf{u} + p \mathbb I) \mathbf{n}=\mathbf{0}$ at the
%% outflow, and on the remaining boundaries, we impose the no-slip
%% condition for the velocity vector field.
%%
In this example, we incorporate the nonlinear convective term in the
momentum equation.
We consider the following parameters
\begin{align*}
  \mu = 0.1\,\, \text{Pa $\cdot$ s}\,\,\qquad
  \epsilon=0.075,\qquad
  \alpha_0=0.001,\qquad
  \alpha_1=1,\qquad
  \mathbf{f}=0, \qquad g=0.
\end{align*}
We investigate the practical behavior of the VEM in the two cases with
$\mathbf{E}=[0.1, -0.1]^T$ and $\mathbf{E}=[1, -1]{^T}.$
In both cases, we employ a Voronoi mesh with 18,000 elements, and VEM
orders $k=1$ and $k=2$, and the externally applied electric field
$\mathbf{E}$ does not act straight down, but with a slight angle to
break symmetry with low and moderate magnitudes.
%% We use 18,000 regular Voronoi elements to conduct this experiment
%% for VEM order $k=1$ and $k=2$.  In both cases, the externally
%% applied electric field $\mathbf{E}$ does not act straight down, but
%% with a slight angle to break symmetry with low and moderate
%% magnitudes.
For $\mathbf{E}=[0.1, -0.1]{^T}$, the stabilized velocity norm, the
stabilized pressure, and the potential distribution in the nanopore
sensor with T-shaped obstacles are shown in Figures \ref{ex3fig1} and
\ref{ex3fig2} for VEM order $k=1$ and $k=2$.
The recirculation region appears near the obstacles due to the
combined effects of viscous drag and electrostatic forces.
The pressure drop is noticed from the inlet to the outlet.
The potential field exhibits smooth gradients from the inflow to
outflow boundary, slightly distorted near the obstacles due to the
applied electric field.
For $k=1$, we observe a singularity for the velocity field near
$(-6,1)$ and $(6,1)$, which is effectively suppressed by employing the
higher order VEM $k=2$ as shown in {Figures~\ref{ex3fig2} and~\ref{ex3fig1r}}.
Additionally, the electro-hydrodynamics flow pattern obtained using
the proposed method is almost identical to that of
\cite{alsohaim2025analysis}.
    
For $\mathbf{E}=[1,-1]{^T}$, the flow pattern in the nanopore sensor
with T-shaped obstacles is shown in Figure \ref{ex3fig3} for VEM order
$k=2$.
We notice that the flow pattern becomes slightly complex as we
increase the magnitude of the applied electric field, and the proposed
method is efficient in capturing the flow dynamics without exhibiting
oscillation or singularities near the corner of the T-shaped
obstacles.

The flow dynamics of the electrically charged fluid in the nanopore
sensor with curved obstacles are shown in Figure \ref{ex3fig4} for
$E=[0.1,-0.1]{^T}$ and $\mathbf{E}=[1,-1]{^T}$ with higher order VEM
$k=2$.
We notice that the flow characteristics are similar to the previous
case, but becomes slightly complex as we increase the magnitude of the
electric field. {Additionally, for a better visualization of the flow, the contour plot of the velocity magnitude and velocity vector field is also shown in Figure~\ref{ex3fig1s}.}
   
\begin{figure}[h]
  \centering
      {\includegraphics[height=6.66cm, width=5cm]{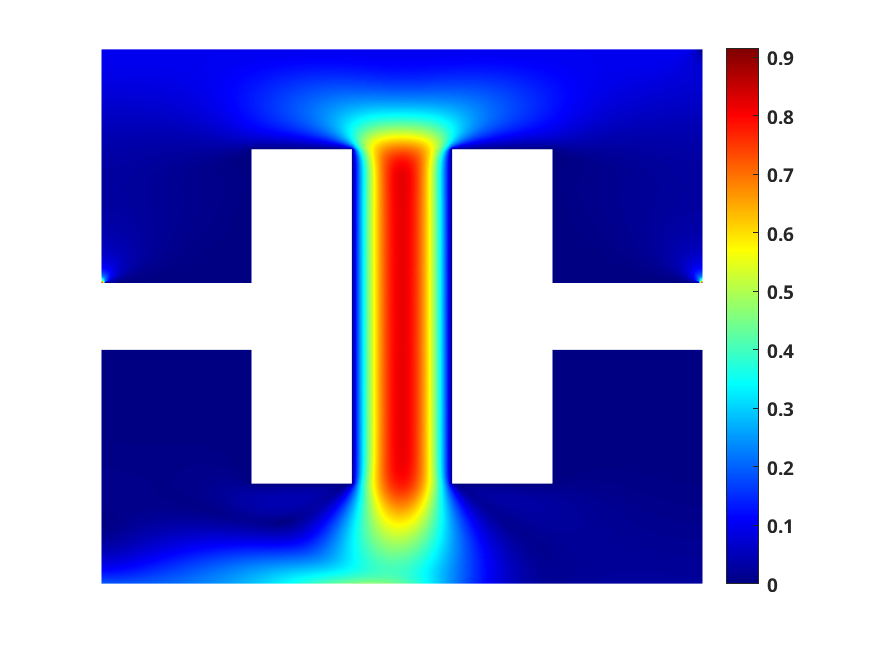}}~~~
      {\includegraphics[height=6.66cm, width=5cm]{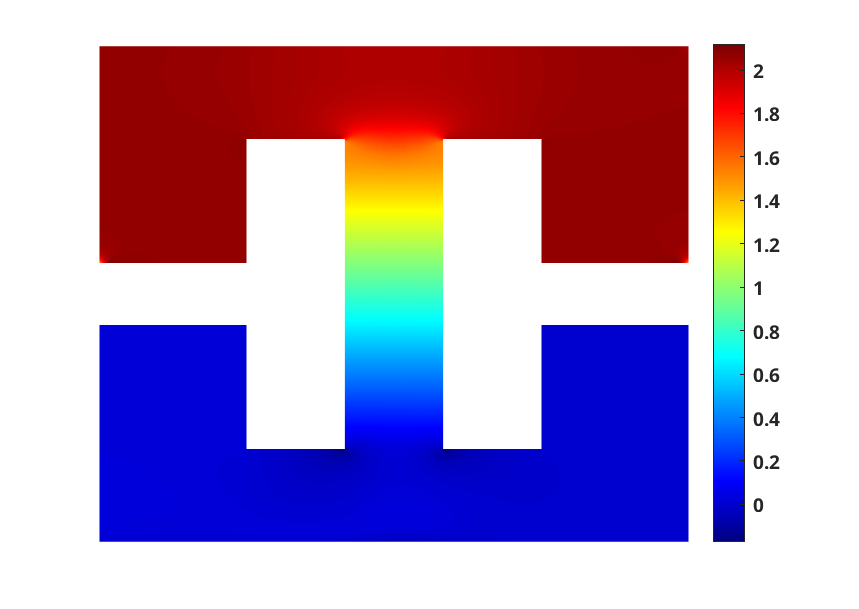}} 
      {\includegraphics[height=6.66cm, width=5cm]{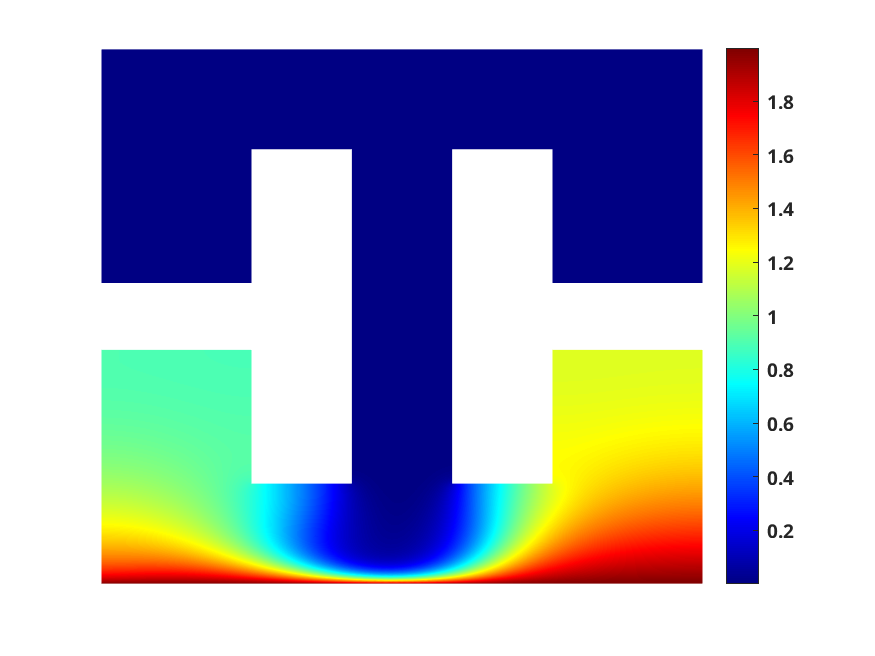}}
      \caption{[Example 3] Snapshots of flow dynamics of the
        electrically charged fluid in a nanosensor: discrete velocity
        magnitude, pressure, and potential for VEM order $k=1$ with
        $\mathbf{E}=[0.1, -0.1]{^T}$}
      \label{ex3fig1} 
\end{figure}

\begin{figure}[h]
  \centering
      {\includegraphics[height=6.66cm, width=5cm]{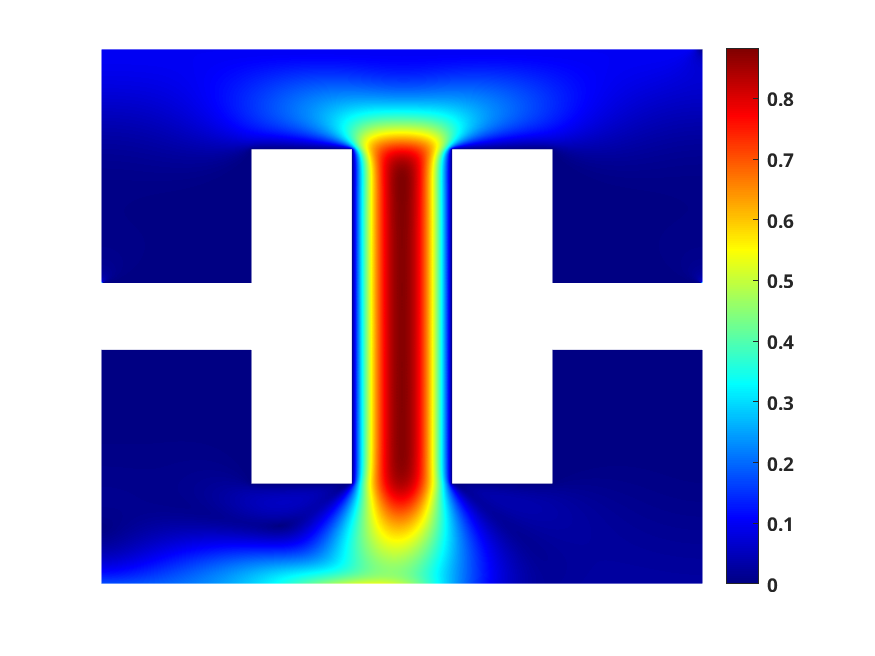}}~~~
      {\includegraphics[height=6.66cm, width=5cm]{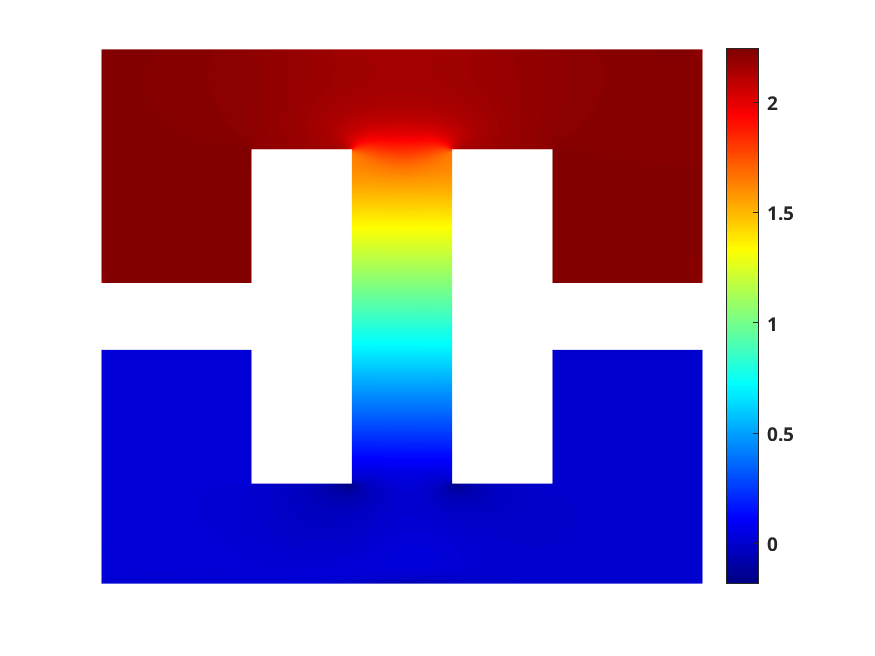}} 
      {\includegraphics[height=6.66cm, width=5cm]{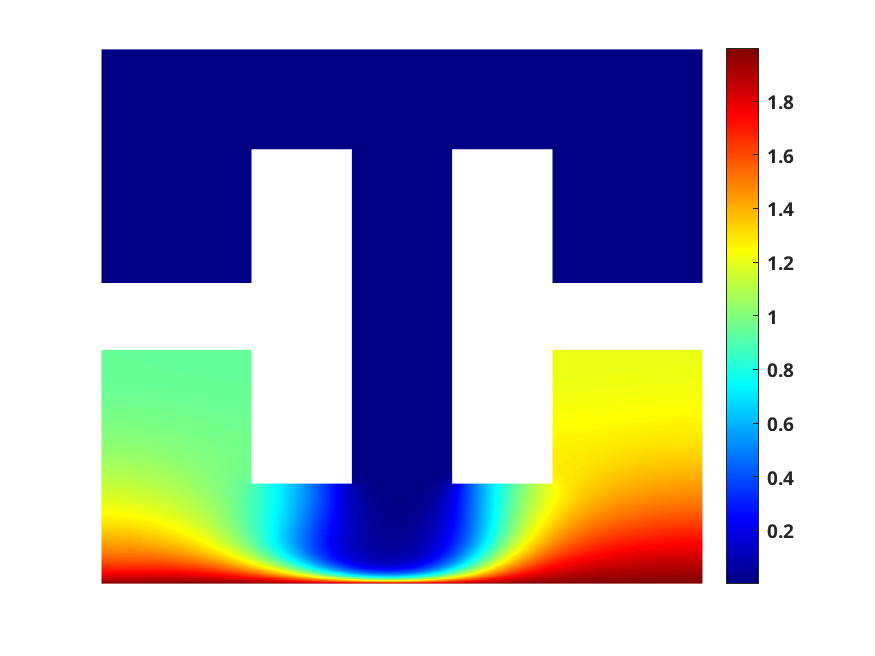}}
      \caption{[Example 3] Snapshots of the flow dynamics of the
        electrically charged fluid in a nanosensor: discrete velocity
        magnitude, pressure, and potential for VEM order $k=2$ with
        $\mathbf{E}=[0.1, -0.1]{^T}$}
\label{ex3fig2} 
\end{figure}

\begin{figure}[h]
	\centering
%	\fcolorbox{black}{blue!50}{%
	{\includegraphics[height=6.66cm, width=7cm]{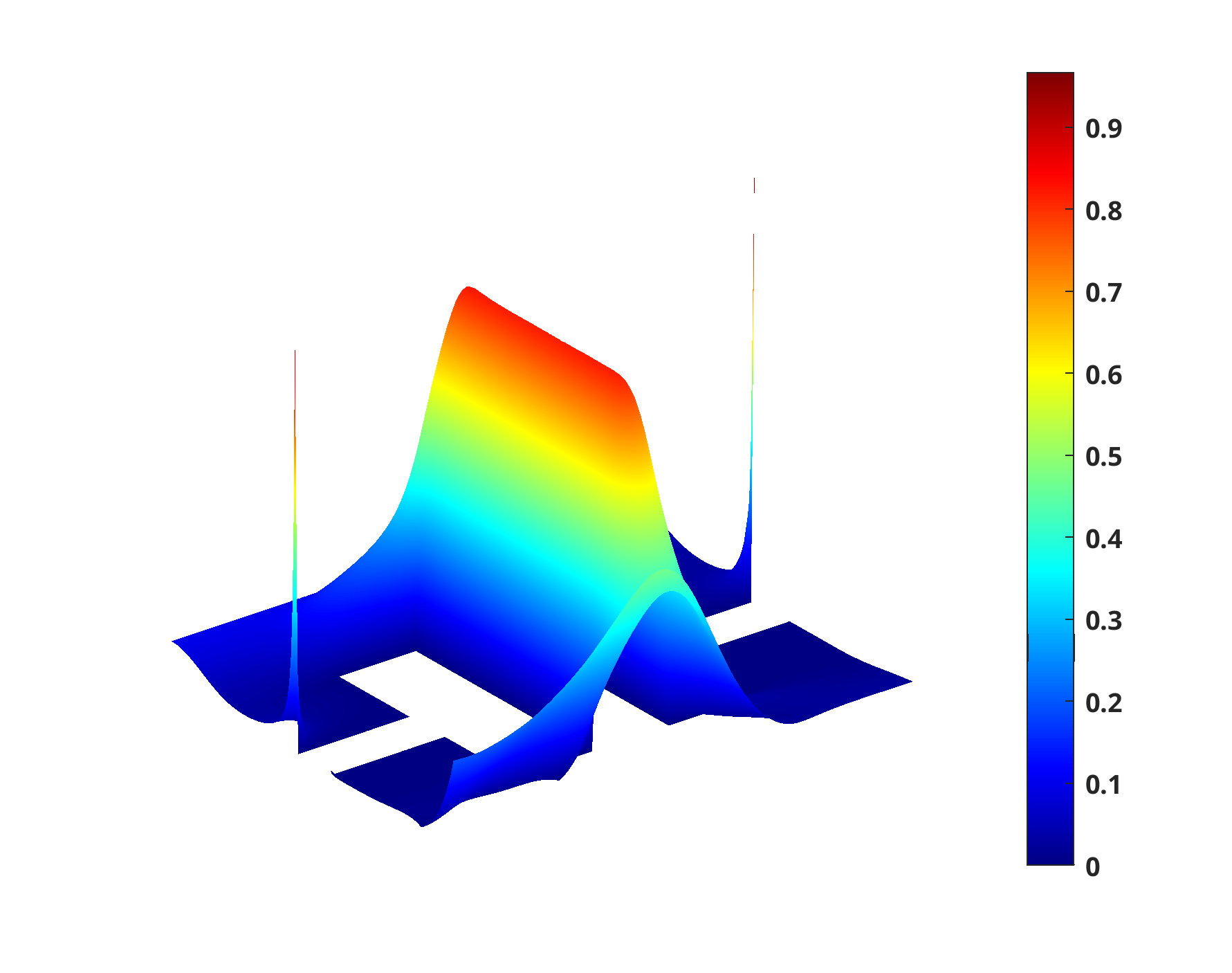}}~~~
	{\includegraphics[height=6.66cm, width=7cm]{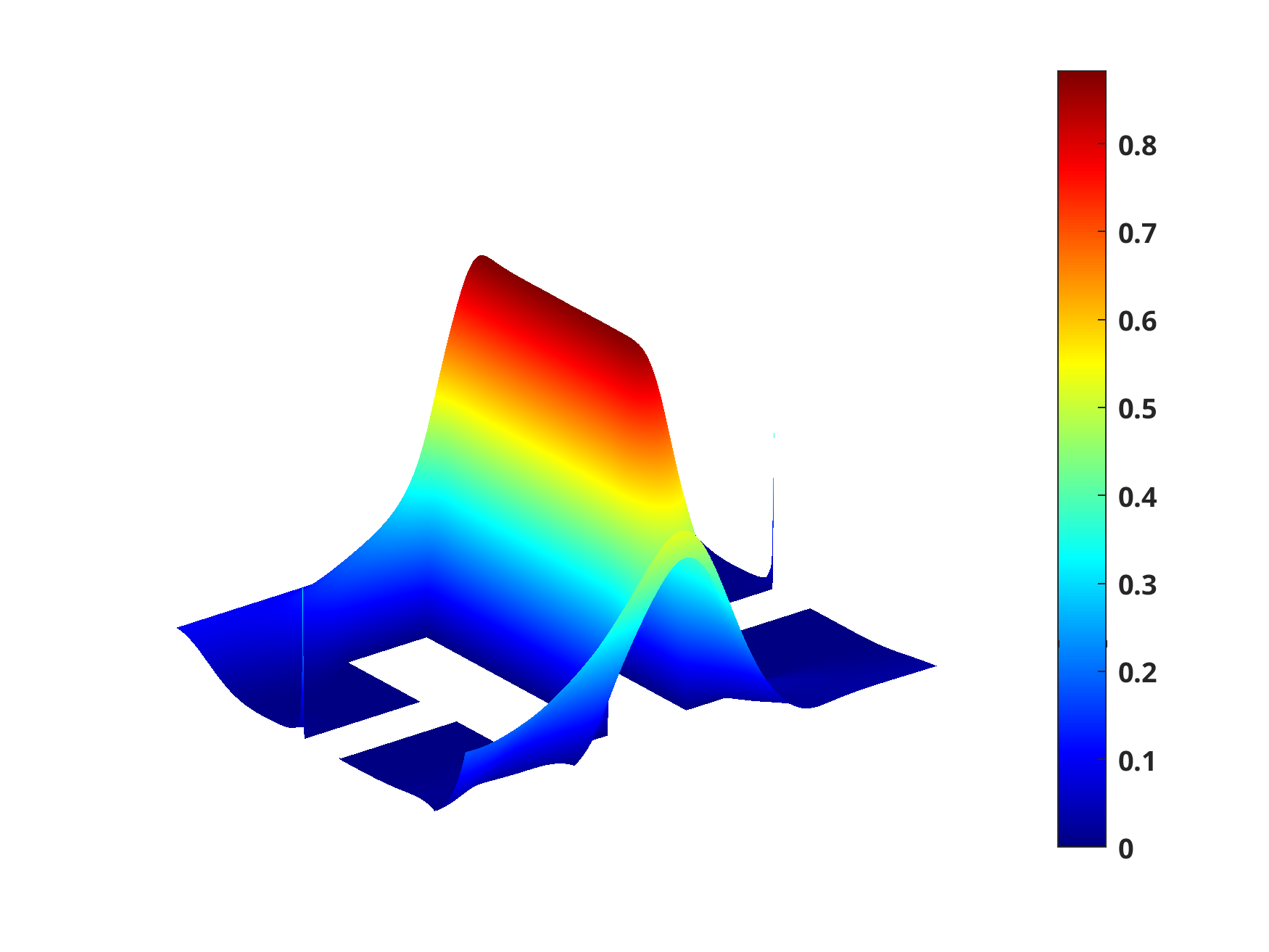}} 
	\caption{{[Example 3] Three-dimensional visualizations of the velocity magnitude in the nanosensor for the virtual element method with polynomial orders $k = 1$ (left) and $k = 2$ (right), corresponding to $\mathbf{E}=[0.1,-0.1]^T$}}
	\label{ex3fig1r} 
\end{figure}

\begin{figure}[h]
  \centering
      {\includegraphics[height=6.33cm, width=5cm]{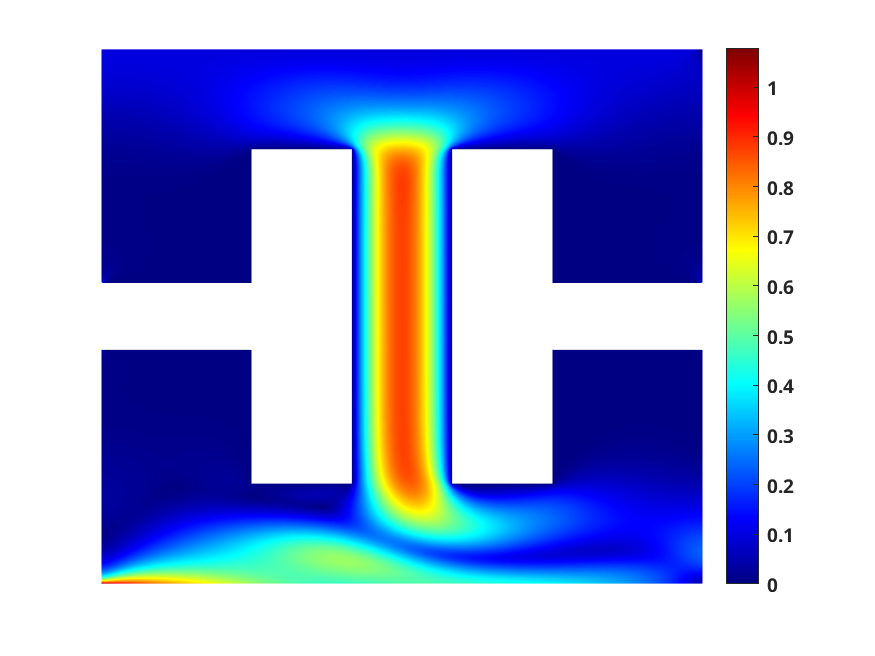}}~~~
      {\includegraphics[height=6.33cm, width=5cm]{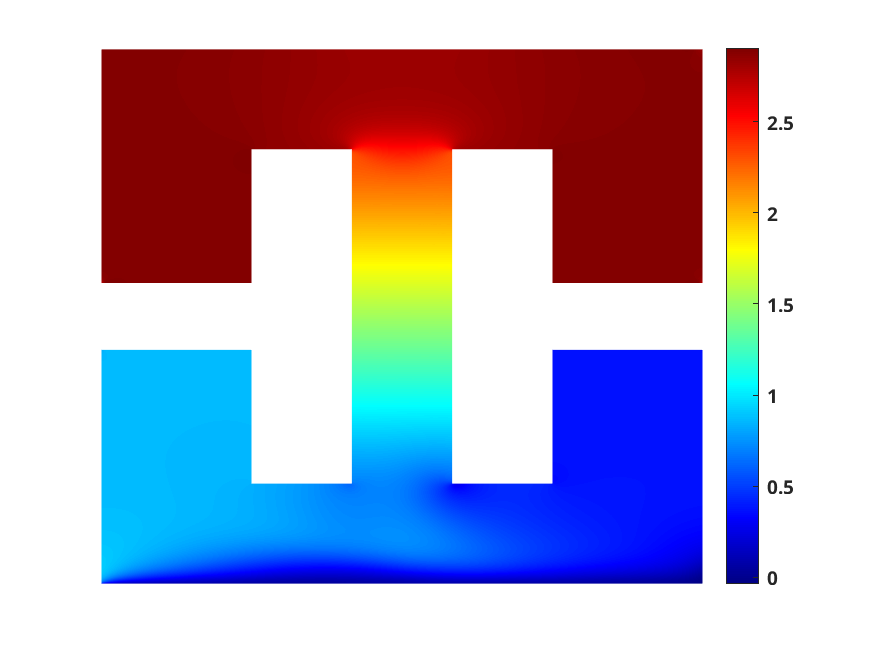}} 
      {\includegraphics[height=6.33cm, width=5cm]{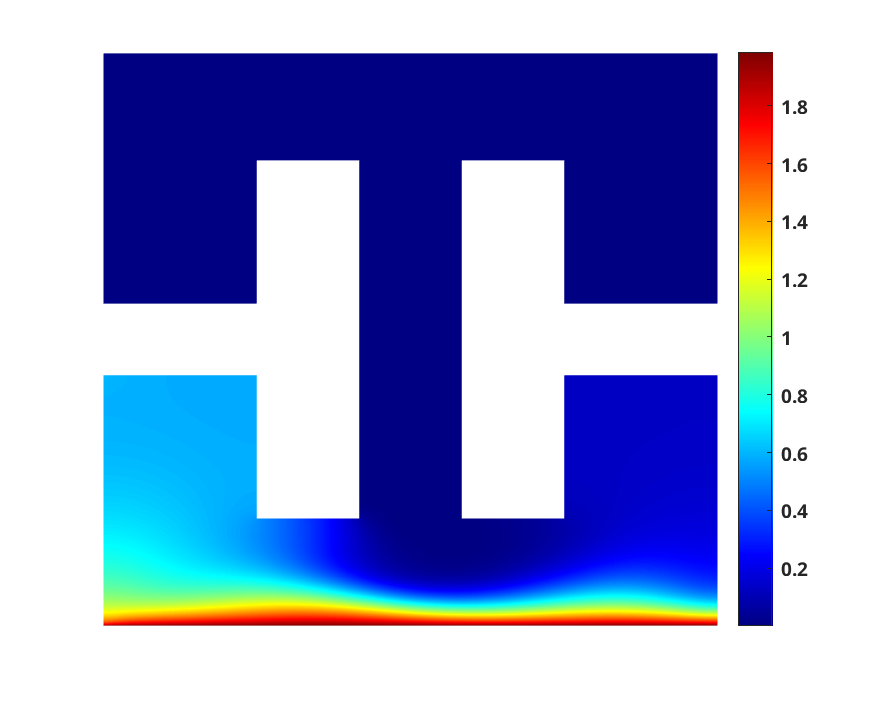}}
      \caption{[Example 3] Snapshots of the flow dynamics of the
        electrically charged fluid in a nanosensor: discrete velocity
        magnitude, pressure, and potential for VEM order $k=2$ with
        $\mathbf{E}=[1, -1]{^T}$.}
      \label{ex3fig3} 
\end{figure}

\begin{figure}[h]
	\centering
	%\fcolorbox{black}{blue!50}{%
	{\includegraphics[height=8.2cm, width=6.5cm]{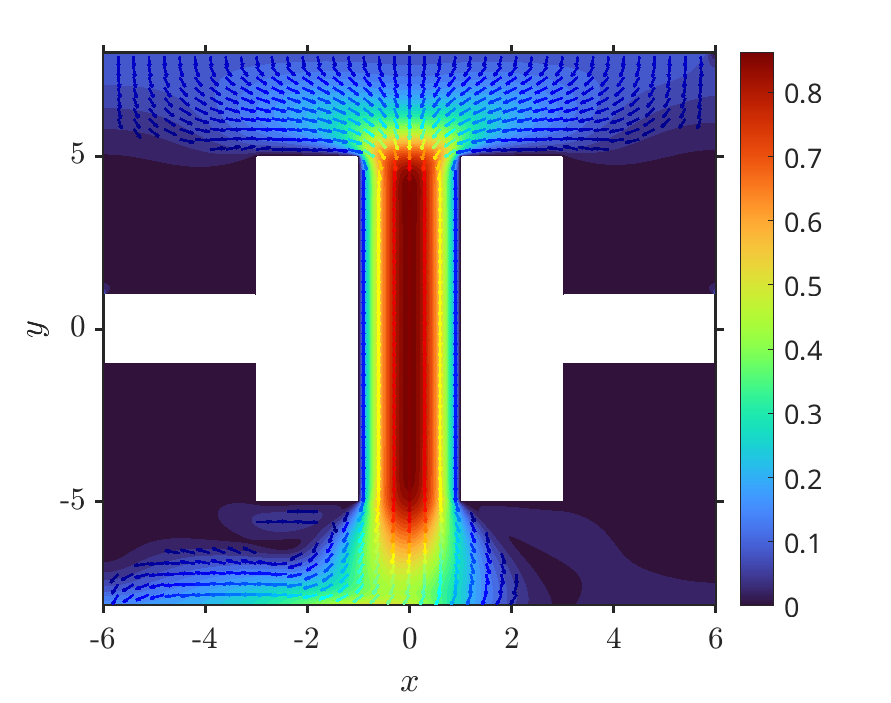}}~~~
	{\includegraphics[height=8.2cm, width=6.5cm]{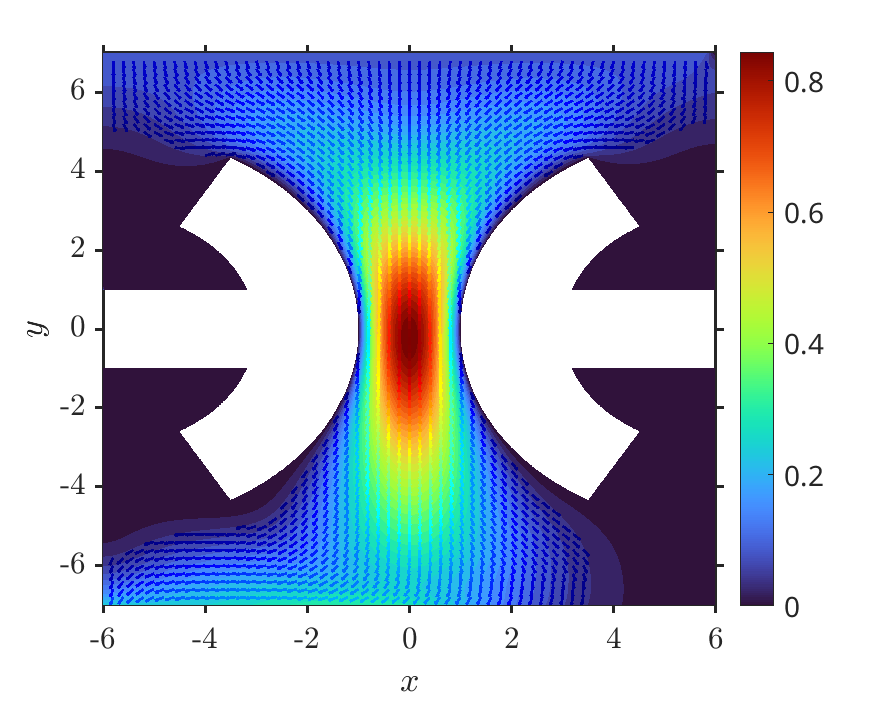}} 
	\caption{ {[Example 3] Contour plot of the speed and velocity vector field of the
		electrically charged fluid in a nanosensor for VEM order $k=2$ with	$\mathbf{E}=[0.1, -0.1]{^T}$}}
	\label{ex3fig1s} 
\end{figure}

\begin{figure}[h]
  \centering
      {\includegraphics[height=5.83cm, width=5cm]{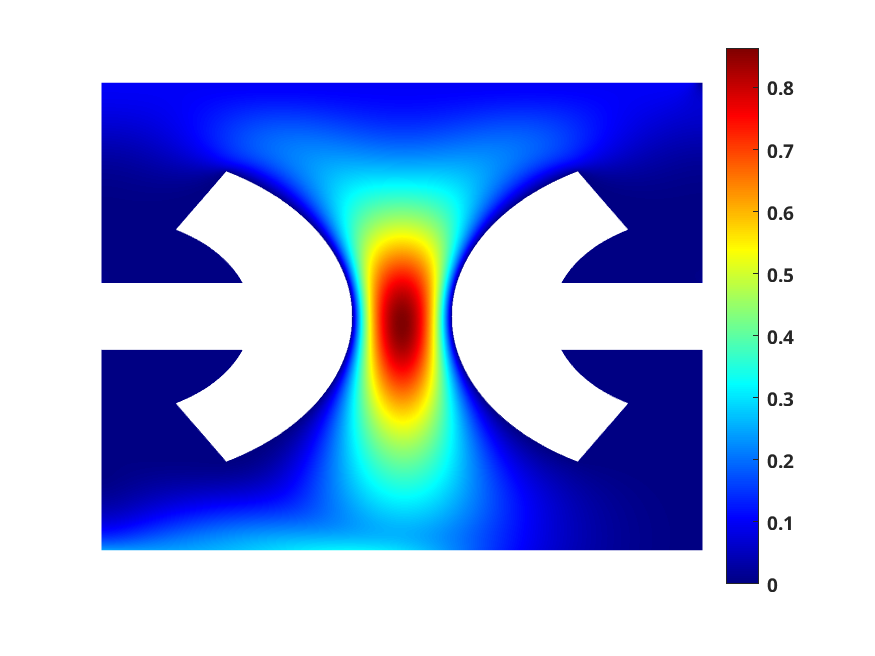}}~~~
      {\includegraphics[height=5.83cm, width=5cm]{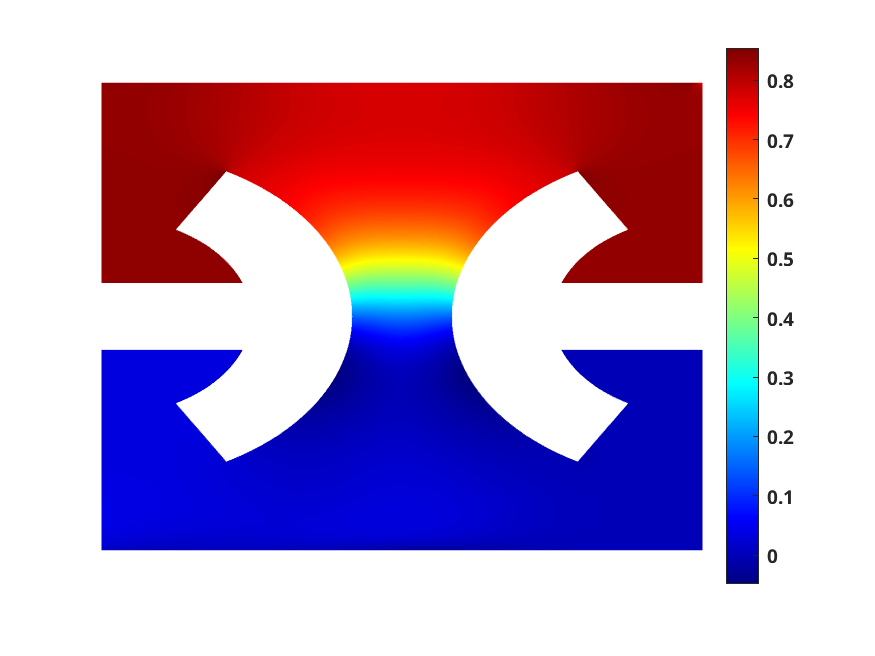}} 
      {\includegraphics[height=5.83cm, width=5cm]{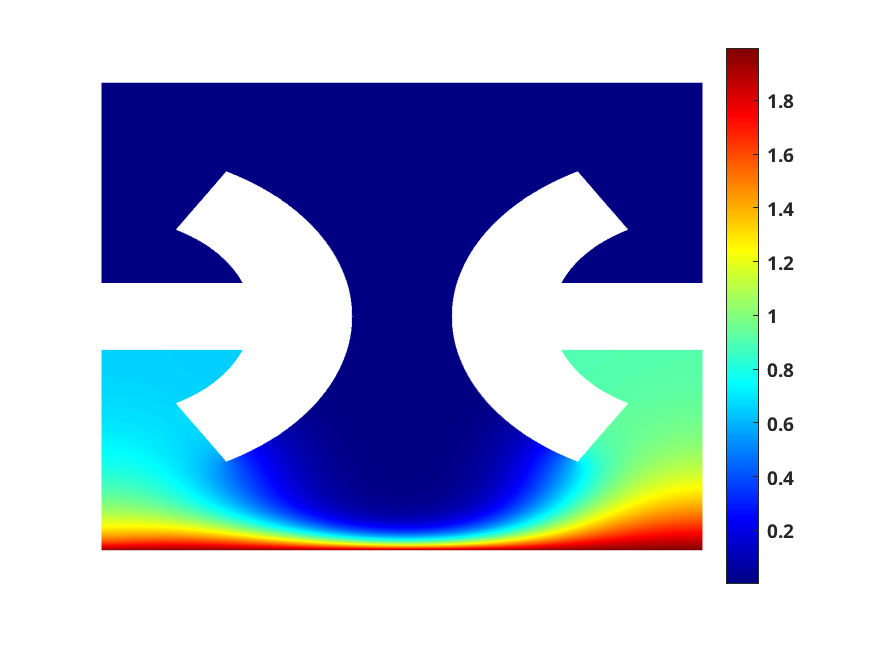}} \\
      {\includegraphics[height=5.83cm, width=5cm]{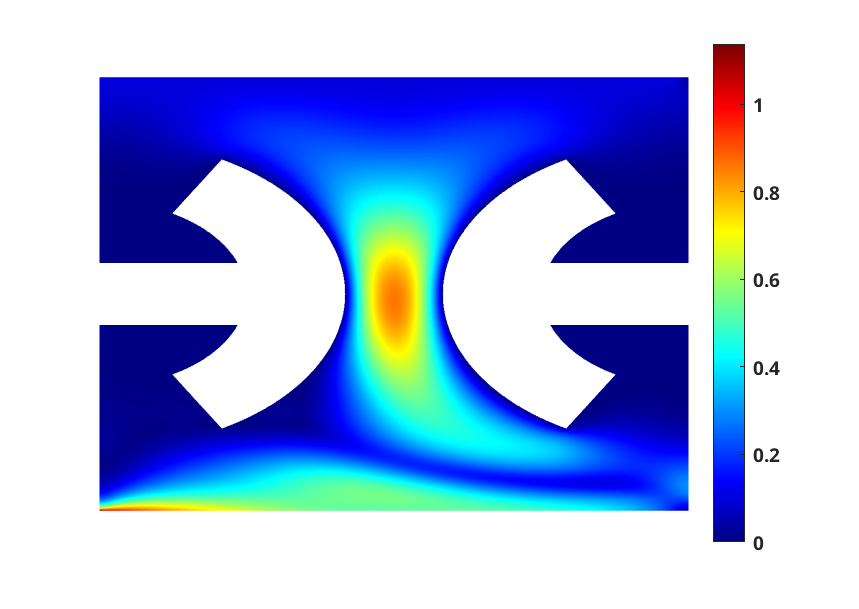}}~~~
      {\includegraphics[height=5.83cm, width=5cm]{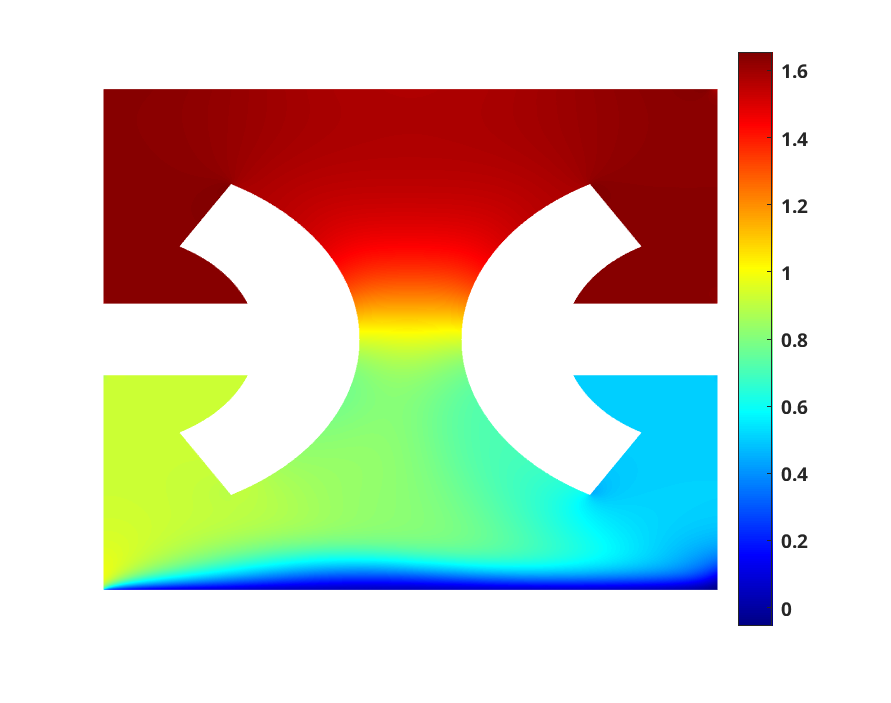}} 
      {\includegraphics[height=5.83cm, width=5cm]{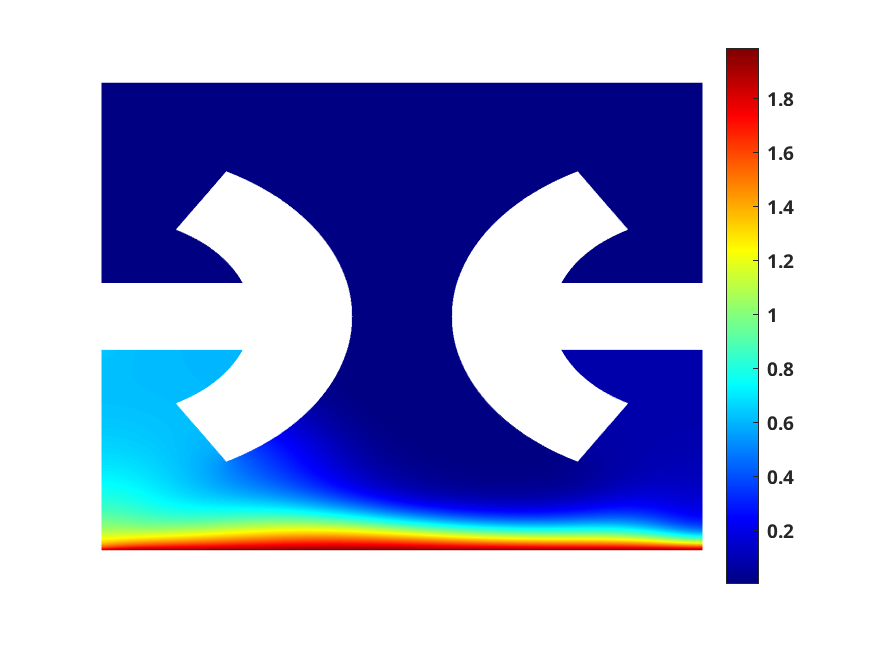}}
      \caption{[Example 3] Snapshots of the flow dynamics of the
        electrically charged fluid in a nanosensor with curved shaped
        obstacles: discrete velocity magnitude, pressure, and
        potential for VEM order $k=2$ with $\mathbf{E}=[0.1, -0.1]{^T}$
        (top) and $\mathbf{E}=[1, -1]{^T}$(bottom).}
      \label{ex3fig4} 
\end{figure}

\noindent
%% \MGT{
\textbf{Physical interpretation.}
The flow patterns observed in Figures~\ref{ex3fig1}--\ref{ex3fig4}
are consistent with well-established electrokinetic transport
phenomena in confined geometries.
The recirculation zones near the T-shaped obstacles arise from the
competition between electro-osmotic slip velocity at charged
boundaries and adverse pressure gradients induced by geometric
constriction.
These recirculation regions are relevant for nanopore sensing
applications, as they create zones of extended residence time that
can enhance biomolecule capture and detection sensitivity.
The slight asymmetry in all fields, caused by the non-axial
electric field orientation, influences charged particle
trajectories---a mechanism exploitable for biomolecule trajectory
control.
The VEM formulation successfully captures these characteristic
electrokinetic phenomena, e.g., electro-osmotic plug flow in the
channel core, boundary-layer-driven recirculation near obstacles, and
smooth potential gradients governed by the Poisson--Boltzmann
coupling, without requiring specialized mesh treatment for the complex
obstacle geometries.

When the field magnitude increases from $\mathbf{E}=[0.1,-0.1]{^T}$ to
$\mathbf{E}=[1,-1]^t$, recirculation intensifies and velocity
magnitudes increase proportionally, consistent with the linear
electro-osmotic velocity--field relationship in the small
Debye--H\"uckel parameter regime.
The curved obstacle configuration produces smoother flow transitions,
smaller recirculation zones, and reduced pressure losses compared to
sharp T-shaped corners---a geometric effect relevant for applications
such as DNA sequencing where consistent translocation velocity is
desirable.
These results confirm that the proposed VEM formulation captures the
essential coupled fluid-electrostatic interactions in realistic
nanopore geometries, with the polygonal mesh framework providing
natural geometric conformity to the obstacle boundaries without
special mesh treatment.
%%}

%% ============================================================
%% Condensed comparison subsection
%% Replaces Section 4 (Computational and geometric advantages)
%% ============================================================

%% \subsection{\MGT{Comparison with finite element methods}}
\subsection{Comparison with finite element methods}
\label{sec-comparison}
\label{sec-4}

\begin{table}[t]
    \centering
  \caption{Comparison of mesh flexibility and geometric
    capabilities}
  \label{tab:mesh-comparison}
  \begin{tabular}{lll}
    \toprule
    \textbf{Feature}
    & \textbf{Taylor--Hood FEM}
    & \textbf{Equal-Order VEM} \\
    \midrule
    Element types
    & Triangles only & Arbitrary polygons \\
    Non-convex elements
    & Not supported & Fully supported \\
    Hanging nodes
    & Constraint equations & Natural treatment \\
    Mesh distortion tolerance
    & Moderate (quality-dependent) & High (projection-based) \\
    Adaptive refinement
    & Transition elements needed & Straightforward \\
    Irregular boundaries
    & Many small elements & Fewer conforming polygons \\
    \bottomrule
  \end{tabular}
\end{table}

The numerical experiments presented above demonstrate optimal
convergence across diverse mesh types and realistic nanopore
geometries.
We now compare the proposed VEM formulation with the Taylor--Hood
finite element approach of AlSohaim et
al.~\cite{alsohaim2025analysis}, which employs
$\PS{k+1}/\PS{k}$ elements on triangular meshes, to
highlight the computational and geometric distinctions between
the two frameworks.

%% Degrees of freedom
For a polygonal mesh $\Omega_h$ with $N_v$ vertices, $N_e$ edges,
and $N_E$ elements, the total degrees of freedom for the VEM
discretization of order $k$ (accounting for velocity, pressure, and
potential, each using the space ${V^h_k(E)}$ from
Section~\ref{subsec-33:discrete:forms}) are:
\begin{align}
  \text{DOF}_{\text{VEM}}
  = 4N_v + 4(k-1)N_e + 4N_E \dim(\PS{k-2}),
  \label{eq:dof-vem}
\end{align}
where $\dim(\PS{k-2}) = (k-1)k/2$ for $k \geq 2$ and equals
zero for $k=1$.
For Taylor--Hood finite elements on a comparable triangular mesh,
the higher-order velocity space $\PS{k+1}$ modifies the count
to:
\begin{align}
  \text{DOF}_{\text{FEM}} =
  4N_v^{\text{FEM}} +
  (4k-2)N_e^{\text{FEM}} +
  N_E^{\text{FEM}}\big(2\dim(\PS{k-2}) +
  2\dim(\mathbb{P}_{k-3})\big).
  \label{eq:dof-fem}
\end{align}
For $k=1$, VEM requires $4N_v$ DOFs while FEM requires
$4N_v + 2N_e$; with typical triangular mesh topology ratios
$N_e \approx 1.5\,N_v$, this yields a reduction of approximately
$43\%$ in total degrees of freedom.
For $k=2$, VEM requires $(4N_v + 4N_e + 4N_E)$ DOFs while FEM
requires $(4N_v + 6N_e + 2N_E)$ DOFs; under the additional
assumption $N_E \approx 0.5\,N_v$, the reduction is approximately
$15\%$.
We note that these ratios depend on mesh topology and element
counts, and the reduction is more pronounced at lower polynomial
orders where the difference in velocity approximation order has a
larger relative impact.

% Implementation and stabilization structure
The equal-order formulation provides two key structural
simplifications relative to mixed finite element methods.
First, all three fields, e.g., velocity, pressure, and potential, are
approximated by the same local virtual element space ${V^h_k(E)}$ with
identical degrees of freedom structure.
The projection operators $\Pi^{\nabla,E}_k$ and $\Pi^{0,E}_k$
defined in \eqref{H1proj}--\eqref{l2proj} are computed once per
element and reused across all field components, eliminating the
need for separate shape function implementations, quadrature rules,
and local-to-global mappings for different approximation orders.
Second, the reformulation of the drag force term
$-\epsilon \Delta \psi \mathbf{E}$ using the transport potential
equation eliminates second-order derivatives from the momentum
residual.
Only first-order derivatives $\nabla \psi_h$ appear in the
stabilization terms $\mathcal{L}_{1,h}$, $\mathcal{L}_{2,h}$, and
$\mathcal{L}_{3,h}$, which are well-defined for functions in the
$H^1$-conforming space $V_h$ and involve only polynomial projected
quantities computable from the degrees of freedom.

%% Geometric flexibility
Table~\ref{tab:mesh-comparison} summarizes the geometric
capabilities distinguishing the two frameworks.
The VEM framework requires only the mild mesh regularity assumption
\textbf{(A1)} from Section~\ref{sec-3}: star-shapedness with
respect to a ball of radius $\geq \delta_0 h_E$ and edge length
$\geq \delta_0 h_E$.
Subject to these constraints, elements may be convex or non-convex,
have any number of edges, and accommodate hanging nodes without
constraint equations.
The convergence studies presented in Examples~\ref{case1}-\ref{case2}
confirm that optimal $\mathcal{O}(h^k)$ rates are achieved across all
these mesh types.

\section{Conclusions}
\label{sec-6}

We have presented an equal-order virtual element method for the
coupled Stokes--Poisson--Boltzmann equations on general polygonal
meshes.
The formulation employs a
$\PS{k}/\PS{k}/\PS{k}$ approximation for velocity,
pressure, and potential, stabilized through a residual-based scheme
that exploits the structure of the drag force coupling to eliminate
second-order derivative terms from the discrete problem.
Well-posedness of the continuous and discrete problems was
established through the Banach and Brouwer fixed-point theorems,
respectively, under explicit smallness conditions on the data, and
optimal a priori error estimates of order $\mathcal{O}(h^k)$ were
derived in the energy norm.
Numerical experiments confirmed these theoretical predictions across
diverse mesh types---including non-convex polygons, Voronoi
tessellations, distorted elements, and meshes with hanging
nodes---and demonstrated the method's applicability to
electro-osmotic flows in nanopore sensors with T-shaped and curved
obstacle geometries.

Several directions for future research emerge from this work.
The formulation can be extended to the coupled
Poisson--Nernst--Planck system to model multiple ionic species and
to time-dependent Stokes--Poisson--Boltzmann equations for dynamic
electrokinetic phenomena.
Development of pressure-robust error estimates, where velocity
errors are independent of pressure regularity, represents a natural
next step; preliminary investigations using $H(\mathrm{div})$-conforming
velocity reconstructions show promise for extending the
pressure-robust framework to the VEM setting, and this work is
currently in progress.
Extension to polyhedral meshes in three dimensions and incorporation
of convective terms through SUPG stabilization for higher Reynolds
number regimes would further broaden applicability to realistic
engineering configurations.

% ----------------------------
% Acknowledgments
% ----------------------------

\section*{Acknowledgments}
G. Manzini is affiliated to GNCS-INdAM, the Italian Gruppo Nazionale
Calcolo Scientifico - Istituto Nazionale di Alta Matematica.
AI tools may have served as writing and editing assistants; however, the authors 
carefully reviewed the manuscript and take full responsibility for the content of this work.

%\clearpage 
\subsection*{Data Availability}
No data was used for the research described in the article.

\section*{Declarations}

\subsection*{Conflict of interest}
The authors declare no conflict of interest during the current study.
G. Manzini is also affiliated with Los Alamos National Laboratory.
However, this work was carried out entirely outside the scope of any
activities, duties, or appointments at Los Alamos National Laboratory
or the U.S. Department of Energy.
Los Alamos National Laboratory and the U.S. Department of Energy had
no role in the conception, execution, funding, or supervision of this
work, and assert no rights, title, or interest, including intellectual
property rights, in the results presented herein.
The views and conclusions expressed are solely those of the authors
and do not represent the views of Los Alamos National Laboratory, the
U.S. Department of Energy, or their employees or contractors.

\section*{Funding}
No funding was received by the Authors to prepare this manuscript.  

\section*{Authors' contribution}
All authors equally contributed to this work.

%Loading bibliography style file
%% \bibliographystyle{unsrt}

%% \bibliography{refs}

\begin{thebibliography}{10}

\bibitem{Dekker2007}
C.~Dekker.
\newblock Solid-state nanopores.
\newblock {\em Nature Nanotechnology}, 2:209--215, 2007.

\bibitem{alsohaim2025analysis}
A.~AlSohaim, R.~Ruiz-Baier, and S.~Villa-Fuentes.
\newblock Analysis of a finite element method for the
  {S}tokes--{P}oisson--{B}oltzmann equations.
\newblock {\em {The Proceedings of ANZIAM}}, {66: C61--C78, 2024}.







\bibitem{Probstein1994}
R.~F. Probstein.
\newblock {\em Physicochemical Hydrodynamics: An Introduction}.
\newblock Wiley-Interscience, New York, 2nd edition, 1994.

\bibitem{kirby2010}
B.~J. Kirby.
\newblock {\em Micro-and nanoscale fluid mechanics: transport in microfluidic
  devices}.
\newblock Cambridge university press, 2010.

\bibitem{Masliyah2006}
J.~H. Masliyah and S.~Bhattacharjee.
\newblock {\em Electrokinetic and Colloid Transport Phenomena}.
\newblock Wiley-Interscience, 2006.

\bibitem{Biesheuvel2010}
P.~M. Biesheuvel and M.~Z. Bazant.
\newblock Nonlinear dynamics of capacitive charging and desalination by porous
  electrodes.
\newblock {\em Physical Review E}, 81(3):031502, 2010.

\bibitem{Whitesides2006}
G.~M. Whitesides.
\newblock The origins and the future of microfluidics.
\newblock {\em Nature}, 442:368--373, 2006.

\bibitem{chen2007finite}
L.~Chen, M.~Holst, and J.~Xu.
\newblock The finite element approximation of the nonlinear
  {P}oisson--{B}oltzmann equation.
\newblock {\em SIAM J. Numer. Anal.}, 45(6):2298--2320, 2007.

\bibitem{kraus2020reliable}
J.~Kraus, S.~Nakov, and S.~Repin.
\newblock Reliable numerical solution of a class of nonlinear elliptic problems
  generated by the {P}oisson--{B}oltzmann equation.
\newblock {\em Comput. Methods Appl. Math.}, 20(2):293--319, 2020.

\bibitem{holst2012adaptive}
M.~Holst, J.~Mccammon, Z.~Yu, Y.~Zhou, and Y.~Zhu.
\newblock Adaptive finite element modeling techniques for the
  {P}oisson-{B}oltzmann equation.
\newblock {\em Commun. Comput. Phys.}, 11(1):179--214, 2012.

\bibitem{Huangvem}
H.~Linghan, S.~Shi, and Y.~Ying.
\newblock Error analysis of virtual element method for the
  {P}oisson–{B}oltzmann equation.
\newblock {\em J. Numer. Math.}, 33(2):187--210, 2025.

\bibitem{vem1}
{L. Beir\~ao da~Veiga}, F.~Brezzi, A.~Cangiani, G.~Manzini, {L.~D.~Marini}, and
  A.~Russo.
\newblock Basic principles of {V}irtual {E}lement {M}ethods.
\newblock {\em Math. Models Methods Appl. Sci.}, 23:199--214, 2013.

\bibitem{vem2}
{L. Beir\~ao da~Veiga}, F.~Brezzi, and {L.~D.~Marini}.
\newblock Virtual elements for linear elasticity problems.
\newblock {\em SIAM J. Numer. Anal.}, 51:794--812, 2013.

\bibitem{mvem19}
A.~Gain, C.~Talischi, and G.~Paulino.
\newblock On the virtual element method for three-dimensional linear elasticity
  problems on arbitrary polyhedral meshes.
\newblock {\em Comput. Methods Appl. Mech. Eng.}, 282:132--160, 2014.

\bibitem{adak2024nonconforming}
D.~Adak, G.~Manzini, and S.~Natarajan.
\newblock Nonconforming virtual element methods for velocity-pressure {S}tokes
  eigenvalue problem.
\newblock {\em Appl. Numer. Math.}, 202:42--66, 2024.

\bibitem{mvem16}
{L. Beir\~ao da~Veiga}, F.~Dassi, and G.~Vacca.
\newblock Vorticity-stabilized virtual elements for the {O}seen equation.
\newblock {\em Math. Models Methods Appl. Sci.}, 31(14):3009--3052, 2021.

\bibitem{mvem9}
{L. Beir\~ao da~Veiga}, C.~Lovadina, and G.~Vacca.
\newblock Virtual elements for the {N}avier-{S}tokes problem on polygonal
  meshes.
\newblock {\em SIAM J. Numer. Anal.}, 56(3):1210--1242, 2018.

\bibitem{mvem11}
{L. Beir\~ao da~Veiga}, C.~Lovadina, and G.~Vacca.
\newblock Divergence free virtual elements for the {S}tokes problem on
  polygonal meshes.
\newblock {\em ESAIM: Math. Model. Numer. Anal.}, 51(2):509--535, 2017.

\bibitem{mvem12}
G.~Vacca.
\newblock An ${H}^1$-conforming virtual element for {D}arcy and {B}rinkman
  equations.
\newblock {\em Math. Models Methods Appl. Sci.}, 28(01):159--194, 2018.

\bibitem{mvem3}
{P.~F.~Antonietti}, G.~Vacca, and M.~Verani.
\newblock Virtual element method for the {N}avier-{S}tokes equation coupled
  with the heat equation.
\newblock {\em IMA J. Numer. Anal.}, pages 1--34, 11 2022.

\bibitem{adak2021virtual}
D.~Adak, D.~Mora, S.~Natarajan, and A.~Silgado.
\newblock A virtual element discretization for the time dependent
  {N}avier--{S}tokes equations in stream-function formulation.
\newblock {\em ESAIM: Math. Model. Numer. Anal.}, 55(5):2535--2566, 2021.

\bibitem{mishra2025supg}
S.~Mishra, E.~Natarajan, and S.~Natarajan.
\newblock A {SUPG}-stabilized virtual element method for the {N}avier--{S}tokes
  equation: approximations of branches of non-singular solutions.
\newblock {\em Adv. Comput. Math.}, 51(4):34, 2025.

\bibitem{da2022virtual}
{L. Beir\~ao da~Veiga}, F.~Dassi, G.~Manzini, and L.~Mascotto.
\newblock Virtual elements for {M}axwell's equations.
\newblock {\em Computers \& Mathematics with Applications}, 116:82--99, 2022.

\bibitem{mvem7}
F.~Brezzi and {L.~D.~Marini}.
\newblock Virtual element methods for plate bending problems.
\newblock {\em Comput. Methods Appl. Mech. Eng.}, 253:455--462, 2013.

\bibitem{aldakheel2018phase}
F.~Aldakheel, B.~Hudobivnik, A.~Hussein, and P.~Wriggers.
\newblock Phase-field modeling of brittle fracture using an efficient virtual
  element scheme.
\newblock {\em Comput. Methods Appl. Mech. Eng.}, 341:443--466, 2018.

\bibitem{benedetto2018virtual}
M.~Benedetto, A.~Caggiano, and G.~Etse.
\newblock Virtual elements and zero thickness interface-based approach for
  fracture analysis of heterogeneous materials.
\newblock {\em Comput. Methods Appl. Mech. Eng.}, 338:41--67, 2018.

\bibitem{antonietti2017virtual}
{P.~F.~Antonietti}, M.~Bruggi, S.~Scacchi, and M.~Verani.
\newblock On the virtual element method for topology optimization on polygonal
  meshes: {A} numerical study.
\newblock {\em Comput. Math. Appl.}, 74(5):1091--1109, 2017.

\bibitem{chi2020virtual}
H.~Chi, A.~Pereira, I.~Menezes, and G.~Paulino.
\newblock Virtual element method ({VEM})-based topology optimization: an
  integrated framework.
\newblock {\em Struct. Multidiscip. Optim.}, 62(3):1089--1114, 2020.

\bibitem{vem42}
{L. Beir\~ao da~Veiga}, F.~Brezzi, F.~Dassi, {L.~D.~Marini}, and A.~Russo.
\newblock Virtual element approximation of 2d magnetostatic problems.
\newblock {\em Comput. Meth. Appl. Mech. Engrg.}, 327:173--195, 2017.

\bibitem{vem028}
J.~Guo and M.~Feng.
\newblock A new projection-based stabilized virtual element method for the
  {S}tokes problem.
\newblock {\em J. Sci. Comput.}, 85(1):16, 2020.

\bibitem{mvem14}
D.~Irisarri and G.~Hauke.
\newblock Stabilized virtual element methods for the unsteady incompressible
  {N}avier-{S}tokes equations.
\newblock {\em Calcolo}, 56(4):38, 2019.

\bibitem{mvem15}
Y.~Li, M.~Feng, and Y.~Luo.
\newblock A new local projection stabilization virtual element method for the
  {O}seen problem on polygonal meshes.
\newblock {\em Adv. Comput. Math.}, 48(3):30, 2022.

\bibitem{mishra2024unified}
S.~Mishra and E.~Natarajan.
\newblock A unified local projection-based stabilized virtual element method
  for the coupled {S}tokes--{D}arcy problem.
\newblock {\em Adv. Comput. Math.}, 50(6):106, 2024.

\bibitem{mishra2025equal}
S.~Mishra and E.~Natarajan.
\newblock An equal-order virtual element framework for the coupled
  {S}tokes-{T}emperature equation with nonlinear viscosity.
\newblock {\em J. Sci. Comput.}, 104(2):73, 2025.

\bibitem{bookgirault}
V.~Girault and P.~Raviart.
\newblock {\em Finite Element Methods for Navier-Stokes Equations. Theory and
  Algorithm}.
\newblock Springer-Verlag, Berlin Heidelberg NewYork, 1986.

\bibitem{vem22}
{L. Beir\~ao da~Veiga}, F.~Brezzi, {L.~D.~Marini}, and A.~Russo.
\newblock Virtual element method for general second-order elliptic problems on
  polygonal meshes.
\newblock {\em Math. Models Methods Appl. Sci.}, 26(04):729--750, 2016.

\bibitem{vem25}
B.~Ahmad, A.~Alsaedi, F.~Brezzi, {L.~D.~Marini}, and A.~Russo.
\newblock Equivalent projectors for virtual element methods.
\newblock {\em Comput. Math. Appl.}, 66(3):376--391, 2013.

\bibitem{scott}
S.~Brenner and L.~Scott.
\newblock {\em The Mathematical Theory of Finite Element Methods}.
\newblock 3rd ed., Springer, New York, 2008.

\bibitem{vem45}
A.~Cangiani, E.~Georgoulis, T.~Pryer, and O.~Sutton.
\newblock A posteriori error estimates for the virtual element method.
\newblock {\em Numer. Math.}, 137:857--892, 2018.

\bibitem{vem28}
L.~Chen and J.~Huang.
\newblock Some error analysis on virtual element methods.
\newblock {\em Calcolo}, 55, 2017.

\bibitem{vem28m}
S.~Mishra and E.~Natarajan.
\newblock Local projection stabilization virtual element method for the
  convection-diffusion equation with nonlinear reaction term.
\newblock {\em Comput. Math. Appl.}, 152:181--198, 2023.

\bibitem{mitscha2017adaptive}
G.~Mitscha-Baude, A.~Buttinger-Kreuzhuber, G.~Tulzer, and C.~Heitzinger.
\newblock Adaptive and iterative methods for simulations of nanopores with the
  {PNP}--{S}tokes equations.
\newblock {\em J. Comput. Phys.}, 338:452--476, 2017.

\end{thebibliography}

\end{document}